\documentclass[10pt]{article}

\usepackage[margin=1in]{geometry}

\usepackage[numbers,sort&compress]{natbib}

\usepackage{authblk}

\usepackage{amsmath, amssymb, amsfonts, mathtools, amsthm}
\mathtoolsset{showonlyrefs=true}
\usepackage{hyperref}
\usepackage{cleveref}
\usepackage{subcaption}

\usepackage{algorithm}

\usepackage{algpseudocode}

\usepackage{graphicx}
\usepackage{booktabs}
\usepackage{xcolor}

\usepackage{multirow}
\usepackage{pifont}

\usepackage{aliascnt}

\newtheorem{theorem}{Theorem}[section]

\newaliascnt{lemma}{theorem}

\aliascntresetthe{lemma}

\newaliascnt{definition}{theorem}

\aliascntresetthe{definition}

\newaliascnt{assumption}{theorem}
\newtheorem{assumption}[assumption]{Assumption}
\aliascntresetthe{assumption}

\newaliascnt{proposition}{theorem}
\newtheorem{proposition}[proposition]{Proposition}
\aliascntresetthe{proposition}

\newaliascnt{example}{theorem}

\aliascntresetthe{example}

\newaliascnt{corollary}{theorem}

\aliascntresetthe{corollary}

\crefname{corollary}{corollary}{corollaries}
\Crefname{corollary}{Corollary}{Corollaries}

\crefname{theorem}{theorem}{theorems}
\Crefname{theorem}{Theorem}{Theorems}
\crefname{lemma}{lemma}{lemmas}
\Crefname{lemma}{Lemma}{Lemmas}
\crefname{definition}{definition}{definitions}
\Crefname{definition}{Definition}{Definitions}
\crefname{assumption}{assumption}{assumptions}
\Crefname{assumption}{Assumption}{Assumptions}
\crefname{proposition}{proposition}{propositions}
\Crefname{proposition}{Proposition}{Propositions}
\crefname{example}{example}{examples}
\Crefname{example}{Example}{Examples}



\newcommand{\tr}{\operatorname{tr}}
\newcommand{\diag}{\operatorname{diag}}

\newcommand{\E}{\mathbb{E}}
\newcommand{\F}{\mathcal{F}}
\newcommand{\R}{\mathbb{R}}
\newcommand{\ip}[2]{\left\langle #1,#2\right\rangle}
\newcommand{\norm}[1]{\left\|#1\right\|}

\newcommand{\Lmat}{\mathbb{L}} 

\begin{document}

\title{Randomized Subspace Nesterov Accelerated Gradient}

\author[1,2]{Gaku Omiya}
\author[2]{Pierre-Louis Poirion}
\author[1,2]{Akiko Takeda}

\affil[1]{Department of Mathematical Informatics, The University of Tokyo, Tokyo, Japan}
\affil[2]{Center for Advanced Intelligence Project, RIKEN, Tokyo, Japan}

\date{}

\maketitle

\begin{abstract}
Randomized-subspace methods reduce the cost of first-order optimization by
using only low-dimensional projected-gradient information, a feature that is
attractive in forward-mode automatic differentiation and communication-limited
settings. While Nesterov acceleration is well understood for full-gradient and
coordinate-based methods, obtaining accelerated methods for general subspace
sketches that use only projected-gradient information and can improve over
full-dimensional Nesterov acceleration in oracle complexity is technically
nontrivial.

We develop randomized-subspace Nesterov accelerated gradient methods for smooth
convex and smooth strongly convex optimization under matrix smoothness and
generic sketch moment assumptions. The key technical ingredient is a
three-sequence formulation tailored to matrix smoothness, which recovers the
corresponding classical Nesterov methods in the full-dimensional case. The
resulting theory establishes accelerated oracle-complexity guarantees and makes
explicit how matrix smoothness and the sketch distribution enter the complexity.
It also provides a unified basis for comparing sketch families and identifying
when randomized-subspace acceleration improves over full-dimensional Nesterov
acceleration in oracle complexity.
\end{abstract}


\section{Introduction}
\label{sec:introduction}

We consider the unconstrained optimization problem:
\begin{equation}
    \min_{x \in \mathbb{R}^d} f(x),
    \label{eq:main-problem}
\end{equation}
where \(f:\mathbb{R}^d \to \mathbb{R}\) is differentiable, and focus on both the smooth convex and smooth strongly convex settings.
Whenever a minimizer exists, we denote by \(x^\star\) an optimal solution and set
\(f^\star \coloneqq f(x^\star)\).
First-order methods are the workhorse of large-scale optimization due to their scalability and low per-iteration cost.
Among them, acceleration techniques---most notably Nesterov's accelerated gradient (NAG) method \cite{nesterov1983,nesterov2004intro}---play a fundamental role.
They achieve optimal convergence rates, improving from \(\mathcal{O}(1/N)\) to \(\mathcal{O}(1/N^2)\) in the convex setting, and from \(L/\mu\) to \(\sqrt{L/\mu}\) dependence in the strongly convex case.
As a result, acceleration has become an indispensable component in modern optimization algorithms.

\paragraph{High-dimensional optimization and randomized subspace methods.}
Many machine learning applications give rise to large-scale, high-dimensional optimization problems.
To address this challenge, randomized coordinate and block-coordinate descent methods have been extensively studied
\cite{nesterov2012,richtarik2014,luxiao2015,wright2015}.
At each iteration, these methods update only one coordinate or a small subset of coordinates, thereby reducing the per-iteration cost.

A natural generalization is given by \emph{randomized subspace methods}, which update along a randomly chosen low-dimensional subspace.
A representative update, introduced by \cite{kozak2021}, is
\begin{equation}
x_{k+1} = x_k - \alpha_k P_k P_k^\top \nabla f(x_k),
\end{equation}
where \(P_k \in \mathbb{R}^{d \times r}\) is a random sketch matrix with \(r \le d\); here, $r$ corresponds to the dimension of the subspace.
This framework includes randomized coordinate and block-coordinate updates as
special cases, and recovers standard gradient descent when \(r=d\) and
\(P_kP_k^\top=I_d\).

Randomized subspace methods are attractive for several reasons.
They provide more flexibility than coordinate-aligned updates, and they are particularly effective under memory and communication constraints.
For instance, reverse-mode AD is efficient for full-gradient computation but
typically requires storing intermediate quantities,
whereas forward-mode AD can compute directional derivatives with lower memory
overhead. 
In a forward-mode implementation, computing the full gradient may
require \(d\) directional derivatives,
whereas subspace methods require only \(r\).
Here, \(P_kP_k^\top\nabla f(x)\) is understood as
\(P_k(P_k^\top\nabla f(x))\), with only \(P_k^\top\nabla f(x)\) queried.
Similarly, in distributed optimization, transmitting an \(r\)-dimensional sketch is significantly cheaper than communicating a full gradient.
These advantages have led to growing interest in randomized subspace gradient methods
\cite{kozak2021,nesterovspokoiny2017,baydin2022gradients,nozawa2023rsg_constrained,omiya2026rsnsgd,flynn2024}.  

\paragraph{Limitations of existing acceleration methods.}
Despite this progress, a fundamental gap remains. While acceleration is essential in first-order optimization, its integration into randomized subspace methods is still poorly understood. For coordinate descent methods, accelerated variants have been developed, including Nesterov's accelerated coordinate descent \cite{nesterov2012} and subsequent refinements such as APCG \cite{lin2014apcg}, APPROX \cite{fercoq2015approx}, ALPHA
\cite{qu2016coordinate}, and non-uniform sampling schemes \cite{allenzhu2016}. However, these methods rely heavily on coordinate-wise structure and do not extend naturally to general subspace directions. 
Acceleration has also been studied for compressed gradient descent in distributed optimization \cite{li2020acgd}.
However, under our oracle model, this does not give the desired oracle-complexity advantage over full-dimensional Nesterov acceleration.
A detailed comparison with compressed-gradient methods is deferred to Appendix~\ref{app:related_work}.

  A natural attempt is to directly combine Nesterov acceleration with randomized subspace gradients by replacing the full gradient with \(P_k P_k^\top \nabla f\):
\[
x_{k+1}=y_k-\eta P_kP_k^\top \nabla f(y_k),
\qquad
y_{k+1}=x_{k+1}+\beta_k(x_{k+1}-x_k).
\]
 However, accelerated guarantees for this direct two-sequence scheme are not obtained by a straightforward adaptation of the classical analysis.
The classical two-sequence Nesterov analysis relies on delicate estimate-sequence arguments,
which do not carry over directly to randomized subspace updates.
In particular, the direct argument closes only under restrictive near-full-dimensional conditions;
a precise discussion is deferred to Appendix~\ref{app:two-sequence-obstruction}.

 \begin{table}[t]
\centering
\caption{Comparison of oracle complexity under a directional-derivative oracle model, 
  where each directional-derivative evaluation counts as one oracle call. 
  Here \(R_0:=\|x_0-x^\star\|\), \(\Delta_0:=f(x_0)-f(x^\star)\), and the sketch parameters \(\ell\) and \(\omega\) are defined in \Cref{ass:sketch-common}. The RS-GD bounds of \cite{kozak2021} are re-evaluated under the same assumptions as our proposed method; see Appendix~\ref{app:kozak-rates}. After instantiating \(\omega\) and \(\ell\) for standard sketches, the resulting
oracle factors are compared in \Cref{sec:sketch-examples}.}
\label{tab:oracle-comparison-intro}
{
\setlength{\tabcolsep}{4pt}
\renewcommand{\arraystretch}{1.1}
\begin{tabular}{@{}l|cc|cc@{}}
\toprule
Method & Convex & & Strongly convex & \\
\midrule
GD
&
\(\mathcal{O}\!\left(d R_0^2 \frac{L}{\epsilon}\right)\)
& 
&
\(\mathcal{O}\!\left(d \frac{L}{\mu}\log\frac{\Delta_0}{\epsilon}\right)\)
& 
\\
RS-GD \cite{kozak2021}
&
\(\mathcal{O}\!\left(\omega r R_0^2 \frac{L}{\epsilon}\right)\)
& Prop.~\ref{prop:kozak-convex}
&
\(\mathcal{O}\!\left(\ell r \frac{L}{\mu}\log\frac{\Delta_0}{\epsilon}\right)\)
& Prop.~\ref{prop:kozak-strong}
\\
\midrule
NAG \cite{nesterov1983,nesterov2004intro}
&
\(\mathcal{O}\!\left(d R_0 \sqrt{\frac{L}{\epsilon}}\right)\)
& 
&
\(\mathcal{O}\!\left(d \sqrt{\frac{L}{\mu}}\log\frac{\Delta_0}{\epsilon}\right)\)
& 
\\
\textbf{RS-NAG (this work)}
&
\(\mathcal{O}\!\left(\sqrt{\omega\ell r^2}\, R_0 \sqrt{\frac{L}{\epsilon}}\right)\)
& Thm.~\ref{thm:rs-nesterov-convex}
&
\(\mathcal{O}\!\left(\sqrt{\omega\ell r^2}\sqrt{\frac{L}{\mu}}\log\frac{\Delta_0}{\epsilon}\right)\)
& Thm.~\ref{thm:rs-nesterov-strong}
\\
\bottomrule
\end{tabular}
}
\end{table}

\paragraph{Our approach and contributions.}
These observations lead to the following question:
\begin{center}
  \emph{Can one design accelerated methods that use only randomized subspace gradients \\ while achieving improved oracle complexity?}
\end{center}
In this paper, we answer this question affirmatively.
We propose randomized-subspace Nesterov accelerated gradient (RS-NAG) methods,
to our knowledge the first accelerated framework for general randomized subspace gradient
methods that can achieve favorable oracle complexity compared with both
non-accelerated randomized-subspace methods and standard NAG; see
\Cref{tab:oracle-comparison-intro}.
Our key technical contribution is a novel \emph{three-sequence} formulation tailored to
matrix smoothness, which combines a sketched descent step with an auxiliary estimate sequence
and enables a clean convergence analysis.

\paragraph{Our contributions.}
\begin{itemize}
\item \textbf{Accelerated randomized subspace methods.}
  We propose Nesterov-type randomized subspace methods for both convex and strongly convex optimization,
  recovering the corresponding classical Nesterov methods in the full-dimensional case.

\item \textbf{Oracle complexity under matrix smoothness.}
  We prove convergence and oracle-complexity bounds that capture the interaction between
  matrix smoothness and the sketch distribution.

\item \textbf{Comparison with full-dimensional acceleration.}
  Our bounds identify when randomized subspace acceleration can outperform full-dimensional
  Nesterov acceleration in oracle complexity.

\item \textbf{Unified comparison of sketching strategies.}
  We analyze Haar, coordinate, and Gaussian sketches, revealing their relative convergence bounds
  and identifying optimal sketch dimensions in terms of oracle complexity.
\end{itemize}

\paragraph{Additional probability guarantees.}
Beyond the expectation bounds stated in \Cref{thm:rs-nesterov-convex,thm:rs-nesterov-strong}, Appendix~\ref{app:hp-as}
also provides uniform-in-time high-probability bounds and almost-sure eventual
rates with only mild losses.
These results show that the accelerated behavior predicted by the expectation
bounds is not merely an average-over-runs phenomenon, but persists with high
probability uniformly over time and eventually along almost every run.

\paragraph{Notation: }
\label{sec:notation}
Unless stated otherwise, $\|\cdot\|$ denotes the Euclidean norm for vectors
  and the corresponding operator norm for matrices induced by the Euclidean norm. Let $I_d$ denote the 
$d \times d$ identity matrix.
For a matrix \(A\in\mathbb{R}^{d\times d}\), \(\diag(A)\in\mathbb{R}^{d\times d}\) denotes the diagonal matrix whose diagonal entries coincide with those of \(A\).
For symmetric matrices \(A,B\in\mathbb{R}^{d\times d}\), we write
\(A\preceq B\) if \(B-A\) is positive semidefinite.

Throughout the paper, we assume that the ambient dimension satisfies \(d\ge 2\). For sketch-based methods, we assume \(1\le r\le d\).

\section{Preliminaries}
\label{sec:preliminaries}

\subsection{Oracle complexity}
\label{sec:oracle-complexity}
We measure oracle complexity in terms of directional-derivative queries.
Under the oracle model considered here, which is compatible with forward-mode AD, a \emph{full-dimensional gradient evaluation} costs \(d\) directional-derivative queries.
In contrast, a \emph{projected gradient evaluation}, which computes a projection of the gradient onto a lower-dimensional subspace, e.g., \(P^\top \nabla f(x)\) for a matrix \(P\in\mathbb{R}^{d\times r}\) with \(r\le d\), costs only \(r\) such queries.
Thus, in this oracle model, the cost scales with \(r\) rather than \(d\).
This oracle measure is also relevant in distributed settings with communication bottlenecks, where transmitting an \(r\)-dimensional sketch requires sending \(r\) real numbers, as opposed to \(d\) for a full gradient, 
and therefore also reduces the communication volume from \(d\) to \(r\) scalars; see
\Cref{app:distributed-comparison} for details.
Accordingly, our complexity results are intended for forward-mode and/or communication-bottlenecked regimes, rather than as a claim about universal wall-clock speedups across all implementations.

\subsection{Problem setting and sketch assumptions}
\begin{assumption}[Matrix smoothness]
\label{ass:matrix-smooth-common}
Let \(f:\R^d\to\R\) be differentiable. Assume that there exists a fixed matrix
\(\Lmat\in\R^{d\times d}\) with
\[
\Lmat \succeq 0,
\qquad
\Lmat \neq 0,
\]
such that, for all \(x,y\in\R^d\),
\begin{equation}
f(y)\le f(x)+\ip{\nabla f(x)}{y-x}+\frac12 (y-x)^\top \Lmat (y-x).
\label{eq:matrix-smooth-common}
\end{equation}
Define
\(
L \coloneqq \|\Lmat\|.
\)
\end{assumption}
Since \(\Lmat\preceq L I_d\), \eqref{eq:matrix-smooth-common} implies the
standard scalar \(L\)-smooth descent lemma.
Assumptions of this form have recently been used increasingly in matrix-smooth optimization and compression; see, e.g., \citet{wang2022smoothnessaware,li2023detCGD,maranjyan2025gradskip,hanzely2019,flynn2024}.

\begin{assumption}[Sketch moment conditions]
\label{ass:sketch-common}
Let \(P\in\R^{d\times r}\) be a random matrix.
Assume that for some constants \(\omega>0\) and \(\ell>0\),
\begin{align}
\E[PP^\top] &= I_d, \label{eq:unbiased-common}\\
\E[(PP^\top)^2] &\preceq \omega I_d, \label{eq:second-moment-common}\\
\E[PP^\top \Lmat PP^\top] &\preceq \ell L\, I_d.
\label{eq:Lsmooth-sketch-common}
\end{align}
\end{assumption}
We will show in \Cref{sec:sketch-examples} that this assumption is satisfied by
Haar, coordinate, and Gaussian sketches.

In all algorithms and analyses below, the sketch matrices \(\{P_k\}_{k\ge0}\) are assumed to be i.i.d.\ copies of a random matrix \(P\in\R^{d\times r}\).

\begin{proposition}
\label{prop:ell-omega-redundant}
Under \Cref{ass:matrix-smooth-common,ass:sketch-common}, any constant
\(\omega\) satisfying \eqref{eq:second-moment-common} satisfies
\(\omega\ge d/r\), and any constant \(\ell\) satisfying
\eqref{eq:Lsmooth-sketch-common} satisfies \(\ell\ge1\). Consequently, every
admissible pair \((\omega,\ell)\) satisfies \(\ell\omega\ge1\). Moreover, for
any \(\omega\) satisfying \eqref{eq:second-moment-common},
\eqref{eq:Lsmooth-sketch-common} holds with \(\ell=\omega\). 
\end{proposition}
The proof is deferred to \Cref{app:missing-proofs-prem}. By the last statement
of \Cref{prop:ell-omega-redundant}, for any admissible \(\omega\),
\eqref{eq:Lsmooth-sketch-common} also holds with \(\ell=\omega\). Therefore,
we assume without loss of generality throughout the rest of the paper that
\(\ell\le\omega\).
The explicit constants derived in \Cref{sec:sketch-examples} and summarized in
\Cref{tab:sketch_constants} satisfy this convention.

\section{RS-NAG for convex problems (RS-NAG-C)}
\label{sec:rs-NAG-c}
We first consider the convex, not necessarily strongly convex, setting.
\begin{assumption}[Convexity]
\label{ass:func-convex}
Let \(f:\R^d\to\R\) be differentiable and convex.
Assume that \(f\) admits a minimizer.
\end{assumption}

Under \Cref{ass:func-convex}, we propose RS-NAG for convex problems
(RS-NAG-C), a randomized-subspace variant of standard NAG for smooth convex
optimization \cite{nesterov1983,nesterov2004intro}. The method is given in
\Cref{alg:rs-NAG-c}. Compared with standard NAG, RS-NAG-C uses only the sketched gradient
\(P_kP_k^\top\nabla f(y_k)\) in place of the full gradient \(\nabla f(y_k)\).

\begin{algorithm}[t]
\caption{RS-NAG-C}
\label{alg:rs-NAG-c}
\begin{algorithmic}[1]
\Require \(x_0\in\mathbb{R}^d\), constants \(L,\omega,\ell>0\), and an i.i.d.\ sketch sequence \(\{P_k\}_{k\ge0}\)
\State \(m \gets 1/(2L\ell)\)
\State \(A_0 \gets 0\), \(z_0 \gets x_0\)
\For{\(k=0,1,2,\dots\)}
    \State \(a_{k+1} \gets \dfrac{m+\sqrt{m^2+2\omega m A_k}}{\omega}\)
    \State \(A_{k+1} \gets A_k+a_{k+1}\)
    \State \(y_k \gets \dfrac{A_k}{A_{k+1}}x_k+\dfrac{a_{k+1}}{A_{k+1}}z_k\)
    \State \(x_{k+1} \gets y_k-\dfrac{1}{L\ell}P_kP_k^\top \nabla f(y_k)\)
    \State \(z_{k+1} \gets z_k-a_{k+1}P_kP_k^\top \nabla f(y_k)\)
\EndFor
\end{algorithmic}
\end{algorithm}

The corresponding two-sequence NAG recursion is recalled in
\Cref{app:missing-proofs-convex},
\eqref{eq:NAG-conv-x-final}--\eqref{eq:NAG-conv-y-final}. We prove below that
RS-NAG-C reduces to this recursion when \(r=d\) and \(P_kP_k^\top=I_d\).
Proofs for this section are deferred to \Cref{app:missing-proofs-convex}.

\begin{proposition}[Convex case: reduction to standard Nesterov]
\label{prop:convex-reduction-nesterov}
Suppose \Cref{ass:matrix-smooth-common,ass:sketch-common} and
\Cref{ass:func-convex} hold. Consider the full-sketch case \(r=d\), and assume
that
\[
P_kP_k^\top=I_d
\qquad\text{for all }k\ge0,
\]
which is the case, for example, for the Haar and block coordinate sketches.
Then \Cref{ass:sketch-common} is satisfied with
\(
\omega=1, \ell=1.
\)
With this choice, the \((x_k,y_k)\)-sequence generated by
\Cref{alg:rs-NAG-c} coincides with classical two-sequence
NAG for convex objectives.
\end{proposition}

Therefore, the proposed randomized-subspace methods in the convex setting
can be viewed as generalizations of the standard NAG.

We next state the convergence guarantee for RS-NAG-C.

\begin{theorem}
\label{thm:rs-nesterov-convex}
Suppose 
\Cref{ass:matrix-smooth-common,ass:sketch-common} and~\Cref{ass:func-convex} hold,
and let \(\{x_k,y_k,z_k\}\)  
be generated by
\Cref{alg:rs-NAG-c}.
Then, for all \(N\ge1\),
\begin{equation}
\E[f(x_N)-f^\star]
\le
2L\omega\ell\,\frac{\norm{x_0-x^\star}^2}{N^2}.
\label{eq:convex-final-rate}
\end{equation}
In particular, for \(R_0\coloneqq\norm{x_0-x^\star}\), the iteration
complexity to guarantee \(\E[f(x_N)-f^\star]\le\epsilon\) is
\[
N
=
\mathcal{O}\!\left(
R_0\sqrt{\frac{L\omega\ell}{\epsilon}}
\right),
\]
and since one iteration uses \(r\) oracle calls, the oracle complexity is
\[
\#\mathrm{Oracle}
=
rN
=
\mathcal{O}\!\left(
R_0\sqrt{\frac{L\omega\ell r^2}{\epsilon}}
\right).
\]
\end{theorem}

Compared with the convex RS-GD bound in
\Cref{tab:oracle-comparison-intro}, RS-NAG-C improves the accuracy dependence
of the oracle complexity from \(1/\epsilon\) to \(1/\sqrt{\epsilon}\), matching
the acceleration effect of NAG over GD.

\section{RS-NAG for strongly convex problems (RS-NAG-SC)}
\label{sec:rs-NAG-sc}

We next consider the strongly convex setting.

\begin{assumption}[Strong convexity]
\label{ass:func}
Let \(f:\R^d\to\R\) be differentiable and \(\mu\)-strongly convex with
\(\mu>0\).
\end{assumption}

Under \Cref{ass:func}, we propose RS-NAG for strongly convex problems
(RS-NAG-SC), a randomized-subspace variant of standard NAG for smooth strongly
convex optimization \cite{nesterov1983,nesterov2004intro}. The method is given
in \Cref{alg:rs-NAG-sc}. Compared with standard NAG, RS-NAG-SC uses only the
sketched gradient \(P_kP_k^\top\nabla f(y_k)\) in place of the full gradient
\(\nabla f(y_k)\).

\begin{algorithm}[t]
\caption{RS-NAG-SC}
\label{alg:rs-NAG-sc}
\begin{algorithmic}[1]
\Require \(x_0\in\mathbb{R}^d\), constants \(L,\mu,\omega,\ell>0\), and an i.i.d.\ sketch sequence \(\{P_k\}_{k\ge0}\)
\State \(\theta \gets \sqrt{\mu/(L\omega\ell)}\)
\State \(z_0 \gets x_0\)
\For{\(k=0,1,2,\dots\)}
    \State \(y_k \gets \dfrac{1}{1+\theta}x_k+\dfrac{\theta}{1+\theta}z_k\)
    \State \(x_{k+1} \gets y_k-\dfrac{1}{L\ell}P_kP_k^\top \nabla f(y_k)\)
    \State \(z_{k+1} \gets (1-\theta)z_k+\theta y_k-\dfrac{\theta}{\mu}P_kP_k^\top \nabla f(y_k)\)
\EndFor
\end{algorithmic}
\end{algorithm}

The corresponding two-sequence NAG recursion is recalled in
\Cref{app:missing-proofs-strong},
\eqref{eq:NAG-sc-x-final}--\eqref{eq:beta-classical}. We prove below that RS-NAG-SC reduces to this
recursion when \(r=d\) and \(P_kP_k^\top=I_d\). Proofs for this section are
deferred to \Cref{app:missing-proofs-strong}.

\begin{proposition}[Strongly convex case: reduction to standard Nesterov]
\label{prop:strong-reduction-nesterov}
Suppose \Cref{ass:matrix-smooth-common,ass:sketch-common} and
\Cref{ass:func} hold. Consider the full-sketch case \(r=d\), and assume that
\[
P_kP_k^\top=I_d
\qquad\text{for all }k\ge0,
\]
which is the case, for example, for the Haar and block coordinate sketches.
Then \Cref{ass:sketch-common} is satisfied with
\(
\omega=1, \ell=1.
\)
With this choice, the \((x_k,y_k)\)-sequence generated by
\Cref{alg:rs-NAG-sc} coincides with classical two-sequence
Nesterov accelerated gradient method
for strongly convex objectives.
\end{proposition}

Therefore, the proposed randomized-subspace methods in the strongly convex
setting can be viewed as generalizations of the standard NAG.

We next state the convergence guarantee for RS-NAG-SC.

\begin{theorem}
\label{thm:rs-nesterov-strong}
Suppose 
\Cref{ass:matrix-smooth-common,ass:sketch-common} and \Cref{ass:func} hold,
and let \(\{x_k,y_k,z_k\}\) be generated by
\Cref{alg:rs-NAG-sc}. Define
\[
\Delta_0\coloneqq f(x_0)-f^\star.
\]
Then the parameter \(\theta\) in \Cref{alg:rs-NAG-sc} satisfies
\(\theta\in(0,1]\), and for all \(N\ge0\),
\begin{equation}
\E[f(x_N)-f^\star]
\le
2(1-\theta)^N\Delta_0.
\label{eq:main-rate}
\end{equation}
In particular, the iteration complexity to guarantee
\(\E[f(x_N)-f^\star]\le\epsilon\) is
\[
N
=
\mathcal{O}\!\left(
\sqrt{\frac{L\omega\ell}{\mu}}\,
\log\frac{\Delta_0}{\epsilon}
\right),
\]
and since one iteration uses \(r\) oracle calls, the oracle complexity is
\begin{equation}
\#\mathrm{Oracle}
=
rN
=
\mathcal{O}\!\left(
\sqrt{\frac{L\omega\ell r^2}{\mu}}\,
\log\frac{\Delta_0}{\epsilon}
\right).
\label{eq:oracle-lower-omega}
\end{equation}
\end{theorem}

Compared with the strongly convex RS-GD bound in
\Cref{tab:oracle-comparison-intro}, RS-NAG-SC improves the condition-number
dependence of the oracle complexity from \(L/\mu\) to \(\sqrt{L/\mu}\),
matching the acceleration effect of NAG over GD.

\section{Examples of sketches: Haar, coordinate, and Gaussian}
\label{sec:sketch-examples}
We verify \Cref{ass:sketch-common} for three standard sketches: Haar,
Coordinate, and Gaussian sketches. We then compare the resulting
sketch-dependent oracle-complexity factors. Proofs for this section are
deferred to \Cref{app:missing-proofs-sketch}.
Throughout this section, let \(L=\|\Lmat\|\), where \(\Lmat\) is the matrix in
\Cref{ass:matrix-smooth-common}, and define
\[
r_{\mathrm{eff}}\coloneqq\frac{\tr(\Lmat)}{L},
\qquad
\delta_{\mathrm{diag}}\coloneqq\frac{\|\diag(\Lmat)\|}{L}.
\]

The oracle-complexity bounds in
\Cref{thm:rs-nesterov-convex,thm:rs-nesterov-strong} are governed by the
sketch-dependent factor
\(\sqrt{\omega\ell r^2}.
\)
For full-dimensional Nesterov, the corresponding factor is \(d\). Hence
\(\sqrt{\omega\ell r^2}<d\) means an improvement over full-dimensional Nesterov in our oracle bound.

\begin{proposition}
\label{prop:verify-sketch-common}
Let \(d\ge2\), \(1\le r\le d\), \(\Lmat\succeq0\), and
\(L=\|\Lmat\|>0\). 
Then \Cref{ass:sketch-common} holds with the constants in
\Cref{tab:sketch_constants} for the following sketches:
Haar \(P=\sqrt{d/r}\,R\), where \(R\) consists of the first \(r\) columns of a
Haar-distributed orthogonal matrix; Coordinate
\(P=\sqrt{d/r}\,S\), where \(S\) consists of \(r\) uniformly sampled distinct
columns of \(I_d\); and Gaussian \(P_{ij}\sim\mathcal N(0,1/r)\) i.i.d.
Moreover,
\[
1\le r_{\mathrm{eff}}\le d,
\qquad
\frac1d\le \delta_{\mathrm{diag}}\le 1,
\qquad
\delta_{\mathrm{diag}}\ge \frac{r_{\mathrm{eff}}}{d}.
\]
\end{proposition}

\begin{table}[t]
\centering
\caption{Sketch-dependent constants for Haar, coordinate, and Gaussian sketches.
Here \(\beta=d(d-r)/((d+2)(d-1))\).}
\label{tab:sketch_constants}
\begin{tabular}{lccc}
\toprule
Sketch & \(\omega\) & \(\ell\) & \(\sqrt{\omega\ell r^2}\) \\
\midrule
Haar &
\(\dfrac{d}{r}\) &
\(\dfrac{d}{r}\left(1-\beta+\beta\dfrac{r_{\rm eff}}{d}\right)\) &
\(d\sqrt{1-\beta+\beta\dfrac{r_{\rm eff}}{d}}\) \\
Coordinate &
\(\dfrac{d}{r}\) &
\(\dfrac{d}{r}\left(\dfrac{r-1}{d-1}
+\dfrac{d-r}{d-1}\delta_{\rm diag}\right)\) &
\(d\sqrt{\dfrac{r-1}{d-1}
+\dfrac{d-r}{d-1}\delta_{\rm diag}}\) \\
Gaussian &
\(\dfrac{d+r+1}{r}\) &
\(\dfrac{r+1+r_{\rm eff}}{r}\) &
\(\sqrt{(d+r+1)(r+1+r_{\rm eff})}\) \\
\bottomrule
\end{tabular}
\end{table}

We next minimize the sketch-dependent oracle factor
\(\sqrt{\omega\ell r^2}\) over the sketch dimension \(r\).
Here \(\omega\) and \(\ell\) are the values in \Cref{tab:sketch_constants}, which
generally depend on \(r\).

\begin{proposition}[Optimal sketch dimension and comparison of sketch constants]
\label{prop:comp-sketch-basic}
For each of the Haar, Coordinate, and Gaussian sketches, with \(\omega\) and
\(\ell\) chosen as in \Cref{tab:sketch_constants}, the factor
\(\sqrt{\omega\ell r^2}\) is minimized over \(r\in\{1,\dots,d\}\) at \(r=1\).
At \(r=1\), the Haar, Coordinate, and Gaussian factors, denoted by
\(Q_{\mathrm H}\), \(Q_{\mathrm C}\), and \(Q_{\mathrm G}\), respectively, are
as follows.

\smallskip
\noindent\normalfont\textbf{Values and ranges.}
\[
\begin{aligned}
Q_{\mathrm H}
&=d\sqrt{\frac{r_{\mathrm{eff}}+2}{d+2}},
&
d\sqrt{\frac{3}{d+2}}
&\le Q_{\mathrm H}\le d,\\
Q_{\mathrm C}
&=d\sqrt{\delta_{\mathrm{diag}}},
&
\sqrt d
&\le Q_{\mathrm C}\le d,\\
Q_{\mathrm G}
&=\sqrt{(d+2)(r_{\mathrm{eff}}+2)},
&
\sqrt{3(d+2)}
&\le Q_{\mathrm G}\le d+2.
\end{aligned}
\]

\smallskip
\noindent\normalfont\textbf{Relations between sketch factors.}
\[
Q_{\mathrm G}=\left(1+\frac{2}{d}\right)Q_{\mathrm H},
\qquad
Q_{\mathrm H}\le \sqrt3\,Q_{\mathrm C}.
\]
Moreover, for \(\Lmat=e_1e_1^\top\),
\[
Q_{\mathrm H}
=
\sqrt{\frac{3}{d+2}}\,Q_{\mathrm C}.
\]
\end{proposition}

\Cref{prop:comp-sketch-basic} shows that the best bound in each sketch family is
attained at \(r=1\). Since the full-dimensional Nesterov factor is \(d\), the
ranges above show that the factor \(Q\) can decrease to about \(\sqrt d\) in
favorable cases. Haar and Coordinate are always no worse than full-dimensional
Nesterov, and
Haar is always better than Gaussian, although the two are nearly identical for
large \(d\). Compared with Coordinate, Haar is never worse by more than
\(\sqrt3\), 
while it can be much better, as shown by the final example.

\section{Numerical Experiments}
\label{sec:experiments}
We evaluate the convex and strongly convex versions of RS-NAG, namely
RS-NAG-C and RS-NAG-SC.
Unless otherwise stated, each curve is the mean over independent random seeds,
and the shaded region denotes mean \(\pm\) one standard deviation across runs.
The seeds determine the Gaussian initialization; for randomized-subspace
methods, they also determine the sampled sketch sequence.

\subsection{Quadratic objectives}
\label{subsec:quadratic}

We first consider four quadratic objectives \(f(x)=\frac12 x^\top \Lmat x\)
on \(\R^d\), with \(d=1000\) and \(f^\star=0\).
These instances are designed to isolate the
effects of the effective rank \(r_{\mathrm{eff}}\) and the diagonal quantity
\(\delta_{\mathrm{diag}}\) appearing in the sketch-dependent constants.
The first two instances are convex but not strongly convex, while the last two
are strongly convex. The diagonal instances have small \(r_{\mathrm{eff}}\) and
large \(\delta_{\mathrm{diag}}\), whereas the dense instances have small
\(\delta_{\mathrm{diag}}\). 
The four matrices \(\Lmat\), together with the corresponding values of
\(L,\mu,r_{\mathrm{eff}}\), and \(\delta_{\mathrm{diag}}\), are specified in
Appendix~\ref{app:quadratic-details}.

We use oracle budget \(10{,}000\).
For randomized-subspace methods, we set \(r=1\) in the main experiments, the
theoretically preferred choice.
Appendix~\ref{app:exp-details-rscan} reports an \(r\)-sweep.
We use \(10\) independent random seeds with
\(x_0\sim\mathcal N(0,I_d)\) for each seed.
For randomized-subspace methods, we consider Haar, Block-coordinate, and
Gaussian sketches. We compare NAG with \emph{RS-NAG-C} in the convex setting
and \emph{RS-NAG-SC} in the strongly convex setting, and plot
\(f(x_k)-f^\star\) vs. oracle calls.

\begin{figure*}[t]
\centering
\begin{minipage}[t]{0.24\textwidth}
    \centering
    \includegraphics[width=\linewidth]{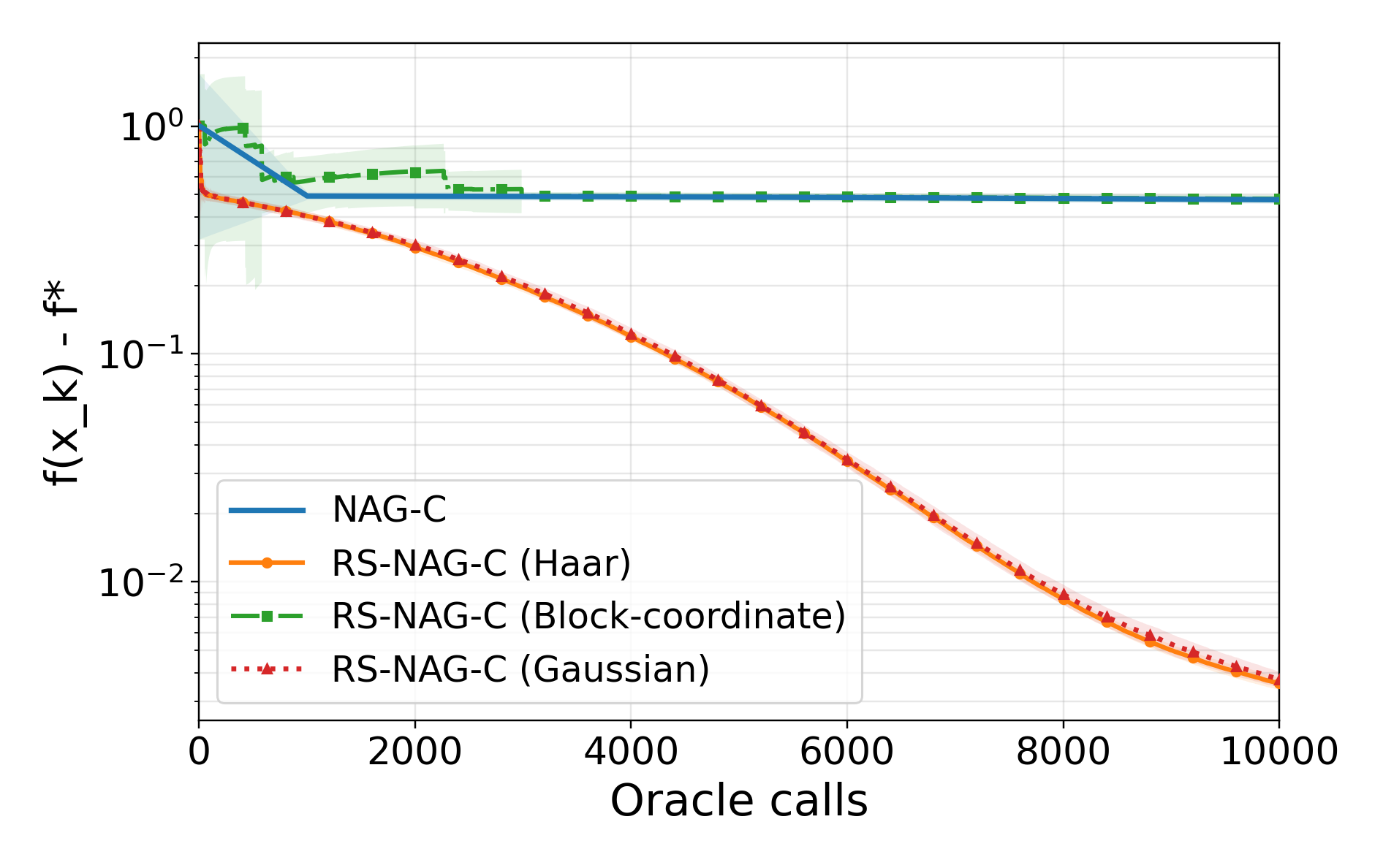}
    \caption*{(a) Convex diagonal}
\end{minipage}
\hfill
\begin{minipage}[t]{0.24\textwidth}
    \centering
    \includegraphics[width=\linewidth]{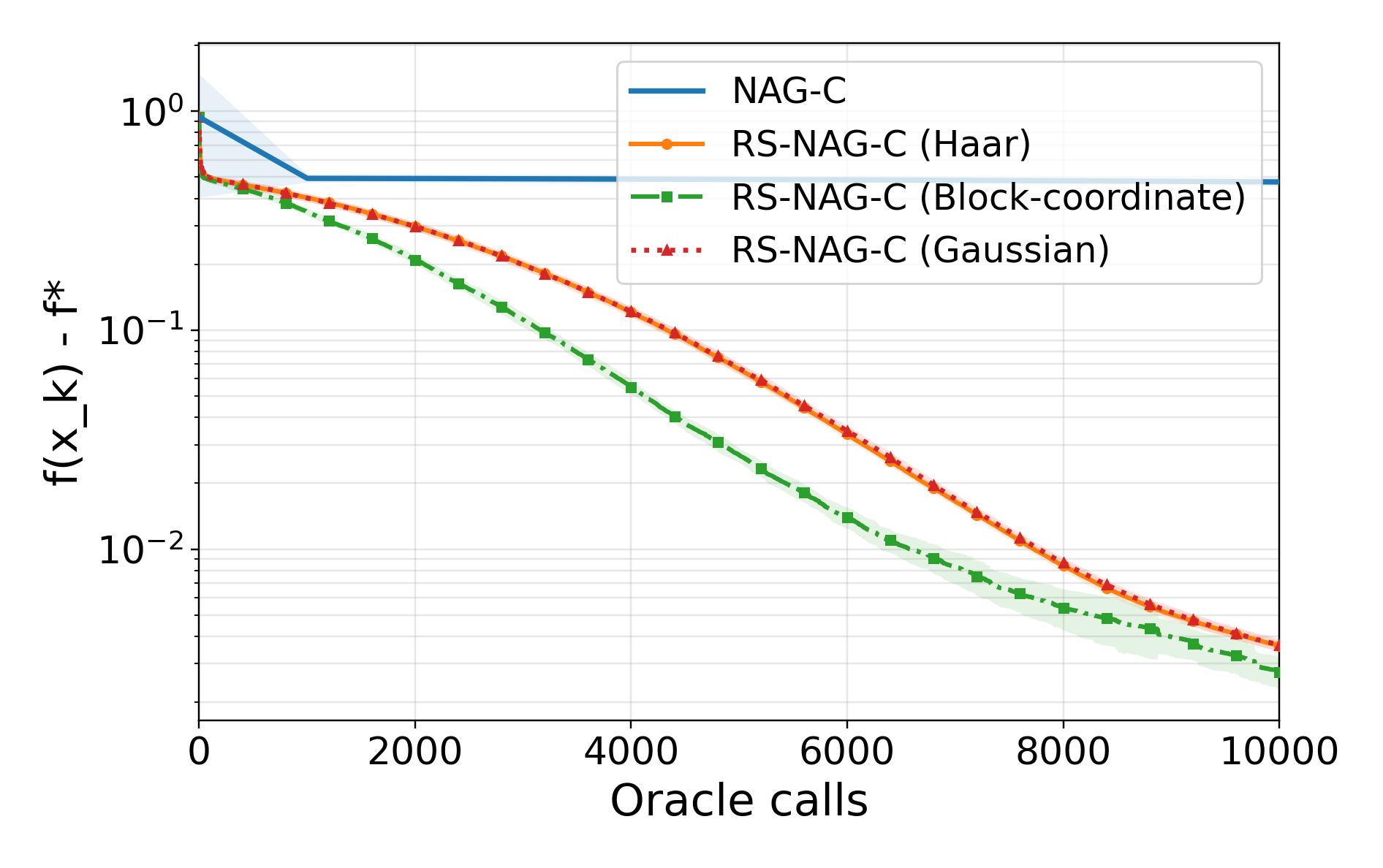}
    \caption*{(b) Convex dense}
\end{minipage}
\hfill
\begin{minipage}[t]{0.24\textwidth}
    \centering
    \includegraphics[width=\linewidth]{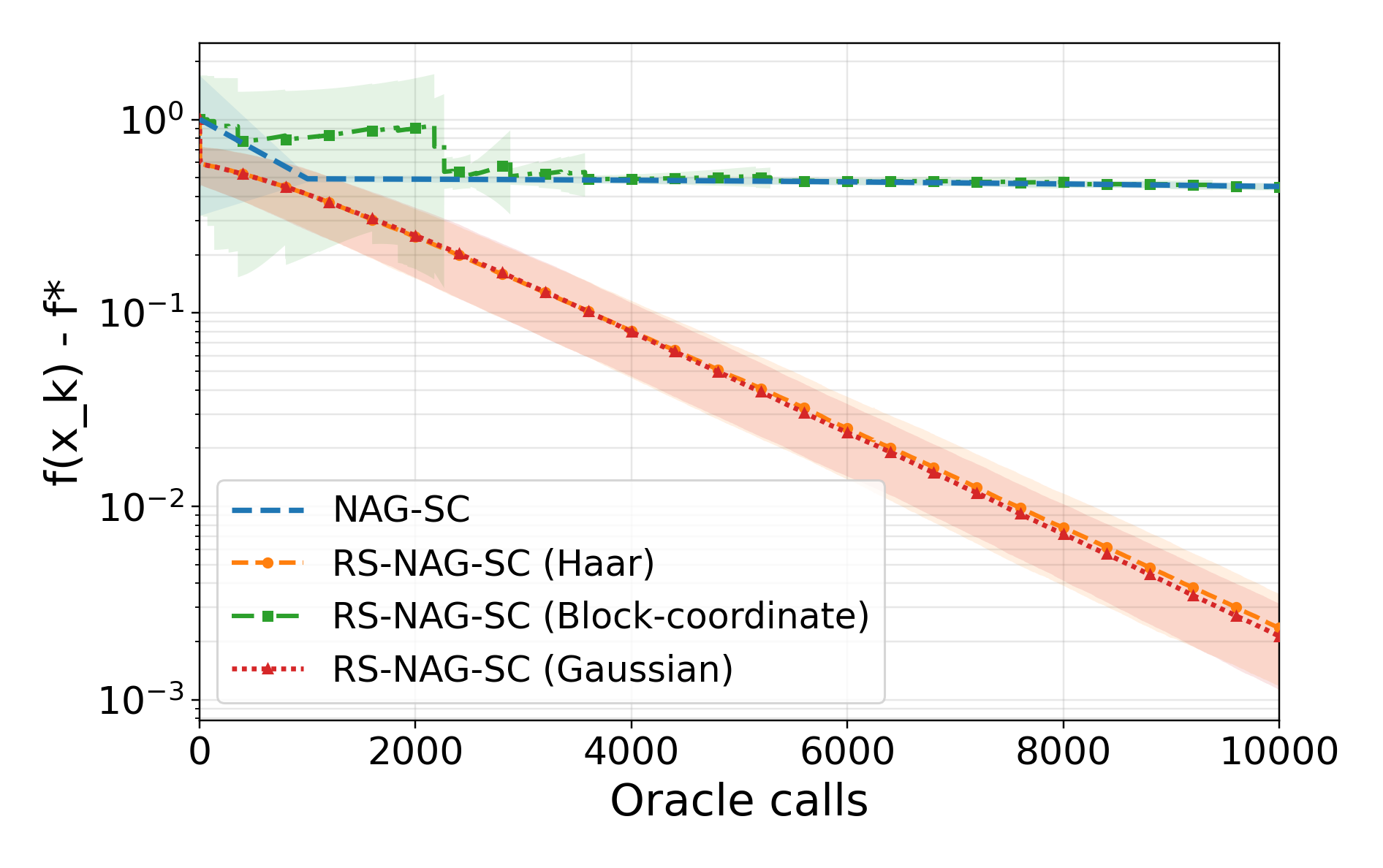}
    \caption*{(c) Strongly convex diagonal}
\end{minipage}
\hfill
\begin{minipage}[t]{0.24\textwidth}
    \centering
    \includegraphics[width=\linewidth]{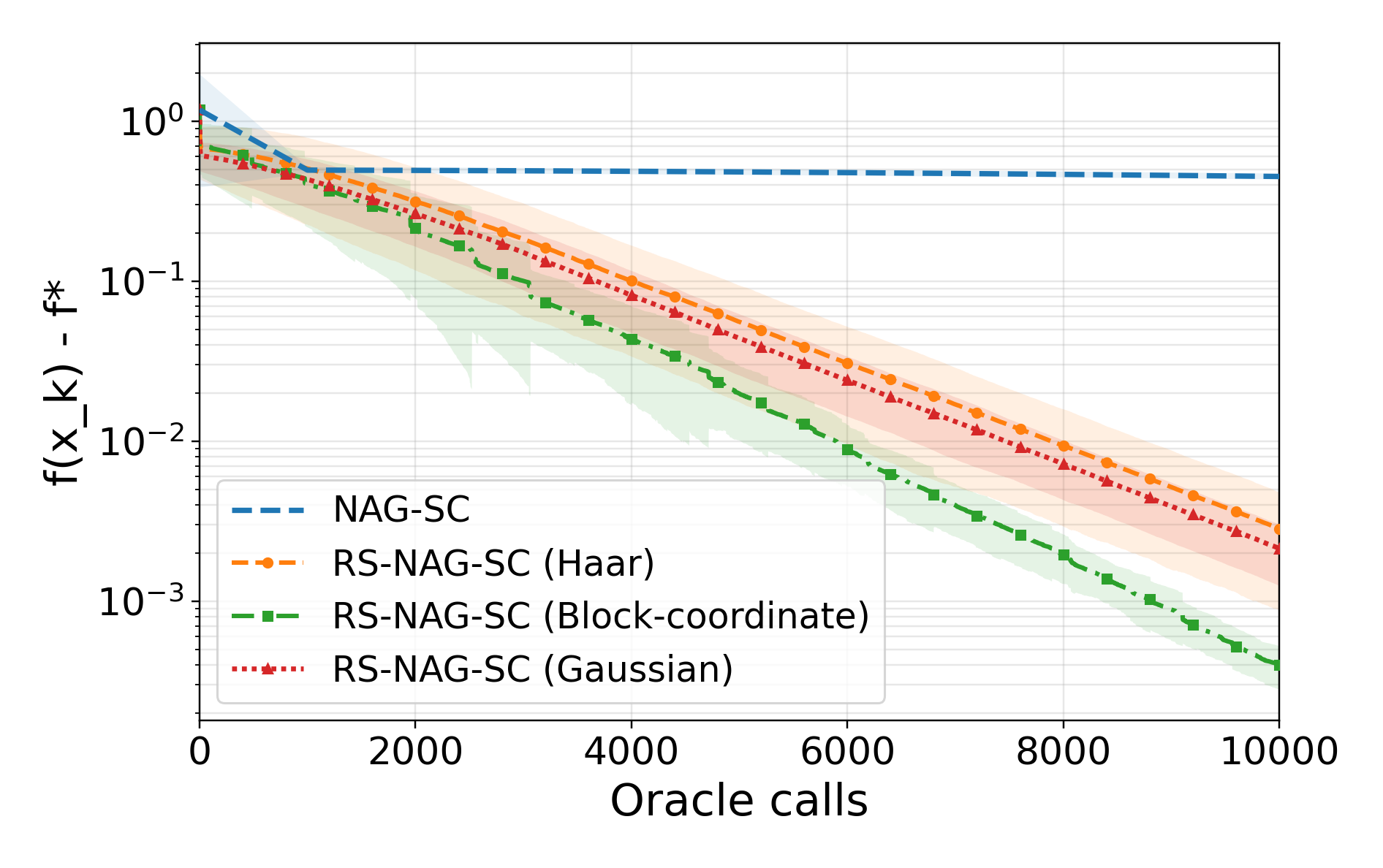}
    \caption*{(d) Strongly convex dense}
\end{minipage}
\caption{Oracle-axis convergence on the four quadratic problems.
The horizontal axis shows the number of oracle calls, and the vertical axis shows the objective gap \(f(x_k)-f^\star\) on a logarithmic scale.}
\label{fig:quadratic-oracle}
\end{figure*}

\paragraph{Discussion.}
The results are consistent with the theoretical predictions in
Proposition~\ref{prop:comp-sketch-basic}. On the diagonal instances, where
\(r_{\mathrm{eff}}\) is small and \(\delta_{\mathrm{diag}}\) is large, Haar and
Gaussian sketches outperform both the Block-coordinate sketch and
full-dimensional Nesterov acceleration. On the dense instances, where
\(\delta_{\mathrm{diag}}\) is small, the Block-coordinate sketch becomes the most
effective, while Haar and Gaussian remain competitive. Thus, the observed
oracle-axis behavior reflects the sketch-dependent quantities
\(Q_{\mathrm H},Q_{\mathrm G},Q_{\mathrm C}\).

\subsection{Logistic regression}
\label{subsec:logreg}
We next evaluate RS-NAG-SC on \(\ell_2\)-regularized logistic regression.
Given binary classification data
\(\{(a_i,y_i)\}_{i=1}^n\), where \(a_i\in\R^d\) and
\(y_i\in\{-1,+1\}\),
we consider
\begin{equation}
f(x)
=
\frac{1}{n}\sum_{i=1}^n \log\!\bigl(1+\exp(-y_i a_i^\top x)\bigr)
+\frac{\mu}{2}\|x\|_2^2,
\qquad \mu>0.
\label{eq:logreg-obj}
\end{equation}
Since \(\mu>0\), the objective is \(\mu\)-strongly convex. For the matrix
smoothness constant, let \(A\in\R^{n\times d}\) be the data matrix whose
\(i\)-th row is \(a_i^\top\), and use
\begin{equation}
\Lmat=\frac{1}{4n}A^\top A+\mu I_d,
\label{eq:logreg-Lmat}
\end{equation}
see Appendix~\ref{app:logreg-smoothness} for the derivation.

We evaluate six real-world binary-classification benchmarks:
colon-cancer~\cite{alon1999colon},
hiva\_agnostic~\cite{guyon2007agnostic},
bioresponse~\cite{hamner2012bioresponse},
gisette~\cite{guyon2005nipsfs},
leukemia~\cite{golub1999leukemia}, and
duke (Duke breast-cancer)~\cite{west2001duke}.
For each dataset, we set \(\mu=1/n\), use \(r=1\), and set
\(L=\|\Lmat\|\), computed numerically as the largest eigenvalue of the matrix
\(\Lmat\) in \eqref{eq:logreg-Lmat}.
We compare GD, NAG-SC,
RS-GD, and RS-NAG-SC with Haar, coordinate, and Gaussian sketches. We initialize
each run from a Gaussian random vector and plot the objective gap
\(f(x_k)-f_{\mathrm{ref}}\) against oracle calls, where \(f_{\mathrm{ref}}\) is
computed by L-BFGS-B \citep{byrd1995limited,zhu1997algorithm}.
The oracle axis can also be read as communication bits under fixed-precision
distributed implementations.\footnote{This follows the communication-bit
accounting in the experiments of \citet{li2020acgd}, where an \(r\)-dimensional
sparse message is counted as \(32r\) bits. The essential comparison is \(d\)
versus \(r\) transmitted scalars per iteration; with \(b\)-bit scalars, these
correspond to \(bd\) and \(br\) bits.}
The results are shown in
\Cref{fig:logreg-real-gap-additional}, and the corresponding
dataset-dependent quantities are summarized in
\Cref{tab:logreg-q-values-additional}. Additional implementation details,
dataset sources, and reference-solver details are provided in
Appendix~\ref{app:experiment-details}.
Appendix~\ref{app:logreg-additional-datasets} 
further reports experiments on six additional real-world datasets.

\begin{figure*}[t]
\centering
\begin{subfigure}[t]{0.32\textwidth}
    \centering
    \includegraphics[width=\linewidth]{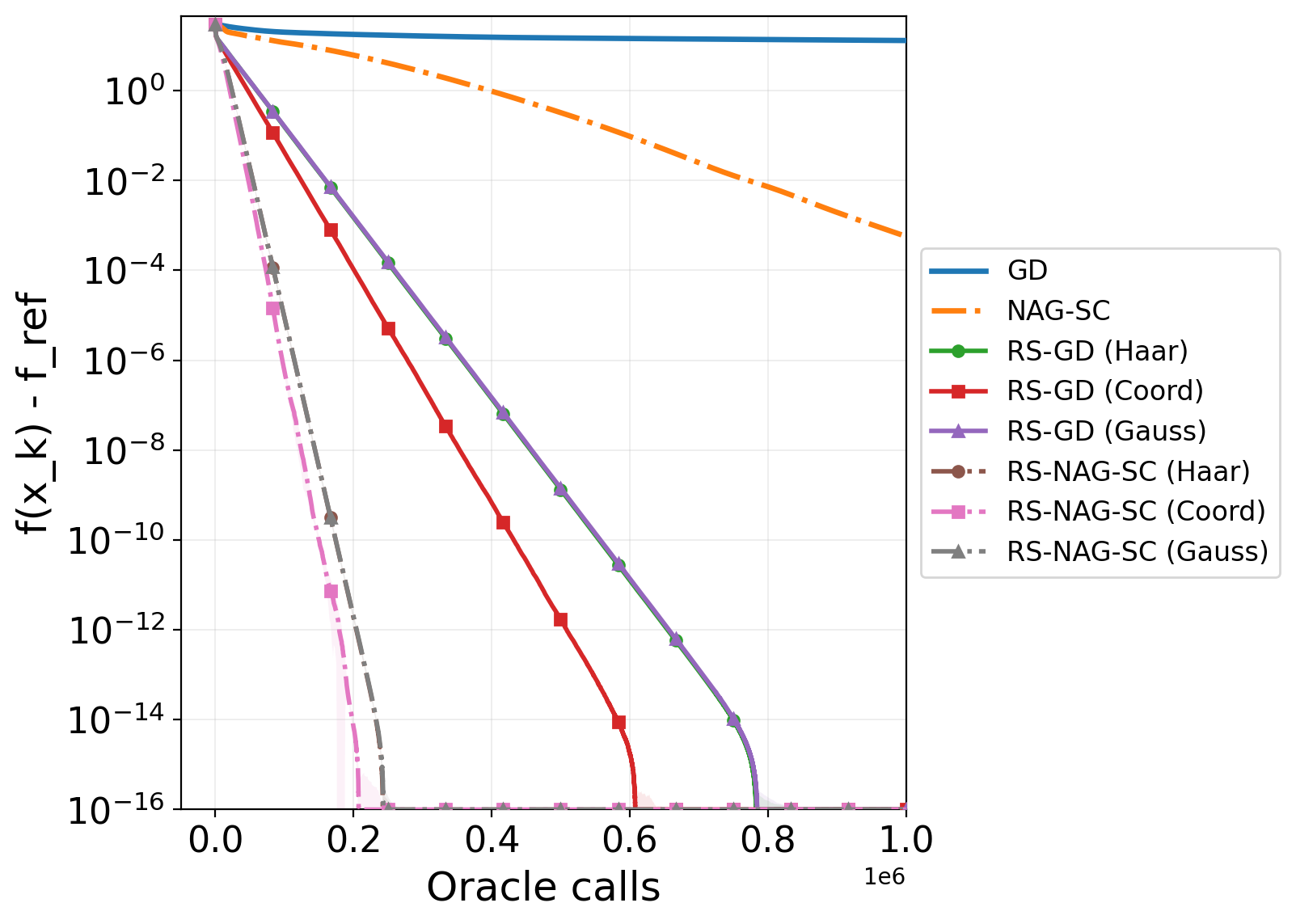}
    \caption{\texttt{colon-cancer}}
\end{subfigure}
\hfill
\begin{subfigure}[t]{0.32\textwidth}
    \centering
    \includegraphics[width=\linewidth]{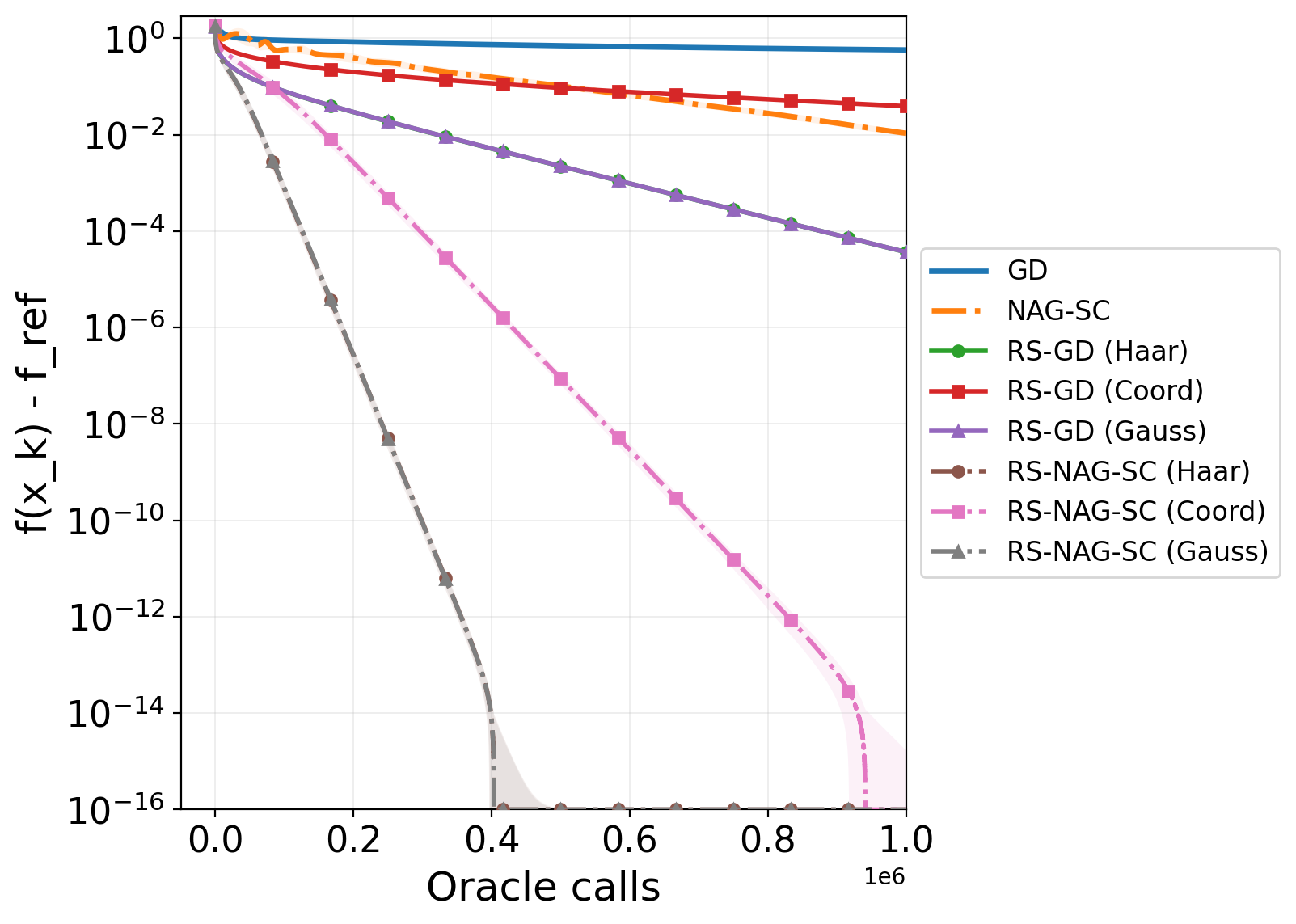}
    \caption{\texttt{hiva\_agnostic}}
\end{subfigure}
\hfill
\begin{subfigure}[t]{0.32\textwidth}
    \centering
    \includegraphics[width=\linewidth]{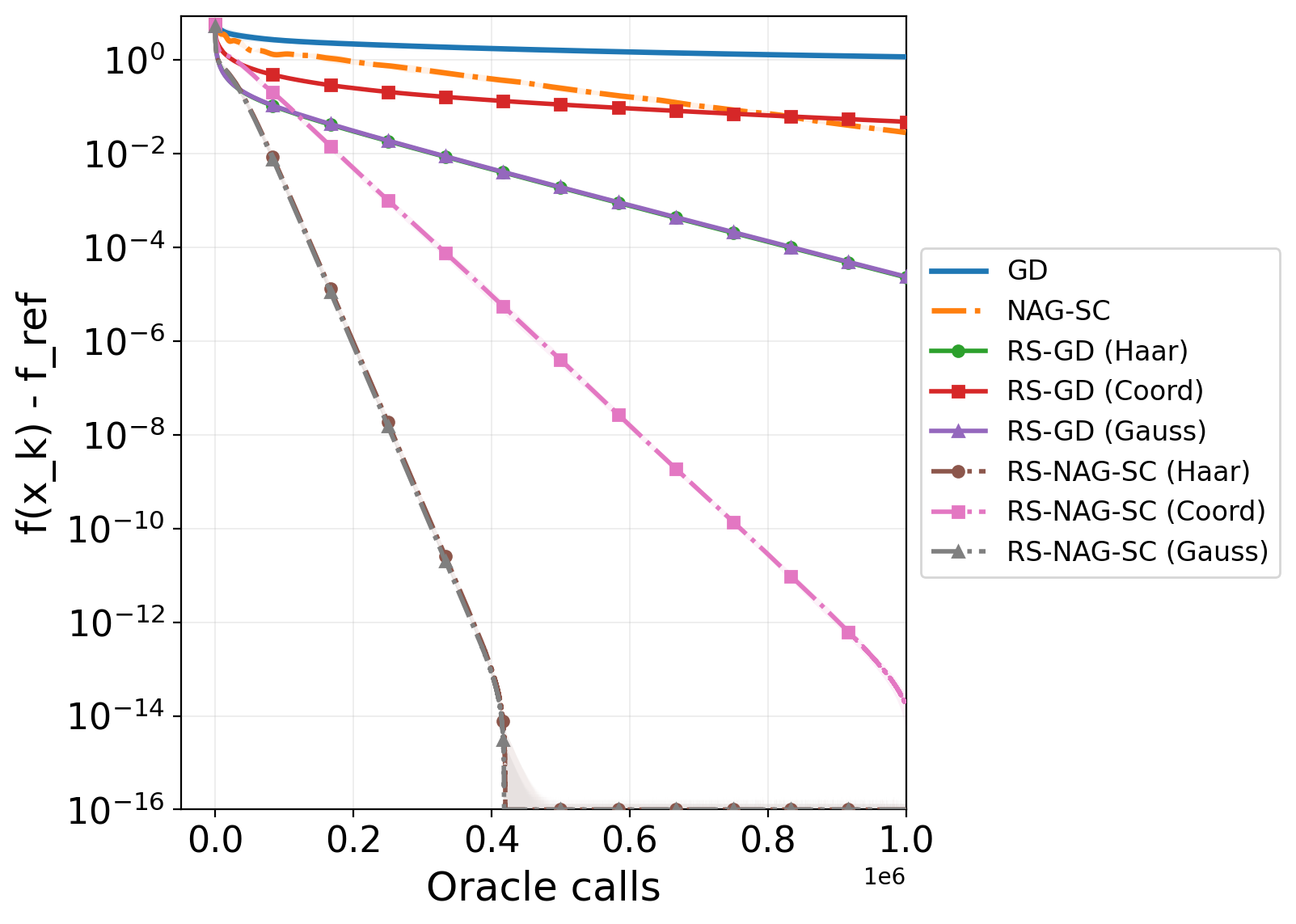}
    \caption{\texttt{bioresponse}}
\end{subfigure}

\vspace{0.6em}

\begin{subfigure}[t]{0.32\textwidth}
    \centering
    \includegraphics[width=\linewidth]{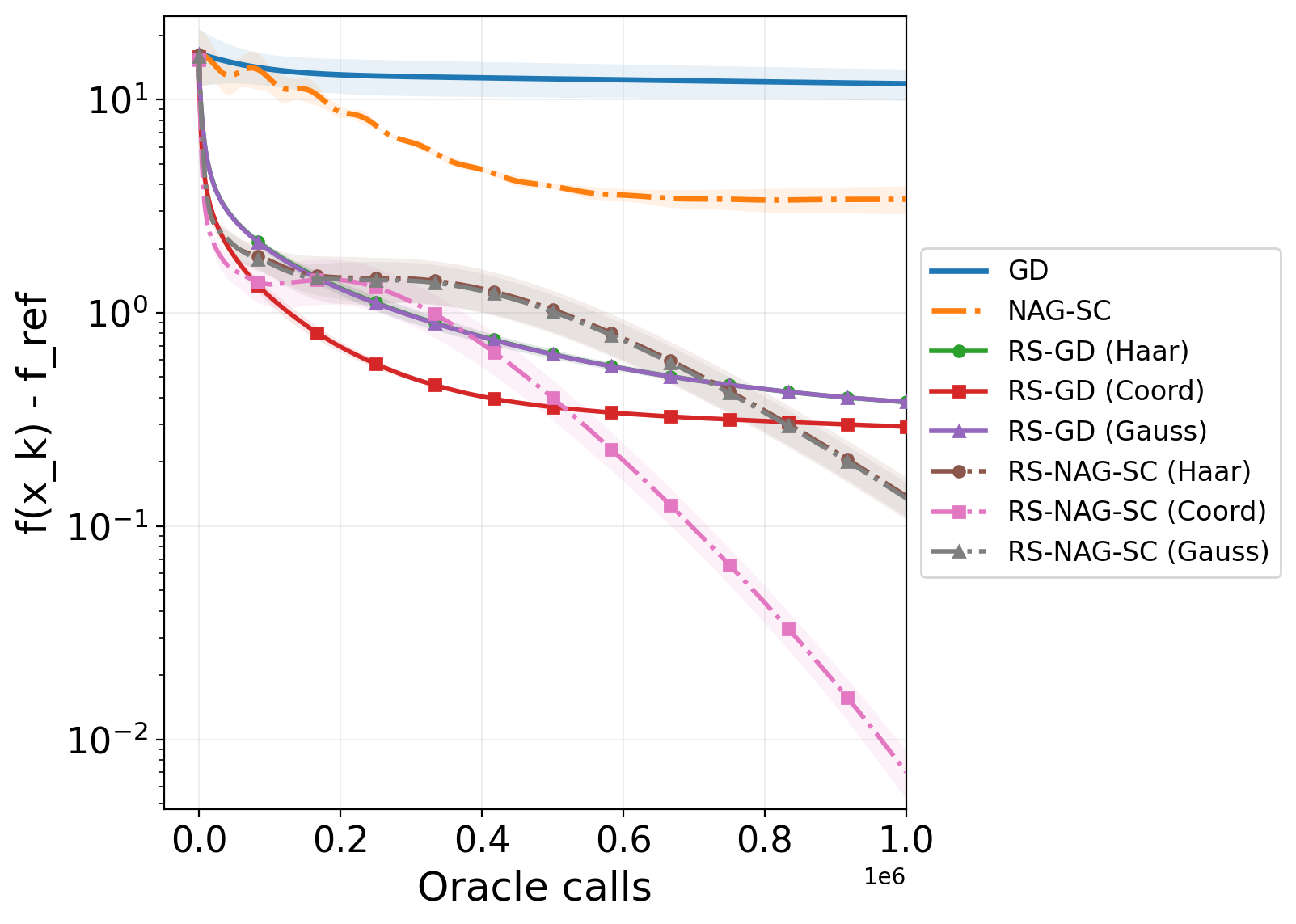}
    \caption{\texttt{gisette}}
\end{subfigure}
\hfill
\begin{subfigure}[t]{0.32\textwidth}
    \centering
    \includegraphics[width=\linewidth]{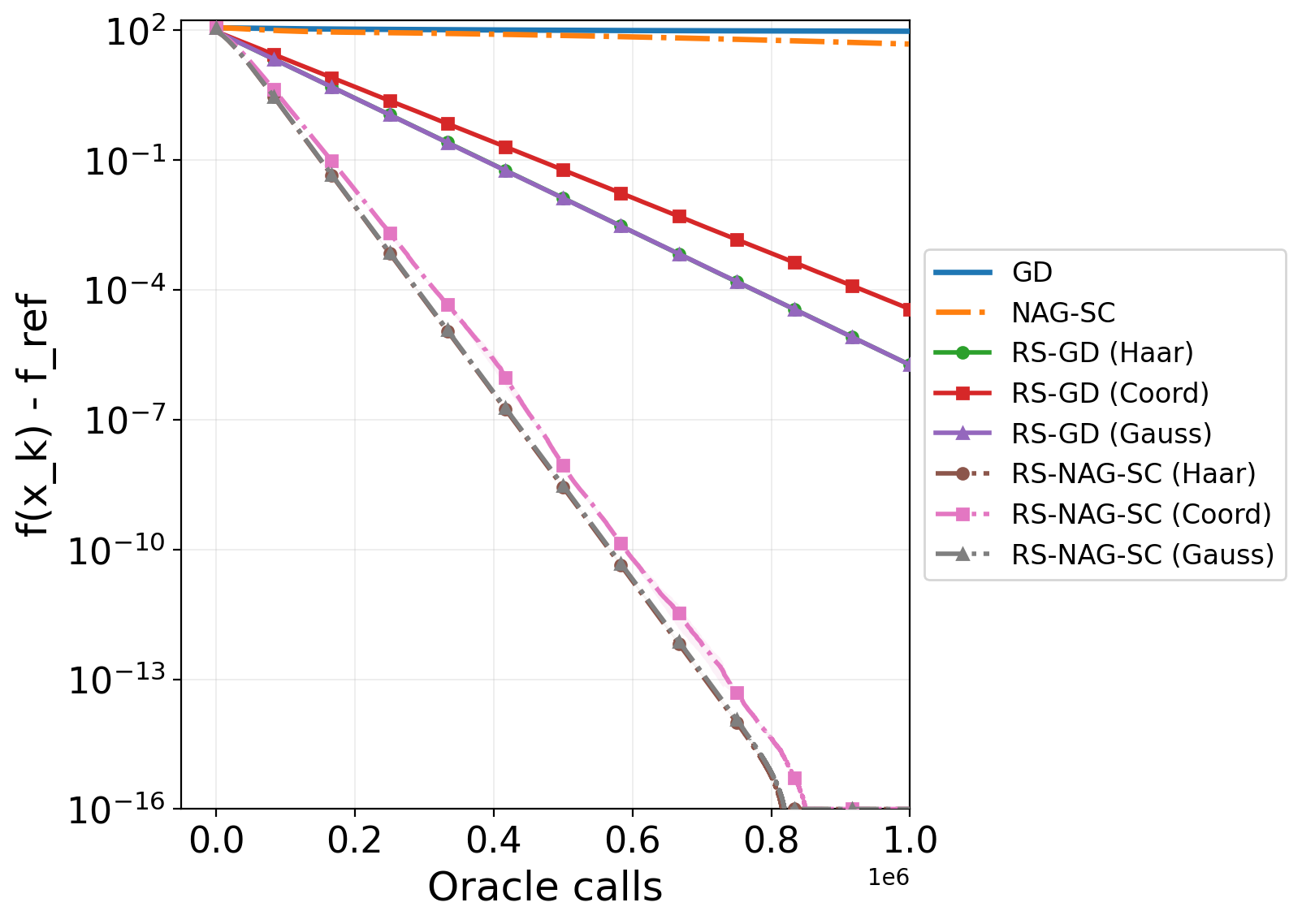}
    \caption{\texttt{leukemia}}
\end{subfigure}
\hfill
\begin{subfigure}[t]{0.32\textwidth}
    \centering
    \includegraphics[width=\linewidth]{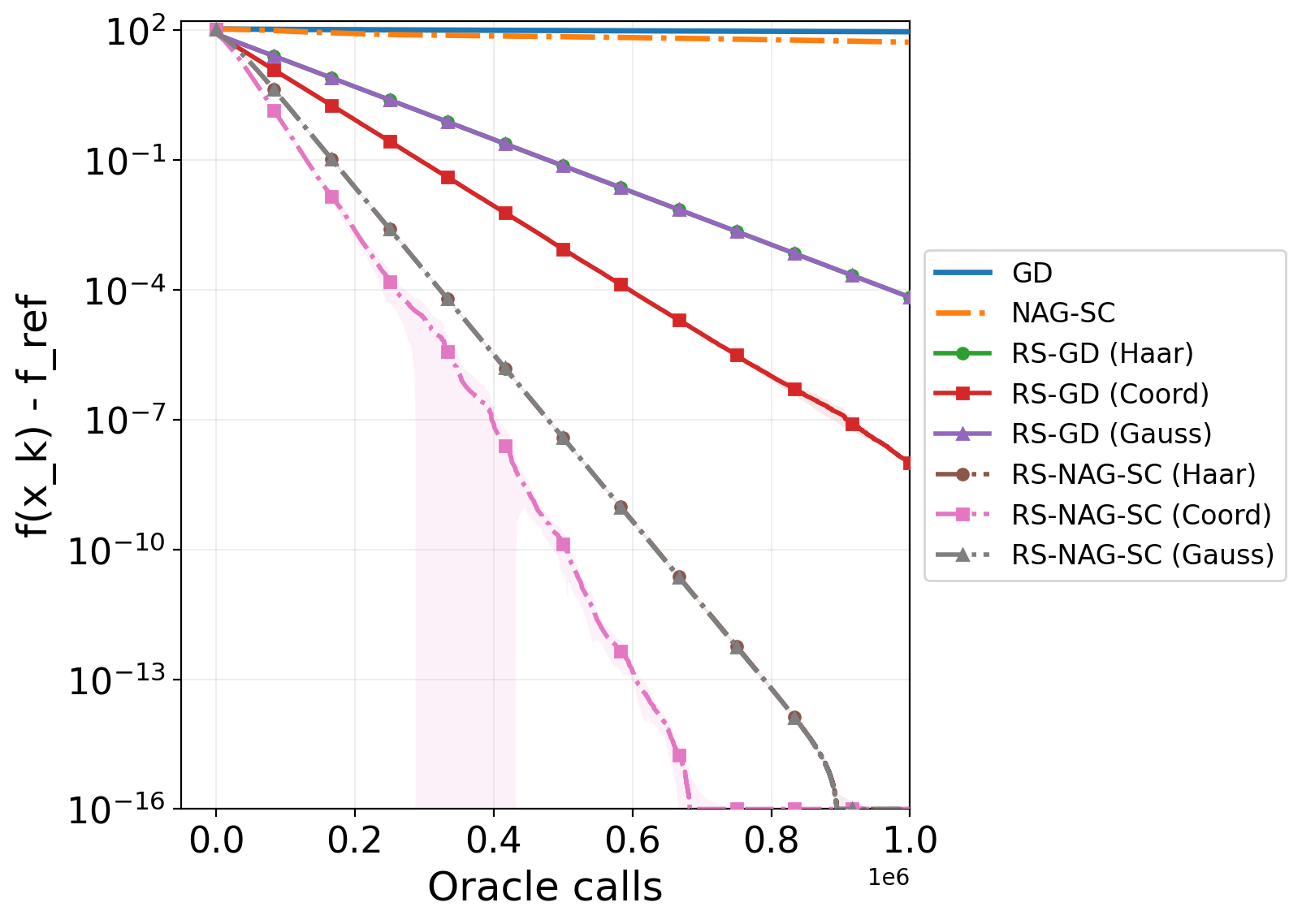}
    \caption{\texttt{duke}}
\end{subfigure}
\caption{
Oracle-axis comparison for \(\ell_2\)-regularized logistic regression on six
real-world datasets. The horizontal axis shows oracle calls, and the vertical
axis shows \(f(x_k)-f_{\mathrm{ref}}\) on a logarithmic scale, where
\(f_{\mathrm{ref}}\) is computed by L-BFGS-B. We compare GD, NAG-SC, RS-GD, and
RS-NAG-SC with Haar, coordinate, and Gaussian sketches. Each curve is the mean
over \(3\) random seeds, and the shaded region shows one standard deviation.
For each dataset, \(\mu=1/n\), \(r=1\), and the oracle budget is \(1{,}000{,}000\).
}
\label{fig:logreg-real-gap-additional}
\end{figure*}

\begin{table}[t]
\centering
\caption{
Dataset-dependent quantities for the datasets.
Here $d$ is the ambient dimension, $n$ is the number of training samples, and
$Q_{\mathrm H}, Q_{\mathrm G}, Q_{\mathrm C}$ denote the $r=1$ constants defined in Proposition~\ref{prop:comp-sketch-basic} for the Haar, Gaussian, and Coordinate sketches, respectively.
}
\label{tab:logreg-q-values-additional}
\small
\setlength{\tabcolsep}{4.5pt}
\renewcommand{\arraystretch}{1.1}
\begin{tabular}{@{}lrrrrrrr@{}}
\toprule
Dataset & $d$ & $Q_{\mathrm H}$ & $Q_{\mathrm G}$ & $Q_{\mathrm C}$ & $r_{\mathrm{eff}}$ & $\delta_{\mathrm{diag}}$ & $n$ \\
\midrule
colon-cancer   & 2000 & 132.1725 & 132.3047 & 116.1338 & 6.7435 & 0.0034 & 62 \\
hiva\_agnostic & 1617 &  92.3197 &  92.4338 & 226.1505 & 3.2773 & 0.0196 & 4229 \\
bioresponse    & 1776 &  88.0942 &  88.1934 & 231.3375 & 2.3746 & 0.0170 & 3751 \\
gisette        & 5000 & 129.9973 & 130.0493 &  86.1400 & 1.3812 & 0.0003 & 6000 \\
leukemia       & 7129 & 215.7402 & 215.8007 & 237.3443 & 4.5306 & 0.0011 & 38 \\
duke           & 7129 & 189.2243 & 189.2773 & 146.8262 & 3.0240 & 0.0004 & 44 \\
\bottomrule
\end{tabular}
\end{table}

\paragraph{Discussion.}
Overall, RS-NAG-SC performs strongly, and the empirical trends align with the
dataset-dependent quantities in \Cref{tab:logreg-q-values-additional}.
For \texttt{hiva\_agnostic}, \texttt{bioresponse}, and \texttt{leukemia}, \(Q_{\mathrm H}\) and \(Q_{\mathrm G}\) are smaller than \(Q_{\mathrm C}\), and the Haar and Gaussian sketches indeed perform better.
By contrast, for \texttt{colon-cancer}, \texttt{gisette}, and \texttt{duke}, \(Q_{\mathrm C}\) is smaller than \(Q_{\mathrm H}\) and \(Q_{\mathrm G}\), which is consistent with the relatively strong performance of the coordinate sketch.
In particular, for \texttt{colon-cancer} and \texttt{leukemia}, where the corresponding \(Q\)-values are relatively close, the empirical performance of the three sketches is also broadly comparable.
Overall, these results suggest that the $Q$ values can serve as a useful practical guide when choosing the sketch distribution before running the method.

\section{Conclusion}

We introduced randomized-subspace Nesterov accelerated gradient methods for
smooth convex and strongly convex optimization under an oracle
model, following the same cost-sensitive viewpoint as 
randomized-subspace methods in general. 
The methods use projected gradients,
recover standard
full-dimensional Nesterov acceleration when \(r=d\), and enjoy oracle-complexity guarantees
under matrix smoothness. For three standard sketch distributions---Haar,
coordinate, and Gaussian sketches---we derived explicit rates, compared the
resulting sketch-dependent constants, and identified the theoretically preferred
subspace dimension under our oracle model.

Future work includes designing sketch distributions beyond the canonical choices considered here. In particular, inspired by non-uniform sampling in accelerated coordinate descent \citep{allenzhu2016}, it would be interesting to develop \(\Lmat\)-aware sketch distributions that lead to faster convergence in oracle calls.

\section*{Acknowledgments}

This project has been partially supported by the Japan Society for the Promotion of Science (JSPS) through JSPS KAKENHI Grant Number JP23H03351 and JST CREST Grant Number JPMJCR24Q2.

\bibliographystyle{plainnat}
\bibliography{references}

\appendix

\section{Related work}
\label{app:related_work}
\subsection{Other random-subspace derivative-free and zeroth-order methods}

Random-subspace ideas have also been studied in related but different
derivative-free and zeroth-order optimization settings. Examples include
model-based derivative-free methods in random subspaces
\cite{cartis2023scalable}, direct-search methods using random subspaces
\cite{roberts2023direct}, zeroth-order methods based on orthogonal random
directions \cite{kozak2023zeroth}, and expected-decrease analyses for
derivative-free algorithms using random subspaces
\cite{hare2025expected}. These works concern derivative-free and zeroth-order
optimization, whereas our focus is on Nesterov-type acceleration under a
directional-derivative oracle.

\subsection{Relation to accelerated compressed-gradient methods}
\label{app:compressed-gradient-comparison}

Several works have studied acceleration in compressed-gradient methods for distributed optimization.
These works are related to our goal of reducing the amount of first-order information used per iteration,
but their models and resulting oracle-complexity implications are different from ours.

First, consider the generic-compressor framework of \citet{li2020acgd}.
If one takes
\[
C(v)=PP^\top v
\]
as a compressor and applies their algorithm, then their compression parameter becomes
\(\omega_{\mathrm{Li}}=\omega-1\) under our notation, since
\(\E[PP^\top]=I_d\) and \(\E[(PP^\top)^2]\preceq \omega I_d\).
Accordingly, their bounds yield oracle complexities
\[
\mathcal{O}\!\left(R_0\,r\omega\sqrt{\frac{L}{\epsilon}}\right)
\qquad\text{and}\qquad
\mathcal{O}\!\left(r\omega\sqrt{\frac{L}{\mu}}\log\frac{\Delta_0}{\epsilon}\right)
\]
in the convex and strongly convex cases, respectively, under our oracle model.
By \Cref{prop:ell-omega-redundant}, we have \(r\omega\ge d\).
Thus, this generic-compressor route is not better than standard full-dimensional Nesterov acceleration
in our oracle model.
In this sense, if \(PP^\top\) is used only through a generic unbiased-compressor framework,
the benefit of randomized-subspace structure does not appear in the resulting oracle bounds.

The work of \citet{safaryan2021smoothness} is also related, as it studies smoothness-aware compression
and accelerated variants under matrix smoothness.
Their setting is distributed: each local loss \(f_i\) is equipped with a local smoothness matrix \(\Lmat_i\),
which is used together with a random diagonal sketch matrix to define a smoothness-aware compression mechanism.
As discussed in their limitations section, this approach requires the server to store \(\Lmat_i^{1/2}\)
for all workers; hence it is not expected
to be practical for large \(d\) unless the matrices \(\Lmat_i\) have special structure, such as low-rank or diagonal structure.
This differs from our setting, where we allow general randomized subspace sketches \(P_k\)
and assume only a global matrix smoothness condition for the objective \(f\).
In particular, our framework does not require local smoothness matrices for individual workers
or the storage of such matrices at the server.
Thus, while both approaches exploit matrix smoothness, the optimization model
and sketching mechanism are different.

\section{Communication perspective on randomized subspace methods}
\label{app:distributed-comparison}

Consider the distributed optimization problem
\[
\min_{x\in\R^d} f(x),
\qquad
f(x)\coloneqq\frac1n\sum_{i=1}^n f_i(x),
\]
where \(n\) is the number of workers and \(f_i\) is the local loss on worker \(i\).
This is a standard objective form in distributed optimization; see, e.g.,
\cite{hong2017proxpda,yang2019survey}.

Each worker sends \(\nabla f_i(y_k)\in\R^d\) to the server, which forms
\[
\nabla f(y_k)=\frac1n\sum_{i=1}^n \nabla f_i(y_k).
\]

To reduce the communication cost per round, we instead use a shared sketch of dimension \(r\).

At iteration \(k\), the server broadcasts the random seed defining \(P_k\in\R^{d\times r}\), and each worker \(i\) computes and sends
\[
s_{i,k}\coloneqq P_k^\top \nabla f_i(y_k)\in\R^r.
\]
The server then averages
\[
\bar s_k\coloneqq\frac1n\sum_{i=1}^n s_{i,k}
\]
and reconstructs
\[
g_k\coloneqq P_k\bar s_k.
\]
By linearity,
\[
g_k
=
P_k\left(\frac1n\sum_{i=1}^n P_k^\top \nabla f_i(y_k)\right)
=
P_kP_k^\top \nabla f(y_k).
\]

Hence the iterates coincide with those of the corresponding
single-machine version of the proposed method applied to \(f\).
Therefore, all convergence guarantees proved for the single-machine method
apply also to the distributed shared-sketch implementation. In particular,
the corresponding convex and strongly convex convergence rates carry over
to the distributed setting.

The matrix smoothness assumption is also natural in this setting.
Indeed, if each local loss \(f_i\) satisfies
\[
f_i(y)\le f_i(x)+\ip{\nabla f_i(x)}{y-x}
+\frac12 (y-x)^\top \Lmat_i (y-x),
\qquad
\Lmat_i\succeq 0,
\]
then averaging over \(i=1,\dots,n\) yields
\[
f(y)\le f(x)+\ip{\nabla f(x)}{y-x}
+\frac12 (y-x)^\top \Lmat (y-x),
\qquad
\Lmat\coloneqq\frac1n\sum_{i=1}^n \Lmat_i.
\]
Such local smoothness matrices \(\Lmat_i\) arise naturally in the matrix-smoothness framework for distributed optimization; see \cite{safaryan2021smoothness}.

\section{Missing Proofs for \Cref{sec:preliminaries} }
\label{app:missing-proofs-prem}

\begin{proof}[Proof of \Cref{prop:ell-omega-redundant}]

Set \(A\coloneqq PP^\top\). Then \(A\) is symmetric positive semidefinite and
\(\operatorname{rank}(A)\le r\). Also, by
\Cref{ass:matrix-smooth-common}, \(\Lmat\succeq0\), \(\Lmat\neq0\), and hence
\(L=\|\Lmat\|>0\).

We prove in order the lower bounds \(\omega\ge d/r\) and \(\ell\ge1\) for
constants satisfying \eqref{eq:second-moment-common} and
\eqref{eq:Lsmooth-sketch-common}, respectively, the product bound
\(\ell\omega\ge1\), and finally that \eqref{eq:Lsmooth-sketch-common} holds
with \(\ell=\omega\) whenever \(\omega\) satisfies
\eqref{eq:second-moment-common}.

\par\smallskip
\noindent\textbf{(i) Lower bound on \(\omega\).}

Let \(\lambda_1,\dots,\lambda_d\ge 0\) be the eigenvalues of \(A\).
Since at least \(d-r\) of them are zero, the Cauchy--Schwarz inequality gives
\[
\bigl(\tr(A)\bigr)^2
=
\Bigl(\sum_{i=1}^d \lambda_i\Bigr)^2
\le
r\sum_{i=1}^d \lambda_i^2
=
r\,\tr(A^2).
\]
Hence
\[
\tr(A^2)\ge \frac{(\tr(A))^2}{r}.
\]
Taking expectations, we obtain
\[
\E[\tr(A^2)]
\ge
\frac{1}{r}\E\bigl[(\tr(A))^2\bigr].
\]
Now apply Jensen's inequality to the convex function \(t\mapsto t^2\):
\[
\E\bigl[(\tr(A))^2\bigr]
\ge
\bigl(\E[\tr(A)]\bigr)^2.
\]
Using the unbiasedness condition \eqref{eq:unbiased-common},
\[
\E[A]=I_d,
\]
we obtain
\[
\E[\tr(A)]
=
\tr(\E[A])
=
\tr(I_d)
=
d.
\]
Therefore,
\[
\E[\tr(A^2)]\ge \frac{d^2}{r}.
\]

On the other hand, taking the trace in \eqref{eq:second-moment-common} gives
\[
\E[\tr(A^2)]
=
\tr(\E[A^2])
\le
\tr(\omega I_d)
=
\omega d.
\]
Combining the lower and upper bounds yields
\[
\omega d \ge \frac{d^2}{r},
\]
and hence
\[
\omega\ge \frac{d}{r}.
\]

\par\smallskip
\noindent\textbf{(ii) Lower bound on \(\ell\).}

Let \(u\in\R^d\) be a unit eigenvector corresponding to the largest
eigenvalue \(L\) of \(\Lmat\), that is,
\[
\|u\|=1,
\qquad
\Lmat u=Lu.
\]
Since \(\Lmat\succeq 0\), we may write its spectral decomposition as
\[
\Lmat
=
Luu^\top + \sum_{i=2}^d \lambda_i v_iv_i^\top,
\qquad
\lambda_i\ge 0.
\]
Therefore,
\[
\Lmat - Luu^\top
=
\sum_{i=2}^d \lambda_i v_iv_i^\top
\succeq 0,
\]
so
\[
\Lmat \succeq Luu^\top.
\]
Therefore, we have
\[
A\Lmat A \succeq L\,Auu^\top A.
\]
Taking the quadratic form with \(u\), we obtain
\[
u^\top A\Lmat A u
\ge
L\,u^\top Auu^\top A u
=
L\,(u^\top Au)^2.
\]
Taking expectation, we get
\[
u^\top \E[A\Lmat A]u
=
\E[u^\top A\Lmat A u]
\ge
L\,\E[(u^\top Au)^2].
\]
Again by Jensen's inequality for the convex function \(t\mapsto t^2\),
\[
\E[(u^\top Au)^2]
\ge
\bigl(\E[u^\top Au]\bigr)^2.
\]
Using once more the unbiasedness condition \(\E[A]=I_d\), we obtain
\[
\bigl(\E[u^\top Au]\bigr)^2
=
(u^\top \E[A]u)^2
=
(u^\top I_du)^2
=
1.
\]
Therefore,
\[
u^\top \E[A\Lmat A]u \ge L.
\]

Now let \(\ell\) be any admissible constant in \eqref{eq:Lsmooth-sketch-common}. Then
\[
\E[A\Lmat A]\preceq \ell L\,I_d.
\]
Taking again the quadratic form with \(u\), we get
\[
u^\top \E[A\Lmat A]u
\le
u^\top (\ell L I_d)u
=
\ell L\|u\|^2
=
\ell L.
\]
Combining the lower and upper bounds yields
\[
L\le \ell L.
\]
Since \(L>0\), we conclude that
\[
\ell\ge 1.
\]

\par\smallskip
\noindent\textbf{(iii) Product bound.}

We have 
\[
\omega\ge \frac{d}{r}\ge 1
\qquad
\text{since }1\le r\le d.
\]
Together with the lower bound \(\ell\ge1\), this implies that for every admissible pair \((\omega,\ell)\),
\[
\ell\omega\ge 1.
\]

\par\smallskip
\noindent\textbf{(iv) Admissibility of \(\ell=\omega\).}

Since \(\Lmat\succeq 0\) and \(L=\|\Lmat\|\), we have
\[
\Lmat\preceq LI_d.
\]
Hence
\[
A\Lmat A \preceq LA^2.
\]
Taking expectations and using \eqref{eq:second-moment-common}, we obtain
\[
\E[A\Lmat A]
\preceq
L\,\E[A^2]
\preceq
\omega L\,I_d.
\]
That is,
\[
\E[PP^\top \Lmat PP^\top]\preceq \omega L\,I_d.
\]
Therefore, the admissible constants \(\ell\) in \eqref{eq:Lsmooth-sketch-common}
contain \(\omega\). Hence one may choose \(\ell\) so that
\[
\ell\le \omega.
\]

This completes the proof.
\end{proof}

\section{Missing Proofs for \Cref{sec:rs-NAG-c} }
\label{app:missing-proofs-convex}
\begin{proof}[Proof of \Cref{prop:convex-reduction-nesterov}]

We show that the resulting \((x_k,y_k)\)-sequence coincides with the classical
two-sequence Nesterov accelerated gradient method
\begin{align}
x_{k+1}^{\mathrm N}
&\coloneqq y_k^{\mathrm N} - \frac{1}{L}\nabla f(y_k^{\mathrm N}),
\label{eq:NAG-conv-x-final}\\
t_{k+1} &\coloneqq \frac{1+\sqrt{1+4t_k^2}}{2},
\label{eq:NAG-conv-t-final}\\
\beta_k &\coloneqq \frac{t_k-1}{t_{k+1}},
\label{eq:NAG-conv-beta-final}\\
y_{k+1}^{\mathrm N}
&\coloneqq x_{k+1}^{\mathrm N} + \beta_k\bigl(x_{k+1}^{\mathrm N}-x_k^{\mathrm N}\bigr),
\label{eq:NAG-conv-y-final}
\end{align}
with initialization \(x_0^{\mathrm N}=x_0\), \(y_0^{\mathrm N}=x_0\),
and \(t_0=1\).

Since \(r=d\) and \(P_kP_k^\top=I_d\) for all \(k\), we have almost surely
\[
(P_kP_k^\top)^2 = I_d,
\qquad
P_kP_k^\top \Lmat P_kP_k^\top = \Lmat.
\]
Taking expectations,
\begin{align*}
\E[PP^\top] &= I_d,\\
\E[(PP^\top)^2] &= I_d,\\
\E[PP^\top \Lmat PP^\top] &= \Lmat.
\end{align*}
Hence \Cref{ass:sketch-common} is satisfied with the
admissible choice
\[
\omega = 1,
\qquad
\ell = 1.
\]

By \Cref{alg:rs-NAG-c}, the iterates satisfy
\begin{align}
&m = \frac{1}{2L\ell}, \label{eq:conv-red-m}\\
&A_{k+1} = A_k+a_{k+1}, \label{eq:conv-red-A}\\
&y_k = \frac{A_k}{A_{k+1}}x_k + \frac{a_{k+1}}{A_{k+1}}z_k, \label{eq:conv-red-y}\\
&x_{k+1} = y_k - \frac{1}{L\ell}P_kP_k^\top \nabla f(y_k), \label{eq:conv-red-x}\\
&z_{k+1} = z_k - a_{k+1}P_kP_k^\top \nabla f(y_k). \label{eq:conv-red-z}
\end{align}
We now fix \(\omega=\ell=1\). By \eqref{eq:conv-red-m},
\[
m=\frac{1}{2L}.
\]
By the scalar update in \Cref{alg:rs-NAG-c}, \(a_{k+1}>0\) satisfies
\begin{equation}
m(A_k+a_{k+1})=\frac{\omega}{2}a_{k+1}^2.
\label{eq:conv-red-ak}
\end{equation}
Hence, since \(\omega=1\),
\[
\frac{1}{2L}(A_k+a_{k+1})=\frac{1}{2}a_{k+1}^2.
\]
By \eqref{eq:conv-red-A}, this is equivalent to
\[
\frac{1}{2L}A_{k+1}=\frac{1}{2}a_{k+1}^2
\quad\Longleftrightarrow\quad
A_{k+1}=La_{k+1}^2.
\]

By \eqref{eq:conv-red-y},
\[
y_k = \frac{A_k}{A_{k+1}}x_k + \frac{a_{k+1}}{A_{k+1}}z_k.
\]
By \eqref{eq:conv-red-x}, together with \(P_kP_k^\top = I_d\) and \(\ell=1\),
\[
x_{k+1}
=
y_k - \frac{1}{L\ell}P_kP_k^\top\nabla f(y_k)
=
y_k - \frac{1}{L}\nabla f(y_k),
\]
and \eqref{eq:conv-red-z} becomes
\[
z_{k+1}
=
z_k - a_{k+1}P_kP_k^\top \nabla f(y_k)
=
z_k - a_{k+1}\nabla f(y_k).
\]
Thus, in the case \(r=d\) with \(P_kP_k^\top=I_d\),  the updates become
\begin{align}
y_k &= \frac{A_k}{A_{k+1}}x_k + \frac{a_{k+1}}{A_{k+1}}z_k,
\label{eq:conv-full-y-proof}\\
x_{k+1} &= y_k - \frac{1}{L}\nabla f(y_k),
\label{eq:conv-full-x-proof}\\
z_{k+1} &= z_k - a_{k+1}\nabla f(y_k),
\label{eq:conv-full-z-proof}\\
A_{k+1} &= A_k + a_{k+1},\quad A_{k+1}=La_{k+1}^2.
\label{eq:conv-full-A-proof}
\end{align}

We next analyze the scalar sequence \(\{A_k\}\).
Define
\[
t_k\coloneqq\sqrt{L A_{k+1}},\qquad k\ge 0.
\]
Since \(A_0=0\), at \(k=0\) we have from \eqref{eq:conv-full-A-proof}
\[
A_1=A_0+a_1=a_1,
\qquad
A_1=La_1^2.
\]
Hence
\[
a_1=La_1^2,
\]
so the nonzero solution is \(a_1=1/L\). Therefore
\[
A_1=A_0+a_1=\frac{1}{L},
\qquad
t_0=\sqrt{L A_1}=1.
\]

For general \(k\ge 0\), combining \(A_{k+2}=La_{k+2}^2\) from
\eqref{eq:conv-full-A-proof} with \(A_{k+2}=A_{k+1}+a_{k+2}\) yields
\[
La_{k+2}^2 = A_{k+1}+a_{k+2}.
\]
Multiplying both sides by \(L\) and using
\[
t_k^2=L A_{k+1},
\qquad
t_{k+1}^2=L A_{k+2}=L^2 a_{k+2}^2,
\]
we obtain
\[
t_{k+1}^2 = t_k^2 + L a_{k+2}.
\]
On the other hand, since
\[
A_{k+2}=La_{k+2}^2,
\]
we have
\[
L a_{k+2}
=
\sqrt{L^2 a_{k+2}^2}
=
\sqrt{L A_{k+2}}
=
t_{k+1},
\]
where in the first equality we use \(a_{k+2}>0\) to select the positive root.
Therefore
\[
t_{k+1}^2 = t_k^2 + t_{k+1},
\]
that is,
\[
t_{k+1}^2 - t_{k+1} - t_k^2 = 0.
\]
For each fixed \(t_k\), the above quadratic equation has the positive root
\begin{equation}
t_{k+1}
=
\frac{1+\sqrt{1+4t_k^2}}{2},
\qquad k\ge 0,
\label{eq:t-recursion}
\end{equation}
since \(t_{k+1}>0\).
This is exactly the recursion \eqref{eq:NAG-conv-t-final}.

We now eliminate $z_k$ and derive the two-sequence recursion for $(x_k,y_k)$.
From \eqref{eq:conv-full-y-proof} we have
\begin{equation}
z_k = \frac{A_{k+1}y_k - A_k x_k}{a_{k+1}}
\label{eq:conv-z-from-xy}
\end{equation}
so for all $k\ge 0$.
Using \eqref{eq:conv-full-x-proof},
\[
\nabla f(y_k) = L\bigl(y_k - x_{k+1}\bigr),
\]
and substituting into the update \eqref{eq:conv-full-z-proof} together with
\eqref{eq:conv-z-from-xy} gives
\begin{align*}
z_{k+1}
&= z_k - a_{k+1}\nabla f(y_k)\\
&= \frac{A_{k+1}y_k - A_k x_k}{a_{k+1}} - a_{k+1}L\bigl(y_k-x_{k+1}\bigr).
\end{align*}
Since $A_{k+1}=La_{k+1}^2$, we have $La_{k+1}=A_{k+1}/a_{k+1}$, hence
\begin{equation}
z_{k+1}
=
\frac{A_{k+1}y_k - A_k x_k}{a_{k+1}}
-
\frac{A_{k+1}}{a_{k+1}}\bigl(y_k-x_{k+1}\bigr)
=
\frac{A_{k+1}x_{k+1} - A_k x_k}{a_{k+1}}.
\label{eq:z-kplus1}
\end{equation}

Next, from \eqref{eq:conv-full-y-proof} with $k$ replaced by $k+1$ we obtain
\[
y_{k+1}
=
\frac{A_{k+1}}{A_{k+2}}x_{k+1}
+
\frac{a_{k+2}}{A_{k+2}}z_{k+1}.
\]
Substituting the expression for $z_{k+1}$ from \eqref{eq:z-kplus1} yields
\begin{align*}
y_{k+1}
&=
\frac{A_{k+1}}{A_{k+2}}x_{k+1}
+
\frac{a_{k+2}}{A_{k+2}}
\cdot
\frac{A_{k+1}x_{k+1}-A_k x_k}{a_{k+1}}\\
&=
\frac{A_{k+1}}{A_{k+2}}\Bigl(1+\frac{a_{k+2}}{a_{k+1}}\Bigr)x_{k+1}
-
\frac{a_{k+2}A_k}{a_{k+1}A_{k+2}}x_k.
\end{align*}
Define
\[
\beta_k \coloneqq \frac{a_{k+2}A_k}{a_{k+1}A_{k+2}}.
\]
Then the coefficient of $x_k$ is $-\beta_k$, and the coefficient of $x_{k+1}$ is
\begin{align*}
\frac{A_{k+1}}{A_{k+2}}\Bigl(1+\frac{a_{k+2}}{a_{k+1}}\Bigr)
& =
\frac{a_{k+1} A_{k+1}+a_{k+2}A_{k+1}}{a_{k+1} A_{k+2}}  \\
& =\frac{a_{k+1} (A_{k+2}-a_{k+2})+a_{k+2}(A_k+a_{k+1})}{a_{k+1} A_{k+2}} = 1+\beta_k,
\end{align*}
where the second equality is derived from 
 $A_{k+2}=A_{k+1}+a_{k+2}$ and $A_{k+1}=A_k+a_{k+1}$.
Thus, we have
\[
y_{k+1}
=
x_{k+1} + \beta_k\bigl(x_{k+1}-x_k\bigr).
\]
It remains to express $\beta_k$ in terms of $\{t_k\}$ and identify it with
\eqref{eq:NAG-conv-beta-final}.
By the definition
\[
t_k\coloneqq\sqrt{L A_{k+1}}
\qquad (k\ge 0),
\]
we have
\[
A_k=\frac{t_{k-1}^2}{L},
\qquad
A_{k+2}=\frac{t_{k+1}^2}{L}
\qquad (k\ge 1).
\]
Moreover, since
\[
A_{k+1}=La_{k+1}^2,
\qquad
a_{k+1}>0,
\]
it follows that
\[
t_k=\sqrt{L A_{k+1}}=\sqrt{L^2a_{k+1}^2}=La_{k+1},
\qquad
t_{k+1}=La_{k+2},
\]
and hence
\[
a_{k+1}=\frac{t_k}{L},
\qquad
a_{k+2}=\frac{t_{k+1}}{L}.
\]
Therefore, for $k\ge 1$,
\begin{align*}
\beta_k
&=
\frac{a_{k+2}A_k}{a_{k+1}A_{k+2}}
=
\frac{\displaystyle \frac{t_{k+1}}{L}\cdot\frac{t_{k-1}^2}{L}}
     {\displaystyle \frac{t_k}{L}\cdot\frac{t_{k+1}^2}{L}}
=
\frac{t_{k-1}^2}{t_k t_{k+1}}.
\end{align*}
Under the positivity condition, the recursion \eqref{eq:NAG-conv-t-final} is equivalent to
\[
t_{j+1}^2 - t_{j+1} - t_j^2 = 0,
\qquad j\ge 0.
\]
In particular, for $k\ge 1$, applying this with $j=k-1$ gives
\[
t_{k-1}^2 = t_k^2 - t_k.
\]
Substituting into the expression for $\beta_k$ yields
\[
\beta_k
=
\frac{t_k^2 - t_k}{t_k t_{k+1}}
=
\frac{t_k-1}{t_{k+1}},
\]
which coincides with \eqref{eq:NAG-conv-beta-final} for $k\ge 1$.
For $k=0$ we have $A_0=0$, so $\beta_0=0$, and also
\(\beta_0=(t_0-1)/t_1=0\), so the formula holds for all $k\ge 0$.

Summarizing, the pair $(x_k,y_k)$ generated by
\eqref{eq:conv-full-y-proof}--\eqref{eq:conv-full-A-proof} satisfies
\[
x_{k+1} = y_k - \frac{1}{L}\nabla f(y_k),
\qquad
y_{k+1} = x_{k+1} + \frac{t_k-1}{t_{k+1}}\bigl(x_{k+1}-x_k\bigr),
\]
with the same initialization $x_0=y_0=x_0^{\mathrm N}=y_0^{\mathrm N}$ and the same sequence
$\{t_k\}$ as in \eqref{eq:NAG-conv-x-final}–\eqref{eq:NAG-conv-y-final}.
Thus $(x_k,y_k)$ obeys exactly the same recursion and initial conditions as
$(x_k^{\mathrm N},y_k^{\mathrm N})$, and hence
\[
x_k = x_k^{\mathrm N},
\qquad
y_k = y_k^{\mathrm N}
\quad\text{for all }k\ge 0.
\]
This completes the proof.
\end{proof}

\begin{proof}[Proof of \Cref{thm:rs-nesterov-convex}]
Define
\[
\Psi_k\coloneqq\E\!\left[
A_k\bigl(f(x_k)-f^\star\bigr)
+\frac12\norm{z_k-x^\star}^2
\right].
\]

Fix \(k\ge 0\), and let
\[
\F_k\coloneqq\sigma(P_0,\dots,P_{k-1}),
\qquad
g_k\coloneqq\nabla f(y_k).
\]

By \Cref{alg:rs-NAG-c}, the iterates satisfy
\begin{align}
&m = \frac{1}{2L\ell}, \label{eq:conv-thm-m}\\
&A_{k+1} = A_k+a_{k+1}, \label{eq:conv-thm-A}\\
&y_k = \frac{A_k}{A_{k+1}}x_k+\frac{a_{k+1}}{A_{k+1}}z_k, \label{eq:conv-thm-y}\\
&x_{k+1} = y_k-\frac{1}{L\ell}P_kP_k^\top \nabla f(y_k), \label{eq:conv-thm-x}\\
&z_{k+1} = z_k-a_{k+1}P_kP_k^\top \nabla f(y_k). \label{eq:conv-thm-z}
\end{align}
Moreover, by the scalar update in \Cref{alg:rs-NAG-c}, \(a_{k+1}>0\) satisfies
\begin{equation}
m(A_k+a_{k+1})=\frac{\omega}{2}a_{k+1}^2.
\label{eq:conv-thm-ak}
\end{equation}
Equivalently, by \eqref{eq:conv-thm-A},
\begin{equation}
mA_{k+1}=\frac{\omega}{2}a_{k+1}^2.
\label{eq:conv-thm-akA}
\end{equation}

By \eqref{eq:conv-thm-x} and \eqref{eq:matrix-smooth-common} applied with
\(x=y_k\) and \(y=x_{k+1}\), we have
\begin{align}
f(x_{k+1})
&\le f(y_k)+\ip{\nabla f(y_k)}{x_{k+1}-y_k}
+\frac12 (x_{k+1}-y_k)^\top \Lmat (x_{k+1}-y_k) \\
&= f(y_k)-\frac{1}{L\ell}\ip{g_k}{P_kP_k^\top g_k}
+\frac{1}{2L^2\ell^2} g_k^\top P_kP_k^\top \Lmat P_kP_k^\top g_k.
\end{align}
Taking conditional expectation and using
\eqref{eq:unbiased-common} and \eqref{eq:Lsmooth-sketch-common},
\begin{align}
\E[f(x_{k+1})\mid \F_k]
&\le f(y_k)-\frac{1}{L\ell}\norm{g_k}^2
+\frac{1}{2L^2\ell^2}g_k^\top \E[P_kP_k^\top\Lmat P_kP_k^\top\mid \F_k] g_k \\
&\le f(y_k)-\frac{1}{L\ell}\norm{g_k}^2
+\frac{1}{2L\ell}\norm{g_k}^2 \\
&= f(y_k)-\frac{1}{2L\ell}\norm{g_k}^2 \\
&= f(y_k)-m\norm{g_k}^2,
\label{eq:x-step-convex}
\end{align}

By \eqref{eq:conv-thm-z},
\[
z_{k+1}-x^\star = (z_k-x^\star)-a_{k+1}P_kP_k^\top g_k.
\]
Hence
\begin{align}
\frac12 \norm{z_{k+1}-x^\star}^2
&=
\frac12 \norm{z_k-x^\star}^2
-a_{k+1}\ip{z_k-x^\star}{P_kP_k^\top g_k}
+\frac{a_{k+1}^2}{2}\norm{P_kP_k^\top g_k}^2.
\end{align}
Taking conditional expectation and using
\eqref{eq:unbiased-common}, \eqref{eq:second-moment-common},
\begin{align}
\E\!\left[\frac12 \norm{z_{k+1}-x^\star}^2 \middle| \F_k\right]
&\le
\frac12 \norm{z_k-x^\star}^2
-a_{k+1}\ip{z_k-x^\star}{g_k}
+\frac{\omega a_{k+1}^2}{2}\norm{g_k}^2.
\label{eq:z-step-convex}
\end{align}

Subtract \(f^\star\) from both sides of \eqref{eq:x-step-convex}, multiply by \(A_{k+1}\),
and then add \eqref{eq:z-step-convex}. This gives
\begin{align}
&\E\!\left[
A_{k+1}\bigl(f(x_{k+1})-f^\star\bigr)+\frac12 \norm{z_{k+1}-x^\star}^2
\middle| \F_k\right]
\nonumber\\
&\le
A_{k+1}\bigl(f(y_k)-f^\star\bigr)
+\frac12 \norm{z_k-x^\star}^2
-a_{k+1}\ip{z_k-x^\star}{g_k}
\nonumber\\
&\qquad
-\left(A_{k+1}m-\frac{\omega a_{k+1}^2}{2}\right)\norm{g_k}^2.
\label{eq:sum-convex-before}
\end{align}
By \eqref{eq:conv-thm-akA}, the last coefficient is zero, hence
\begin{align}
&\E\!\left[
A_{k+1}\bigl(f(x_{k+1})-f^\star\bigr)+\frac12 \norm{z_{k+1}-x^\star}^2
\middle| \F_k\right]
\nonumber\\
&\le
A_{k+1}\bigl(f(y_k)-f^\star\bigr)
+\frac12 \norm{z_k-x^\star}^2
-a_{k+1}\ip{z_k-x^\star}{g_k}.
\label{eq:sum-convex-after}
\end{align}

From \eqref{eq:conv-thm-y},
\[
A_{k+1}y_k = A_k x_k + a_{k+1} z_k.
\]
Therefore
\[
A_k(x_k-y_k)+a_{k+1}(x^\star-y_k)=a_{k+1}(x^\star-z_k).
\]
Taking inner product with \(g_k\),
\begin{equation}
-a_{k+1}\ip{z_k-x^\star}{g_k}
=
A_k\ip{x_k-y_k}{g_k}+a_{k+1}\ip{x^\star-y_k}{g_k}.
\label{eq:key-identity-convex}
\end{equation}

By convexity of \(f\),
\begin{align}
f(x_k) &\ge f(y_k)+\ip{g_k}{x_k-y_k},
\label{eq:convex-1}\\
f^\star=f(x^\star) &\ge f(y_k)+\ip{g_k}{x^\star-y_k}.
\label{eq:convex-2}
\end{align}
Multiply \eqref{eq:convex-1} by \(A_k\), \eqref{eq:convex-2} by \(a_{k+1}\), and add:
\[
A_k\ip{x_k-y_k}{g_k}+a_{k+1}\ip{x^\star-y_k}{g_k}
\le
A_k\bigl(f(x_k)-f(y_k)\bigr)+a_{k+1}\bigl(f^\star-f(y_k)\bigr).
\]
Using \eqref{eq:key-identity-convex}, we obtain
\begin{equation}
-a_{k+1}\ip{z_k-x^\star}{g_k}
\le
A_k\bigl(f(x_k)-f(y_k)\bigr)+a_{k+1}\bigl(f^\star-f(y_k)\bigr).
\label{eq:key-upper-convex}
\end{equation}

Substitute \eqref{eq:key-upper-convex} into \eqref{eq:sum-convex-after}:
\begin{align}
&\E\!\left[
A_{k+1}\bigl(f(x_{k+1})-f^\star\bigr)+\frac12 \norm{z_{k+1}-x^\star}^2
\middle| \F_k\right]
\nonumber\\
&\le
A_{k+1}\bigl(f(y_k)-f^\star\bigr)
+\frac12 \norm{z_k-x^\star}^2
+A_k\bigl(f(x_k)-f(y_k)\bigr)+a_{k+1}\bigl(f^\star-f(y_k)\bigr)
\nonumber\\
&=
A_k\bigl(f(x_k)-f(y_k)\bigr)
+(A_{k+1}-a_{k+1})\bigl(f(y_k)-f^\star\bigr)
+\frac12 \norm{z_k-x^\star}^2
\nonumber\\
&=
A_k\bigl(f(x_k)-f^\star\bigr)+\frac12 \norm{z_k-x^\star}^2,
\label{eq:conditional-psi-contraction}
\end{align}

By nonnegativity, the tower property applies to
\eqref{eq:conditional-psi-contraction}; hence, taking expectation yields
\[
\Psi_{k+1}\le \Psi_k.
\]
Iterating gives \(\Psi_N\le\Psi_0\). Since \(A_0=0\) and \(z_0=x_0\),
\[
\Psi_0=\frac12\norm{x_0-x^\star}^2.
\]
Therefore,
\begin{align}
\label{eq:AN-bound-convex}
A_N\,\E[f(x_N)-f^\star]
\le
\Psi_N
\le
\Psi_0
=
\frac12 \norm{x_0-x^\star}^2.
\end{align}

It remains to lower bound \(A_N\). Define
\[
\rho\coloneqq\frac{2m}{\omega}>0.
\]
From \eqref{eq:conv-thm-akA},
\[
mA_{k+1}=\frac{\omega}{2}a_{k+1}^2,
\]
or equivalently,
\[
A_{k+1}=\frac{a_{k+1}^2}{\rho}.
\]
Since \(A_{k+1}=A_k+a_{k+1}\), we have
\[
A_{k+1}-A_k=\sqrt{\rho A_{k+1}}.
\]
Let
\[
B_k\coloneqq\sqrt{\frac{A_k}{\rho}}.
\]
Then
\[
B_{k+1}^2-B_k^2=B_{k+1}.
\]
Hence
\[
(B_{k+1}-B_k)(B_{k+1}+B_k)=B_{k+1}.
\]
Since \(A_{k+1}=A_k+a_{k+1}\) with \(a_{k+1}>0\), we have \(A_{k+1}>A_k\ge 0\), so
\(B_{k+1}> B_k\ge 0\) and in particular \(B_{k+1}+B_k>0\). 
Thus,
\[
B_{k+1}-B_k=\frac{B_{k+1}}{B_{k+1}+B_k}\ge \frac12.
\]
Since \(B_0=0\), summing gives
\[
B_N\ge \frac{N}{2}.
\]
Therefore,
\begin{equation}
A_N
=
\rho B_N^2
\ge
\frac{\rho N^2}{4}
=
\frac{m}{2\omega}N^2.
\label{eq:AN-lower}
\end{equation}
Combining \eqref{eq:AN-bound-convex} and \eqref{eq:AN-lower}, for all \(N\ge1\),
\[
\E[f(x_N)-f^\star]
\le
\frac{1}{2A_N}\norm{x_0-x^\star}^2
\le
\frac{\omega}{m}\frac{\norm{x_0-x^\star}^2}{N^2}.
\]
Since \(m=1/(2L\ell)\), this gives
\begin{equation}
\E[f(x_N)-f^\star]
\le
2L\omega\ell\,
\frac{\norm{x_0-x^\star}^2}{N^2},
\end{equation}
which is \eqref{eq:convex-final-rate}.

Finally, let \(R_0\coloneqq\norm{x_0-x^\star}\). Solving
\[
2L\omega\ell\,\frac{R_0^2}{N^2}\le \epsilon
\]
gives
\[
N \;\ge\; R_0\sqrt{\frac{2L\omega\ell}{\epsilon}}.
\]
This completes the proof.

\end{proof}

\section{Missing Proofs for \Cref{sec:rs-NAG-sc} }
\label{app:missing-proofs-strong}

\begin{proof}[Proof of \Cref{prop:strong-reduction-nesterov}]
We show that the resulting \((x_k,y_k)\)-sequence coincides with the classical
two-sequence accelerated gradient method
\begin{align}
x_{k+1}^{\mathrm{N}}
&\coloneqq y_k^{\mathrm{N}} - \frac{1}{L}\nabla f(y_k^{\mathrm{N}}),
\label{eq:NAG-sc-x-final}\\
y_{k+1}^{\mathrm{N}}
&\coloneqq x_{k+1}^{\mathrm{N}} + \beta\bigl(x_{k+1}^{\mathrm{N}}-x_k^{\mathrm{N}}\bigr),
\label{eq:NAG-sc-y-final}
\end{align}
where
\begin{equation}
\beta \coloneqq \frac{\sqrt{L}-\sqrt{\mu}}{\sqrt{L}+\sqrt{\mu}}
\label{eq:beta-classical}
\end{equation}
with initialization \(x_0^{\mathrm{N}}=x_0\) and \(y_0^{\mathrm{N}}=x_0\).

As in the convex case, the assumptions \(r=d\) and \(P_kP_k^\top=I_d\)
imply
\[
(P_kP_k^\top)^2 = I_d,
\qquad
P_kP_k^\top \Lmat P_kP_k^\top = \Lmat
\]
almost surely, so \Cref{ass:sketch-common} is satisfied
with
\[
\omega = 1,
\qquad
\ell = 1.
\]

By \Cref{alg:rs-NAG-sc}, the iterates satisfy
\begin{align}
&\theta = \sqrt{\frac{\mu}{L\omega\ell}}, \label{eq:sc-red-theta}\\
&y_k = \frac{1}{1+\theta}x_k + \frac{\theta}{1+\theta}z_k, \label{eq:sc-red-y}\\
&x_{k+1} = y_k - \frac{1}{L\ell}P_kP_k^\top \nabla f(y_k), \label{eq:sc-red-x}\\
&z_{k+1} = (1-\theta)z_k + \theta y_k - \frac{\theta}{\mu}P_kP_k^\top \nabla f(y_k). \label{eq:sc-red-z}
\end{align}

We now fix \(\omega=\ell=1\). By \eqref{eq:sc-red-theta},
\[
\theta=\sqrt{\frac{\mu}{L}}.
\]
Using \eqref{eq:sc-red-y}--\eqref{eq:sc-red-z} together with \(P_kP_k^\top=I_d\), we obtain
\begin{align}
y_k &= \frac{1}{1+\theta}x_k + \frac{\theta}{1+\theta}z_k,
\label{eq:sc-full-y-proof}\\
x_{k+1} &= y_k - \frac{1}{L}\nabla f(y_k),
\label{eq:sc-full-x-proof}\\
z_{k+1} &= (1-\theta)z_k + \theta y_k - \frac{\theta}{\mu}\nabla f(y_k).
\label{eq:sc-full-z-proof}
\end{align}

We first rewrite the \(z\)-update.
From \eqref{eq:sc-full-x-proof},
\[
x_{k+1} = y_k - \frac{1}{L}\nabla f(y_k)
\quad\Longrightarrow\quad
\nabla f(y_k) = L\bigl(y_k-x_{k+1}\bigr).
\]
Since \(\mu=\theta^2 L\), it follows that
\[
\frac{\theta}{\mu}\nabla f(y_k)
=
\frac{\theta}{\theta^2 L}L\bigl(y_k-x_{k+1}\bigr)
=
\frac{1}{\theta}\bigl(y_k-x_{k+1}\bigr).
\]
Substituting this into the update for \(z_{k+1}\) yields
\begin{align}
z_{k+1}
&=
(1-\theta)z_k + \theta y_k - \frac{1}{\theta}\bigl(y_k-x_{k+1}\bigr)
\nonumber\\
&=
(1-\theta)z_k + \Bigl(\theta-\frac{1}{\theta}\Bigr)y_k + \frac{1}{\theta}x_{k+1}.
\label{eq:z-update-expanded-proof}
\end{align}
We now derive a convenient representation of \(z_k\) in terms of
\(\{x_j\}_{j\le k}\).
From \eqref{eq:sc-full-y-proof},
\begin{equation}
(1+\theta) y_k = x_k + \theta z_k.
\label{eq:y-relation}
\end{equation}
Thus
\begin{align}
\theta(1-\theta) z_k
&= (1-\theta)\bigl((1+\theta) y_k - x_k\bigr) \nonumber\\
&= (1-\theta^2) y_k - (1-\theta) x_k.
\label{eq:theta1-theta-z}
\end{align}
Equivalently,
\begin{equation}
\theta(1-\theta) z_k + (\theta^2 - 1) y_k = -(1-\theta) x_k.
\label{eq:linear-comb-zy}
\end{equation}
Dividing \eqref{eq:linear-comb-zy} by \(\theta>0\) gives
\begin{equation}
(1-\theta) z_k + \Bigl(\theta - \frac{1}{\theta}\Bigr) y_k
= -\frac{1-\theta}{\theta}\,x_k.
\label{eq:comb-zy}
\end{equation}

On the other hand, from \eqref{eq:z-update-expanded-proof} we have
\begin{equation}
z_{k+1}
=
(1-\theta) z_k + \Bigl(\theta - \frac{1}{\theta}\Bigr) y_k + \frac{1}{\theta} x_{k+1}.
\label{eq:z-update-expanded-again}
\end{equation}
Substituting \eqref{eq:comb-zy} into \eqref{eq:z-update-expanded-again} yields
\begin{align}
z_{k+1}
&= -\frac{1-\theta}{\theta} x_k + \frac{1}{\theta} x_{k+1} \nonumber\\
&= x_{k+1} + \frac{1-\theta}{\theta}\bigl(x_{k+1}-x_k\bigr).
\label{eq:z-kplus1-x}
\end{align}

Using this representation of \(z_{k+1}\) in \eqref{eq:sc-full-y-proof}
at step \(k+1\), we obtain
\[
y_{k+1}
=
\frac{1}{1+\theta}x_{k+1} + \frac{\theta}{1+\theta}z_{k+1}
=
\frac{1}{1+\theta}x_{k+1}
+
\frac{\theta}{1+\theta}\left(
x_{k+1} + \frac{1-\theta}{\theta}(x_{k+1}-x_k)
\right),
\]
and therefore
\begin{align*}
y_{k+1}
&=
\left(\frac{1}{1+\theta}+\frac{\theta}{1+\theta}\right)x_{k+1}
+
\frac{1-\theta}{1+\theta}(x_{k+1}-x_k)\\
&=
x_{k+1} + \frac{1-\theta}{1+\theta}(x_{k+1}-x_k).
\end{align*}
Using $\theta=\sqrt{\mu/L}$, we can rewrite
\[
\beta
=
\frac{\sqrt{L}-\sqrt{\mu}}{\sqrt{L}+\sqrt{\mu}}
=
\frac{1-\sqrt{\mu/L}}{1+\sqrt{\mu/L}}
=
\frac{1-\theta}{1+\theta}.
\]
Hence
\[
y_{k+1}
=
x_{k+1} + \beta(x_{k+1}-x_k),
\]
and the $x$-update \eqref{eq:sc-full-x-proof} is exactly
\eqref{eq:NAG-sc-x-final} with $y_k^{\mathrm N}=y_k$.

Thus the pair \((x_k,y_k)\) generated by
\eqref{eq:sc-full-y-proof}--\eqref{eq:sc-full-z-proof} satisfies the
same recursion \eqref{eq:NAG-sc-x-final}--\eqref{eq:NAG-sc-y-final}
and the same initialization \(x_0=y_0=x_0^{\mathrm N}=y_0^{\mathrm N}\)
as the classical accelerated gradient method.
Combining the base case $k=0$ with the inductive step, we have shown that
\[
x_k^{\mathrm N}=x_k,
\qquad
y_k^{\mathrm N}=y_k
\quad\text{for all }k\ge 0.
\]
This completes the proof of the proposition.
\end{proof}

\begin{proof}[Proof of \Cref{thm:rs-nesterov-strong}]
First, we show that the parameter \(\theta\) in \Cref{alg:rs-NAG-sc}
satisfies \(\theta\in(0,1]\). Since \(\Lmat\succeq0\) and \(L=\|\Lmat\|\),
we have
\[
\Lmat\preceq LI_d.
\]
Hence \Cref{ass:matrix-smooth-common} implies
\[
f(y)\le f(x)+\ip{\nabla f(x)}{y-x}+\frac{L}{2}\norm{y-x}^2.
\]
Since \(f\) is also \(\mu\)-strongly convex by \Cref{ass:func}, it follows that
\[
\mu\le L.
\]
Moreover, by \Cref{prop:ell-omega-redundant},
\[
\omega\ell\ge 1.
\]
Therefore, using the definition of \(\theta\) in \Cref{alg:rs-NAG-sc},
\[
0<\theta^2=\frac{\mu}{L\omega\ell}\le \frac{\mu}{L}\le 1,
\]
and hence \(\theta\in(0,1]\).

Define
\[
\Phi_k\coloneqq\E\!\left[
f(x_k)-f^\star+\frac{\mu}{2}\norm{z_k-x^\star}^2
\right].
\]

Fix \(k\ge 0\), and let
\[
\F_k\coloneqq\sigma(P_0,\dots,P_{k-1}),
\qquad
g_k\coloneqq\nabla f(y_k).
\]

By \Cref{alg:rs-NAG-sc}, the iterates satisfy
\begin{align}
&\theta = \sqrt{\frac{\mu}{L\omega\ell}}, \label{eq:thm-sc-theta}\\
&y_k = \frac{1}{1+\theta}x_k + \frac{\theta}{1+\theta}z_k, \label{eq:thm-sc-y}\\
&x_{k+1} = y_k - \frac{1}{L\ell}P_kP_k^\top \nabla f(y_k), \label{eq:thm-sc-x}\\
&z_{k+1} = (1-\theta)z_k + \theta y_k - \frac{\theta}{\mu}P_kP_k^\top \nabla f(y_k). \label{eq:thm-sc-z}
\end{align}
As shown above, \(\theta\in(0,1]\). 
In particular, \(\theta>0\), so division by \(\theta\) below is valid, and
\(1-\theta\ge 0\).

By \eqref{eq:thm-sc-x} and \eqref{eq:matrix-smooth-common} applied with
\(x=y_k\) and \(y=x_{k+1}\), we get
\begin{align}
f(x_{k+1})
&\le f(y_k)+\ip{\nabla f(y_k)}{x_{k+1}-y_k}
+\frac12 (x_{k+1}-y_k)^\top \Lmat (x_{k+1}-y_k) \\
&= f(y_k) - \frac{1}{L\ell}\ip{g_k}{P_kP_k^\top g_k}
    + \frac{1}{2L^2\ell^2} g_k^\top P_kP_k^\top \Lmat P_kP_k^\top g_k.
\end{align}
Taking conditional expectation and using \eqref{eq:unbiased-common} and \eqref{eq:Lsmooth-sketch-common},
\begin{align}
\E[f(x_{k+1})\mid \F_k]
&\le f(y_k) - \frac{1}{L\ell}\norm{g_k}^2
    + \frac{1}{2L^2\ell^2}\, g_k^\top \E[P_kP_k^\top\Lmat P_kP_k^\top\mid \F_k] g_k \\
&\le f(y_k) - \frac{1}{2L\ell}\norm{g_k}^2.
\end{align}
Define
\begin{equation}
m \coloneqq \frac{1}{2L\ell}.
\label{eq:mt}
\end{equation}
Then
\begin{equation}
\E[f(x_{k+1})\mid \F_k] \le f(y_k) - m \norm{g_k}^2.
\label{eq:x-step-bound}
\end{equation}

Define
\begin{align}\label{w_def}
w_k \coloneqq (1-\theta)(z_k-x^\star) + \theta (y_k-x^\star).
\end{align}
From \eqref{eq:thm-sc-z},
\begin{align}
z_{k+1}-x^\star
&= (1-\theta)z_k+\theta y_k-\frac{\theta}{\mu}P_kP_k^\top g_k - x^\star \\
&= (1-\theta)(z_k-x^\star)+\theta(y_k-x^\star)-\frac{\theta}{\mu}P_kP_k^\top g_k \\
&= w_k - \frac{\theta}{\mu}P_kP_k^\top g_k.
\end{align}
Hence
\begin{align}
\frac{\mu}{2}\norm{z_{k+1}-x^\star}^2
&= \frac{\mu}{2}\norm{w_k}^2
   - \theta \ip{w_k}{P_kP_k^\top g_k}
   + \frac{\theta^2}{2\mu}\norm{P_kP_k^\top g_k}^2.
\end{align}
Taking conditional expectation and using \eqref{eq:unbiased-common}, \eqref{eq:second-moment-common},
\begin{align}
\E\!\left[\frac{\mu}{2}\norm{z_{k+1}-x^\star}^2\middle|\F_k\right]
&\le \frac{\mu}{2}\norm{w_k}^2
    - \theta \ip{w_k}{g_k}
    + \frac{\omega\theta^2}{2\mu}\norm{g_k}^2.
\label{eq:z-step-bound}
\end{align}

Summing \eqref{eq:x-step-bound} and \eqref{eq:z-step-bound}, and subtracting $f^\star$,
\begin{align}
&\E\!\left[f(x_{k+1})-f^\star+\frac{\mu}{2}\norm{z_{k+1}-x^\star}^2 \middle|\F_k\right] \nonumber\\
&\qquad \le
f(y_k)-f^\star + \frac{\mu}{2}\norm{w_k}^2 - \theta\ip{w_k}{g_k}
-\left(m-\frac{\omega\theta^2}{2\mu}\right)\norm{g_k}^2.
\label{eq:sum-before-cancel}
\end{align}
By \eqref{eq:thm-sc-theta} and \eqref{eq:mt},
\[
m=\frac{\omega\theta^2}{2\mu},
\]
so the $\norm{g_k}^2$ term vanishes. Thus
\begin{align}
&\E\!\left[f(x_{k+1})-f^\star+\frac{\mu}{2}\norm{z_{k+1}-x^\star}^2 \middle|\F_k\right]
\label{eq:after-cancel}\\
&\qquad \le
f(y_k)-f^\star + \frac{\mu}{2}\norm{w_k}^2 - \theta\ip{w_k}{g_k}. \nonumber
\end{align}

From \eqref{eq:thm-sc-y},
\[
(1+\theta)y_k = x_k+\theta z_k
\quad\Longrightarrow\quad
\theta(z_k-y_k)=y_k-x_k.
\]
Let
\[
a_k\coloneqq z_k-x^\star,\qquad b_k\coloneqq y_k-x^\star.
\]
Then, recalling the definition \eqref{w_def}, we have
\[
w_k=(1-\theta)a_k+\theta b_k.
\]
Also,
\begin{align}
\frac{\mu}{2}\norm{w_k}^2
&=
\frac{\mu}{2}\norm{(1-\theta)a_k+\theta b_k}^2 \\
&=
(1-\theta)\frac{\mu}{2}\norm{a_k}^2
+\theta\frac{\mu}{2}\norm{b_k}^2
-\theta(1-\theta)\frac{\mu}{2}\norm{a_k-b_k}^2.
\label{eq:w-square}
\end{align}
Also,
\begin{align}
-\theta\ip{w_k}{g_k}
&= -\theta\ip{(1-\theta)a_k+\theta b_k}{g_k} \\
&= -\theta(1-\theta)\ip{a_k}{g_k}-\theta^2\ip{b_k}{g_k} \\
&= -\theta\ip{b_k}{g_k} - \theta(1-\theta)\ip{a_k-b_k}{g_k} \\
&= -\theta\ip{b_k}{g_k} + (1-\theta)\ip{x_k-y_k}{g_k},
\label{eq:w-inner}
\end{align}
where in the last step we used $a_k-b_k=z_k-y_k=(y_k-x_k)/\theta$.

Since $f$ is $\mu$-strongly convex, for any $u,v$,
\[
f(u)\ge f(v)+\ip{\nabla f(v)}{u-v}+\frac{\mu}{2}\norm{u-v}^2.
\]
Applying this with $(u,v)=(x^\star,y_k)$ gives
\begin{equation}
\ip{g_k}{b_k}=\ip{g_k}{y_k-x^\star}
\ge f(y_k)-f^\star + \frac{\mu}{2}\norm{y_k-x^\star}^2
= f(y_k)-f^\star + \frac{\mu}{2}\norm{b_k}^2.
\label{eq:sc-1}
\end{equation}
Applying it with $(u,v)=(x_k,y_k)$ gives
\begin{equation}
\ip{g_k}{x_k-y_k}
\le f(x_k)-f(y_k)-\frac{\mu}{2}\norm{x_k-y_k}^2.
\label{eq:sc-2}
\end{equation}

Substitute \eqref{eq:w-square}, \eqref{eq:w-inner}, \eqref{eq:sc-1}, \eqref{eq:sc-2}
into \eqref{eq:after-cancel}:
\begin{align}
&\E\!\left[f(x_{k+1})-f^\star+\frac{\mu}{2}\norm{z_{k+1}-x^\star}^2 \middle|\F_k\right] \nonumber\\
&\le f(y_k)-f^\star
+(1-\theta)\frac{\mu}{2}\norm{a_k}^2
+\theta\frac{\mu}{2}\norm{b_k}^2
-\theta(1-\theta)\frac{\mu}{2}\norm{a_k-b_k}^2 \nonumber\\
&\quad
-\theta\!\left(f(y_k)-f^\star+\frac{\mu}{2}\norm{b_k}^2\right)
+(1-\theta)\!\left(f(x_k)-f(y_k)-\frac{\mu}{2}\norm{x_k-y_k}^2\right)\\
&\le f(y_k)-f^\star
+(1-\theta)\frac{\mu}{2}\norm{a_k}^2
+\theta\frac{\mu}{2}\norm{b_k}^2 \nonumber\\
&\quad
-\theta\!\left(f(y_k)-f^\star+\frac{\mu}{2}\norm{b_k}^2\right)
+(1-\theta)\!\left(f(x_k)-f(y_k)\right).
\end{align}
The coefficients of \(f(y_k)\) and \(\norm{b_k}^2\) cancel.
Using also \(a_k=z_k-x^\star\), we obtain
\begin{align}
&\E\!\left[f(x_{k+1})-f^\star+\frac{\mu}{2}\norm{z_{k+1}-x^\star}^2 \middle|\F_k\right] \nonumber\\
&\le
(1-\theta)\bigl(f(x_k)-f^\star\bigr)
+(1-\theta)\frac{\mu}{2}\norm{z_k-x^\star}^2 \nonumber\\
&= (1-\theta)\left(f(x_k)-f^\star+\frac{\mu}{2}\norm{z_k-x^\star}^2\right).
\label{eq:conditional-contraction}
\end{align}

By nonnegativity, the tower property applies to
\eqref{eq:conditional-contraction}; hence, taking expectation gives
\begin{equation}
\Phi_{k+1}\le (1-\theta)\Phi_k.
\label{eq:phi-contraction}
\end{equation}
Iterating yields
\[
\Phi_N\le (1-\theta)^N\Phi_0.
\]
Moreover,
\[
\E[f(x_N)-f^\star]\le \Phi_N.
\]
By \(\mu\)-strong convexity,
\[
f(x_0)-f^\star
\ge
\frac{\mu}{2}\norm{x_0-x^\star}^2.
\]
Since \(z_0=x_0\), this implies
\[
\Phi_0
=
f(x_0)-f^\star+\frac{\mu}{2}\norm{z_0-x^\star}^2
\le
2\bigl(f(x_0)-f^\star\bigr)
=
2\Delta_0.
\]
Therefore,
\[
\E[f(x_N)-f^\star]
\le \Phi_N
\le (1-\theta)^N\Phi_0
\le 2(1-\theta)^N\Delta_0
\]
which is \eqref{eq:main-rate}.

It remains to derive the complexity bound. Since \(1-\theta\le e^{-\theta}\),
\[
\E[f(x_N)-f^\star]
\le
2e^{-N\theta}\Delta_0.
\]
Thus, to guarantee \(\E[f(x_N)-f^\star]\le\epsilon\), it suffices that
\[
N
\ge
\frac{1}{\theta}\log\!\left(\frac{2\Delta_0}{\epsilon}\right).
\]
Using the definition of \(\theta\) in \Cref{alg:rs-NAG-sc}, this is
\[
N
\ge
\sqrt{\frac{L\omega\ell}{\mu}}\,
\log\!\left(\frac{2\Delta_0}{\epsilon}\right).
\]

This completes the proof.
\end{proof}

\section{Missing Proofs for \Cref{sec:sketch-examples} }
\label{app:missing-proofs-sketch}

\begin{proof}[Proof of \Cref{prop:verify-sketch-common}]
We verify the three parts of \Cref{ass:sketch-common} for each sketch:
(i) unbiasedness \(\E[PP^\top]=I_d\),
(ii) the second-moment bound \(\E[(PP^\top)^2]\preceq \omega I_d\),
and (iii) the matrix-smoothness interaction
\(\E[PP^\top \Lmat PP^\top]\preceq \ell L I_d\).

\paragraph{Haar sketch.}
Let \(U\) be Haar-distributed on \(O(d)\), let
\(R\in\mathbb{R}^{d\times r}\) be the matrix formed by the first \(r\) columns
of \(U\), and set \(P=\sqrt{d/r}\,R\).
It is standard that \(\E[RR^\top]=\frac{r}{d}I_d\), so \(\E[PP^\top]=I_d\), verifying unbiasedness.

For the second moment, 
\[
(PP^\top)^2
=
PP^\top PP^\top
=
\frac{d}{r}\,PP^\top.
\]
Taking expectations and using \(\E[PP^\top]=I_d\) gives
\(\E[(PP^\top)^2]=\frac{d}{r}I_d\), so we may take \(\omega_{\mathrm{Haar}}=\frac{d}{r}\).

For the interaction with \(\Lmat\), write
\[
L\coloneqq\|\Lmat\|,
\qquad
r_{\mathrm{eff}}\coloneqq\frac{\tr(\Lmat)}{\|\Lmat\|}.
\]
By results in \cite{flynn2024},
\[
\E[PP^\top \Lmat PP^\top]
=
\frac{d}{r}(1-\beta)\Lmat+\frac{\beta}{r}\tr(\Lmat)\,I_d,
\]
where \(\beta=\frac{d(d-r)}{(d+2)(d-1)}\).
Using \(\Lmat\succeq 0\) and \(\|\Lmat\|=L\),
\[
\E[PP^\top \Lmat PP^\top]
\preceq
\left[\frac{d}{r}(1-\beta)+\frac{\beta}{r}\frac{\tr(\Lmat)}{L}\right]L\,I_d
=
\frac{d}{r}\left(1-\beta+\beta\frac{r_{\mathrm{eff}}}{d}\right)L\,I_d.
\]
Thus \eqref{eq:Lsmooth-sketch-common}  in \Cref{ass:sketch-common} holds with
\[
\ell_{\mathrm{Haar}}
=
\frac{d}{r}\left(1-\beta+\beta\frac{r_{\mathrm{eff}}}{d}\right).
\]

\paragraph{Coordinate sketch.}
Let \(S\in\R^{d\times r}\) consist of \(r\) distinct columns sampled uniformly from the identity
matrix \(I_d\), and let \(P\coloneqq\sqrt{\frac{d}{r}}S\).
Then \(SS^\top\) is a coordinate projection, and it is straightforward to check that
\(\E[SS^\top]=\frac{r}{d}I_d\), so \(\E[PP^\top]=I_d\).

For the second moment, 
\[
(PP^\top)^2
=
PP^\top PP^\top
=
\frac{d}{r}\,PP^\top.
\]
Taking expectations, we obtain
\(\E[(PP^\top)^2]=\frac{d}{r}I_d\), so we may take \(\omega_{\mathrm{Coord}}=\frac{d}{r}\).

For the interaction with \(\Lmat\), define
\[
\delta_{\mathrm{diag}}\coloneqq\frac{\|\diag(\Lmat)\|}{\|\Lmat\|}.
\]
By a result in \cite{flynn2024},
\[
\E[PP^\top \Lmat PP^\top]
=
\frac{d}{r}\left(\frac{r-1}{d-1}\Lmat+\frac{d-r}{d-1}\diag(\Lmat)\right).
\]
Using \(\|\Lmat\|=L\) and \(\|\diag(\Lmat)\|=\delta_{\mathrm{diag}}L\),
\[
\E[PP^\top \Lmat PP^\top]
\preceq
\frac{d}{r}\left(\frac{r-1}{d-1}L+\frac{d-r}{d-1}\delta_{\mathrm{diag}}L\right)I_d
=
\frac{d}{r}\left(\frac{r-1}{d-1}+\frac{d-r}{d-1}\delta_{\mathrm{diag}}\right)L\,I_d.
\]
Thus we may take
\[
\ell_{\mathrm{Coord}}
=
\frac{d}{r}\left(\frac{r-1}{d-1}+\frac{d-r}{d-1}\delta_{\mathrm{diag}}\right).
\]

\paragraph{Gaussian sketch.}
Finally, assume \(P\in\R^{d\times r}\) has i.i.d.\ entries
\(P_{ij}\sim \mathcal{N}(0,1/r)\).
Then \(\E[PP^\top]=I_d\) by construction.
Let \(G\coloneqq\sqrt{r}\,P\), so that \(G_{ij}\sim\mathcal{N}(0,1)\) i.i.d.\ and
\(PP^\top=\frac{1}{r}GG^\top\).

For the second moment, a result from \cite{flynn2024} implies
\[
\E[GG^\top I_d GG^\top]=r(r+1)I_d+r\,\tr(I_d)\,I_d
= r(r+1)I_d + rd\,I_d.
\]
Hence
\[
\E[(PP^\top)^2]
=
\frac{1}{r^2}\E[(GG^\top)^2]
=
\frac{1}{r^2}\bigl(r(r+1)+rd\bigr)I_d
=
\frac{d+r+1}{r}\,I_d,
\]
so 
\(\omega_{\mathrm{Gauss}}=\frac{d+r+1}{r}\).

For the interaction with \(\Lmat\), a result in \cite{flynn2024} implies
\[
\E[GG^\top \Lmat GG^\top]=r(r+1)\Lmat+r\,\tr(\Lmat)\,I_d.
\]
Recalling that \(P=G/\sqrt{r}\) and hence \(PP^\top = \frac{1}{r} GG^\top\), we obtain
\[
\E[PP^\top \Lmat PP^\top]
=
\frac{1}{r^2}\,\E[GG^\top \Lmat GG^\top]
=
\frac{r+1}{r}\,\Lmat+\frac{\tr(\Lmat)}{r}\,I_d.
\]
Writing
\[
L\coloneqq\|\Lmat\|,
\qquad
r_{\mathrm{eff}}\coloneqq\frac{\tr(\Lmat)}{\|\Lmat\|},
\]
and using \(\Lmat\succeq 0\) and \(\|\Lmat\|=L\), we obtain
\begin{align*}
\E[PP^\top \Lmat PP^\top]
&\preceq
\frac{r+1}{r}\,L I_d+\frac{r_{\mathrm{eff}}L}{r}\,I_d \\
&=
\frac{r+1+r_{\mathrm{eff}}}{r}\,L\,I_d.
\end{align*}

Thus,
\[
\ell_{\mathrm{Gauss}}
=
\frac{r+1+r_{\mathrm{eff}}}{r}.
\]

It remains to prove the stated basic bounds on \(r_{\mathrm{eff}}\) and
\(\delta_{\mathrm{diag}}\).

Let \(\lambda_1,\dots,\lambda_d\) be the eigenvalues of \(\Lmat\).
Since \(\Lmat\succeq 0\) and \(L=\|\Lmat\|\), we have
\[
0\le \lambda_i\le L
\qquad (i=1,\dots,d),
\]
and since \(\Lmat\neq 0\), at least one eigenvalue equals \(L\).
Therefore
\[
L \le \tr(\Lmat)=\sum_{i=1}^d \lambda_i \le dL,
\]
which gives
\[
1\le r_{\mathrm{eff}}=\frac{\tr(\Lmat)}{L}\le d.
\]

Next, since \(\delta_{\mathrm{diag}}=\|\diag(\Lmat)\|/L\), we have
\[
\delta_{\mathrm{diag}}
=
\frac{\max_{1\le i\le d} \Lmat_{ii}}{L}.
\]
Because \(\Lmat\succeq 0\), for each \(i\),
\[
\Lmat_{ii}=e_i^\top \Lmat e_i \le \|\Lmat\|=L,
\]
so \(\delta_{\mathrm{diag}}\le 1\).

Also,
\[
\max_{1\le i\le d}\Lmat_{ii}
\ge
\frac{1}{d}\sum_{i=1}^d \Lmat_{ii}
=
\frac{\tr(\Lmat)}{d}
=
\frac{L\,r_{\mathrm{eff}}}{d}.
\]
Dividing by \(L\) gives
\[
\delta_{\mathrm{diag}}\ge \frac{r_{\mathrm{eff}}}{d}.
\]
Since \(r_{\mathrm{eff}}\ge 1\), this also implies
\[
\delta_{\mathrm{diag}}\ge \frac1d.
\]
This completes the proof.
\end{proof}

\begin{proof}[Proof of \Cref{prop:comp-sketch-basic}]
We first prove that \(\sqrt{\omega\ell r^2}\) is minimized at \(r=1\) for each sketch.

For the Haar sketch, by \Cref{tab:sketch_constants},
\[
\sqrt{\omega_{\mathrm{Haar}}\ell_{\mathrm{Haar}}r^2}
=
d\sqrt{\,1-\beta+\beta\frac{r_{\mathrm{eff}}}{d}\,},
\qquad
\beta=\frac{d(d-r)}{(d+2)(d-1)}.
\]
Since
\[
1-\beta+\beta\frac{r_{\mathrm{eff}}}{d}
=
1-\frac{(d-r)(d-r_{\mathrm{eff}})}{(d+2)(d-1)},
\]
and \(d-r_{\mathrm{eff}}\ge 0\) by \Cref{prop:verify-sketch-common}, the quantity inside the square root is nondecreasing in \(r\). Hence
\[
\sqrt{\omega_{\mathrm{Haar}}\ell_{\mathrm{Haar}}r^2}
\]
is minimized at \(r=1\).

For the Coordinate sketch,
\[
\sqrt{\omega_{\mathrm{Coord}}\ell_{\mathrm{Coord}}r^2}
=
d\sqrt{\frac{r-1}{d-1}+\frac{d-r}{d-1}\delta_{\mathrm{diag}}}
=
d\sqrt{\delta_{\mathrm{diag}}+\frac{r-1}{d-1}(1-\delta_{\mathrm{diag}})}.
\]
Since \(\delta_{\mathrm{diag}}\le 1\) by \Cref{prop:verify-sketch-common}, the quantity inside the square root is nondecreasing in \(r\). Hence
\[
\sqrt{\omega_{\mathrm{Coord}}\ell_{\mathrm{Coord}}r^2}
\]
is minimized at \(r=1\).

For the Gaussian sketch,
\[
\sqrt{\omega_{\mathrm{Gauss}}\ell_{\mathrm{Gauss}}r^2}
=
\sqrt{(d+r+1)(r+1+r_{\mathrm{eff}})}.
\]
Both factors are increasing in \(r\), so this quantity is increasing in \(r\), and therefore it is minimized at \(r=1\).

Substituting \(r=1\) into the three expressions gives
\[
Q_{\mathrm H}
= d\sqrt{\frac{r_{\mathrm{eff}}+2}{d+2}},
\qquad
Q_{\mathrm C}
= d\sqrt{\delta_{\mathrm{diag}}},
\qquad
Q_{\mathrm G}
= \sqrt{(d+2)(r_{\mathrm{eff}}+2)}.
\]

By \Cref{prop:verify-sketch-common},
\[
1\le r_{\mathrm{eff}}\le d,
\qquad
\frac1d\le \delta_{\mathrm{diag}}\le1,
\qquad
\delta_{\mathrm{diag}}\ge \frac{r_{\mathrm{eff}}}{d}.
\]
Therefore,
\[
d\sqrt{\frac{3}{d+2}}
\le Q_{\mathrm H}\le d,
\qquad
\sqrt d\le Q_{\mathrm C}\le d,
\qquad
\sqrt{3(d+2)}\le Q_{\mathrm G}\le d+2.
\]

Next, the Haar and Gaussian factors satisfy
\[
\frac{Q_{\mathrm G}}{Q_{\mathrm H}}
=
\frac{\sqrt{(d+2)(r_{\mathrm{eff}}+2)}}
{d\sqrt{(r_{\mathrm{eff}}+2)/(d+2)}}
=
1+\frac{2}{d},
\]
and hence
\[
Q_{\mathrm G}=\left(1+\frac{2}{d}\right)Q_{\mathrm H}.
\]

For the Haar--Coordinate comparison, using
\(\delta_{\mathrm{diag}}\ge r_{\mathrm{eff}}/d\), we obtain
\[
Q_{\mathrm C}^2
=
d^2\delta_{\mathrm{diag}}
\ge
d\,r_{\mathrm{eff}}.
\]
On the other hand,
\[
Q_{\mathrm H}^2
=
d^2\frac{r_{\mathrm{eff}}+2}{d+2}
\le
d(r_{\mathrm{eff}}+2).
\]
Thus
\[
\frac{Q_{\mathrm H}^2}{Q_{\mathrm C}^2}
\le
\frac{r_{\mathrm{eff}}+2}{r_{\mathrm{eff}}}
\le 3,
\]
where the last inequality follows from \(r_{\mathrm{eff}}\ge1\). Hence
\[
Q_{\mathrm H}\le \sqrt3\,Q_{\mathrm C}.
\]

Finally, if \(\Lmat=e_1e_1^\top\), then
\[
L=1,\qquad
r_{\mathrm{eff}}=\frac{\tr(\Lmat)}{L}=1,
\qquad
\delta_{\mathrm{diag}}=\frac{\|\diag(\Lmat)\|}{L}=1.
\]
Therefore,
\[
Q_{\mathrm C}=d,
\qquad
Q_{\mathrm H}
=
d\sqrt{\frac{3}{d+2}}
=
\sqrt{\frac{3}{d+2}}\,Q_{\mathrm C}.
\]
This completes the proof.
\end{proof}

\section{Experimental Details}
\label{app:experiment-details}

\subsection{Quadratic instances}
\label{app:quadratic-details}

We describe the four quadratic instances used in \Cref{subsec:quadratic}. In all
cases,
\[
f(x)=\frac12 x^\top \Lmat x,
\qquad d=1000,
\qquad f^\star=0.
\]

\paragraph{Convex diagonal.}
Let
\[
\Lmat
=
\diag\!\left(
1,
\frac{1}{d-2},
\ldots,
\frac{1}{d-2},
0
\right).
\]
Then \(\Lmat\succeq 0\) but is not positive definite. Its eigenvalues are
\[
\left\{
1,
\frac{1}{d-2},
\ldots,
\frac{1}{d-2},
0
\right\}.
\]
Hence
\[
L=\|\Lmat\|=1,
\qquad
\tr(\Lmat)=1+(d-2)\frac{1}{d-2}=2,
\qquad
r_{\mathrm{eff}}=\frac{\tr(\Lmat)}{\|\Lmat\|}=2.
\]
Moreover, since the maximum diagonal entry of \(\Lmat\) is \(1\),
\[
\delta_{\mathrm{diag}}
=
\frac{\|\diag(\Lmat)\|}{\|\Lmat\|}
=
1.
\]

\paragraph{Convex dense.}
Let
\[
\mathbf{1}\coloneqq(1,\dots,1)^\top\in\R^d,
\qquad
u\coloneqq\frac{1}{\sqrt d}(1,-1,1,-1,\dots,1,-1)^\top\in\R^d.
\]
Since \(d\) is even, \(u^\top \mathbf{1}=0\). Define
\[
\Lmat
=
\frac{1}{d-2}I_d
+
\left(1-\frac{1}{d-2}\right)uu^\top
-
\frac{1}{d(d-2)}\,\mathbf{1}\mathbf{1}^\top.
\]
Because \(u\) and \(\mathbf{1}\) are orthogonal, the eigenspaces split as
follows. Along \(u\), the eigenvalue is
\[
\frac{1}{d-2}+\left(1-\frac{1}{d-2}\right)=1.
\]
Along \(\mathbf{1}\), the eigenvalue is
\[
\frac{1}{d-2}-\frac{1}{d(d-2)}\,d=0.
\]
On the orthogonal complement of \(\mathrm{span}\{u,\mathbf{1}\}\), the
eigenvalue is
\[
\frac{1}{d-2}.
\]
Thus the eigenvalues are
\[
\left\{
1,
\frac{1}{d-2},
\ldots,
\frac{1}{d-2},
0
\right\},
\]
and therefore
\[
L=\|\Lmat\|=1,
\qquad
\tr(\Lmat)=1+(d-2)\frac{1}{d-2}=2,
\qquad
r_{\mathrm{eff}}=\frac{\tr(\Lmat)}{\|\Lmat\|}=2.
\]
Each diagonal entry equals
\[
\frac{1}{d-2}
+
\frac{1}{d}\left(1-\frac{1}{d-2}\right)
-
\frac{1}{d(d-2)}
=
\frac{2}{d},
\]
and hence
\[
\|\diag(\Lmat)\|=\frac{2}{d},
\qquad
\delta_{\mathrm{diag}}
=
\frac{\|\diag(\Lmat)\|}{\|\Lmat\|}
=
\frac{2}{d}.
\]

\paragraph{Strongly convex diagonal.}
Let
\[
\Lmat
=
\diag\!\left(
1,
\frac{1}{d-1},
\dots,
\frac{1}{d-1}
\right).
\]
Then the eigenvalues of \(\Lmat\) are
\[
\left\{
1,
\frac{1}{d-1},
\ldots,
\frac{1}{d-1}
\right\}.
\]
Therefore,
\[
\mu=\lambda_{\min}(\Lmat)=\frac{1}{d-1},
\qquad
L=\|\Lmat\|=\lambda_{\max}(\Lmat)=1.
\]
Moreover,
\[
\tr(\Lmat)=1+(d-1)\frac{1}{d-1}=2,
\qquad
r_{\mathrm{eff}}=\frac{\tr(\Lmat)}{\|\Lmat\|}=2.
\]
Since the maximum diagonal entry of \(\Lmat\) is \(1\), we have
\[
\delta_{\mathrm{diag}}
=
\frac{\|\diag(\Lmat)\|}{\|\Lmat\|}
=
1.
\]

\paragraph{Strongly convex dense.}
Let
\[
\Lmat
=
\frac{1}{d-1}I_d
+
\frac{d-2}{d(d-1)}\,\mathbf{1}\mathbf{1}^\top.
\]
The vector \(\mathbf{1}\) is an eigenvector with eigenvalue
\[
\frac{1}{d-1}
+
\frac{d-2}{d(d-1)}\cdot d
=
1,
\]
while every vector orthogonal to \(\mathbf{1}\) is an eigenvector with eigenvalue
\[
\frac{1}{d-1}.
\]
Therefore, the eigenvalues of \(\Lmat\) are
\[
\left\{
1,
\frac{1}{d-1},
\ldots,
\frac{1}{d-1}
\right\}.
\]
Hence
\[
\mu=\lambda_{\min}(\Lmat)=\frac{1}{d-1},
\qquad
L=\|\Lmat\|=\lambda_{\max}(\Lmat)=1.
\]
Moreover,
\[
\tr(\Lmat)
=
1+(d-1)\frac{1}{d-1}
=
2,
\qquad
r_{\mathrm{eff}}
=
\frac{\tr(\Lmat)}{\|\Lmat\|}
=
2.
\]
Finally, each diagonal entry equals
\[
\frac{1}{d-1}
+
\frac{d-2}{d(d-1)}
=
\frac{2}{d},
\]
and hence
\[
\|\diag(\Lmat)\|=\frac{2}{d},
\qquad
\delta_{\mathrm{diag}}
=
\frac{\|\diag(\Lmat)\|}{\|\Lmat\|}
=
\frac{2}{d}.
\]

\subsection{Effect of the sketch dimension}
\label{app:exp-details-rscan}

In the main quadratic experiments, we fixed the sketch dimension to \(r=1\).
Here we additionally examine how the sketch dimension affects oracle-axis convergence.
More precisely, for each of the four quadratic instances and for each sketch family
(Haar, Block-coordinate, and Gaussian), we compare the proposed method with
\(r\in\{1,10,100\}\) against the corresponding full-dimensional accelerated method
under the same oracle budget \(10{,}000\).
All other settings are unchanged: \(d=1000\), \(10\) random seeds, and
independent Gaussian initialization \(x_0\sim\mathcal N(0,I_d)\).

The purpose of this experiment is to examine whether the theoretical prediction
that \(r=1\) is oracle-optimal is also reflected numerically in the quadratic examples.
For convenience, we recall the four quadratic instances and their associated quantities
\((r_{\mathrm{eff}},\delta_{\mathrm{diag}})\):

\begin{itemize}
    \item \textbf{Convex diagonal:}
    \[
    \Lmat
    =
    \diag\!\left(
    1,\frac{1}{d-2},\ldots,\frac{1}{d-2},0
    \right),
    \qquad
    r_{\mathrm{eff}}=2,
    \qquad
    \delta_{\mathrm{diag}}=1.
    \]

    \item \textbf{Convex dense:}
    \[
    \Lmat
    =
    \frac{1}{d-2}I_d
    +
    \left(1-\frac{1}{d-2}\right)uu^\top
    -
    \frac{1}{d(d-2)}\,\mathbf{1}\mathbf{1}^\top,
    \qquad
    r_{\mathrm{eff}}=2,
    \qquad
    \delta_{\mathrm{diag}}=\frac{2}{d}.
    \]

    \item \textbf{Strongly convex diagonal:}
    \[
    \Lmat
    =
    \diag\!\left(
    1,\frac{1}{d-1},\ldots,\frac{1}{d-1}
    \right),
    \qquad
    r_{\mathrm{eff}}=2,
    \qquad
    \delta_{\mathrm{diag}}=1.
    \]

    \item \textbf{Strongly convex dense:}
    \[
    \Lmat
    =
    \frac{1}{d-1}I_d
    +
    \frac{d-2}{d(d-1)}\,\mathbf{1}\mathbf{1}^\top,
    \qquad
    r_{\mathrm{eff}}=2,
    \qquad
    \delta_{\mathrm{diag}}=\frac{2}{d}.
    \]
\end{itemize}

\begin{figure}[p]
\centering

\begin{subfigure}[t]{0.32\linewidth}
    \centering
    \includegraphics[width=\linewidth]{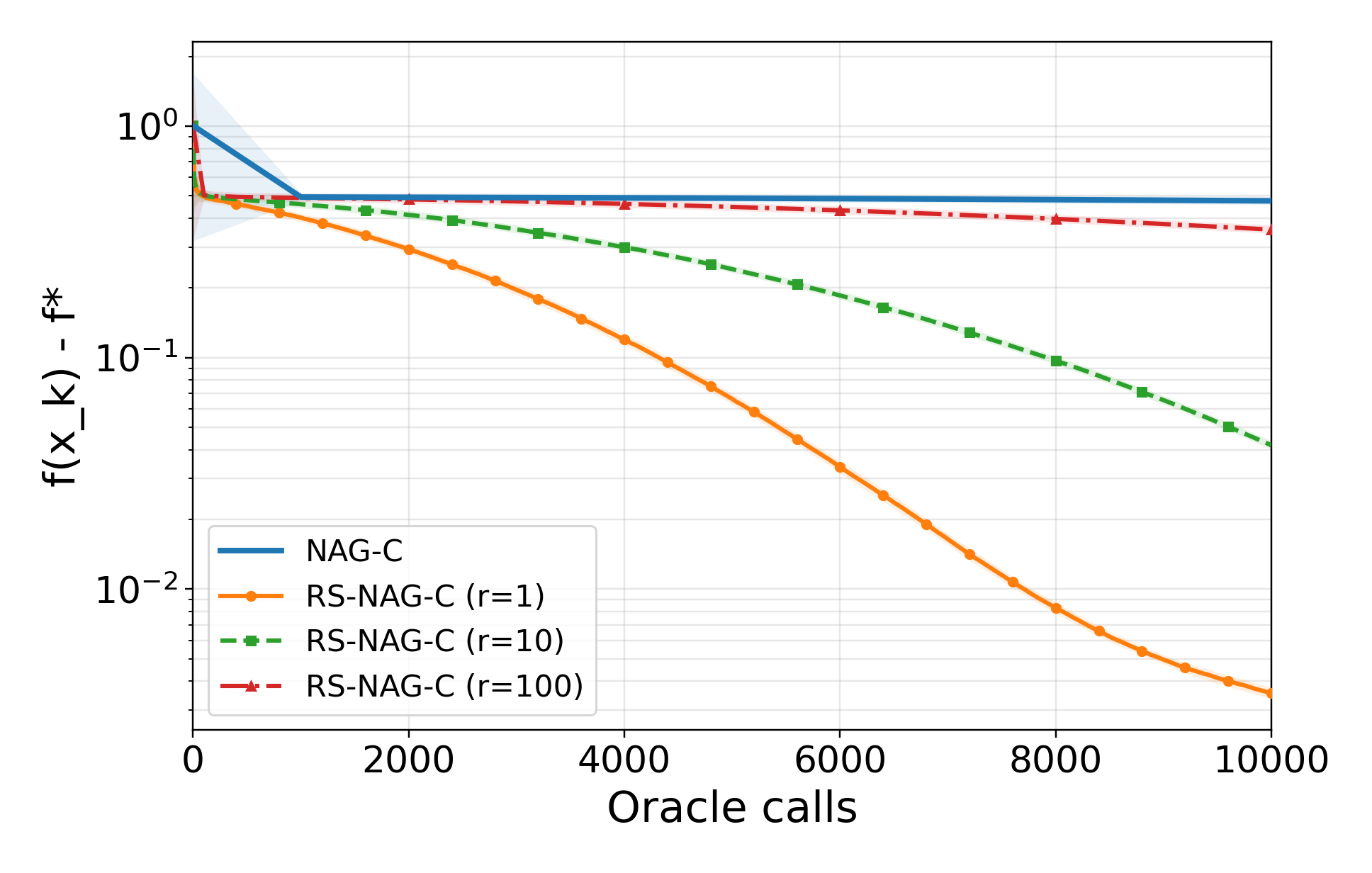}
    \caption{Convex diagonal / Haar}
    \label{fig:rscan-convex-diag-haar}
\end{subfigure}\hfill
\begin{subfigure}[t]{0.32\linewidth}
    \centering
    \includegraphics[width=\linewidth]{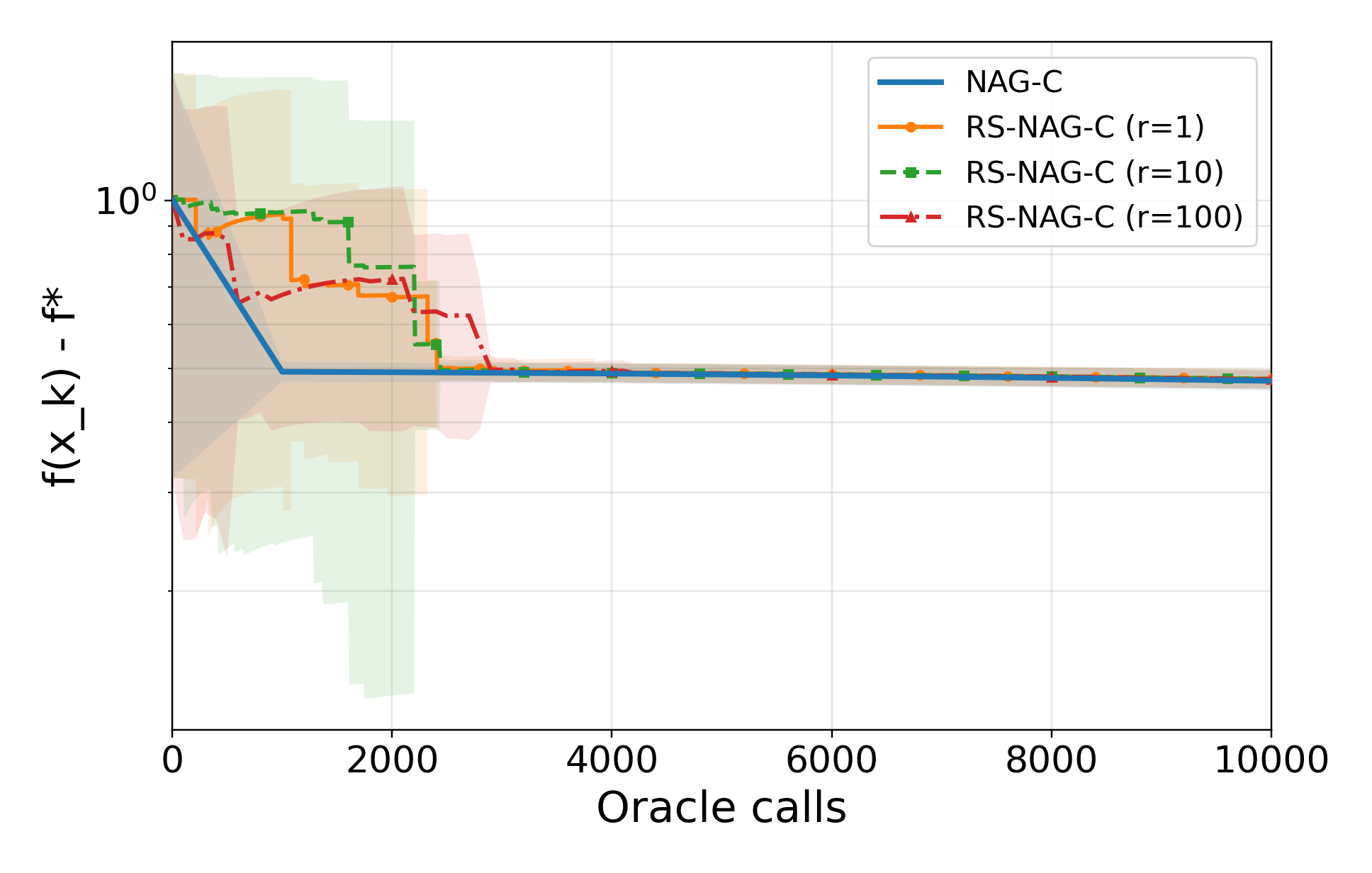}
    \caption{Convex diagonal / Block-coordinate}
    \label{fig:rscan-convex-diag-coord}
\end{subfigure}\hfill
\begin{subfigure}[t]{0.32\linewidth}
    \centering
    \includegraphics[width=\linewidth]{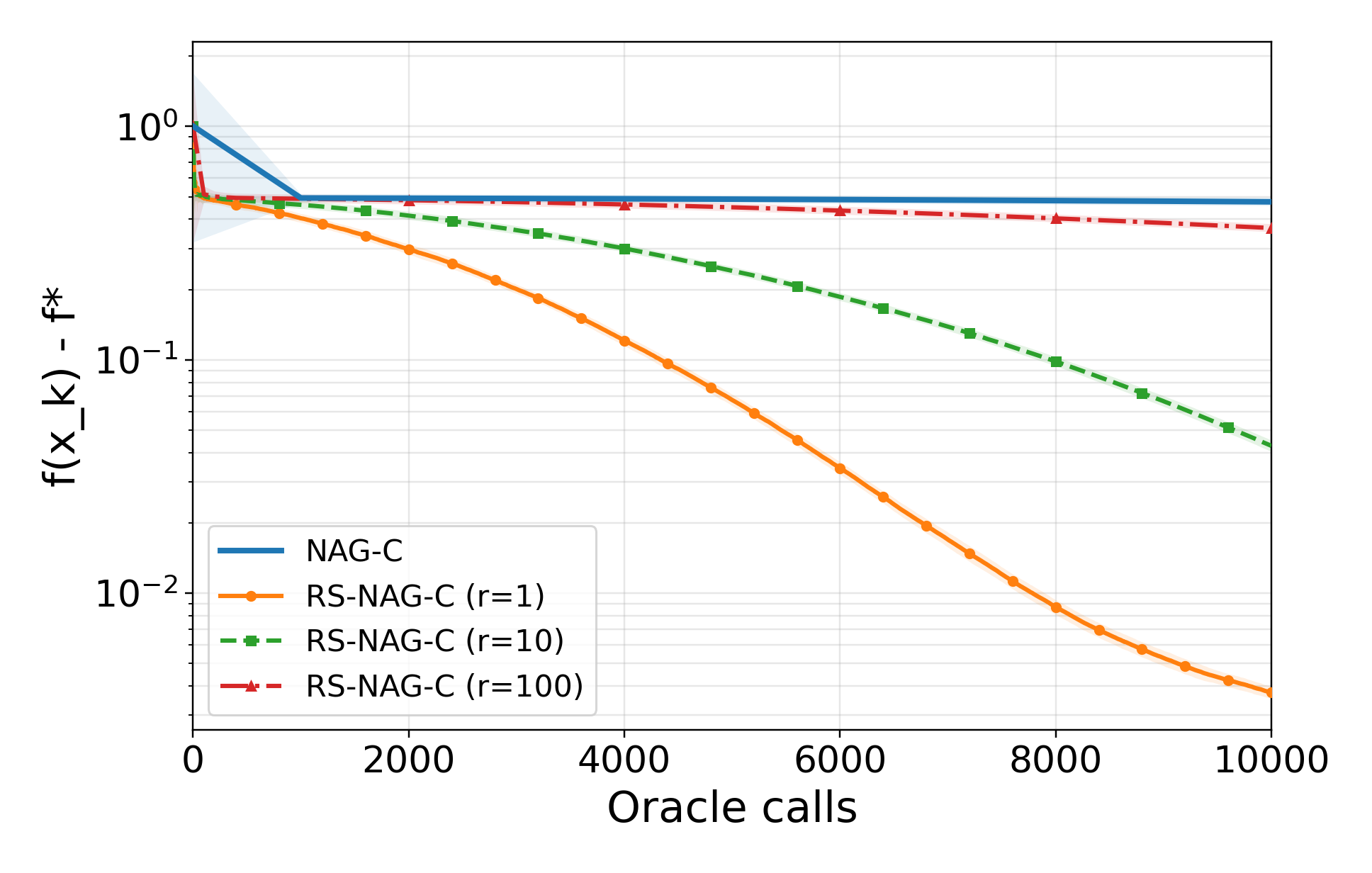}
    \caption{Convex diagonal / Gaussian}
    \label{fig:rscan-convex-diag-gauss}
\end{subfigure}

\vspace{0.5em}

\begin{subfigure}[t]{0.32\linewidth}
    \centering
    \includegraphics[width=\linewidth]{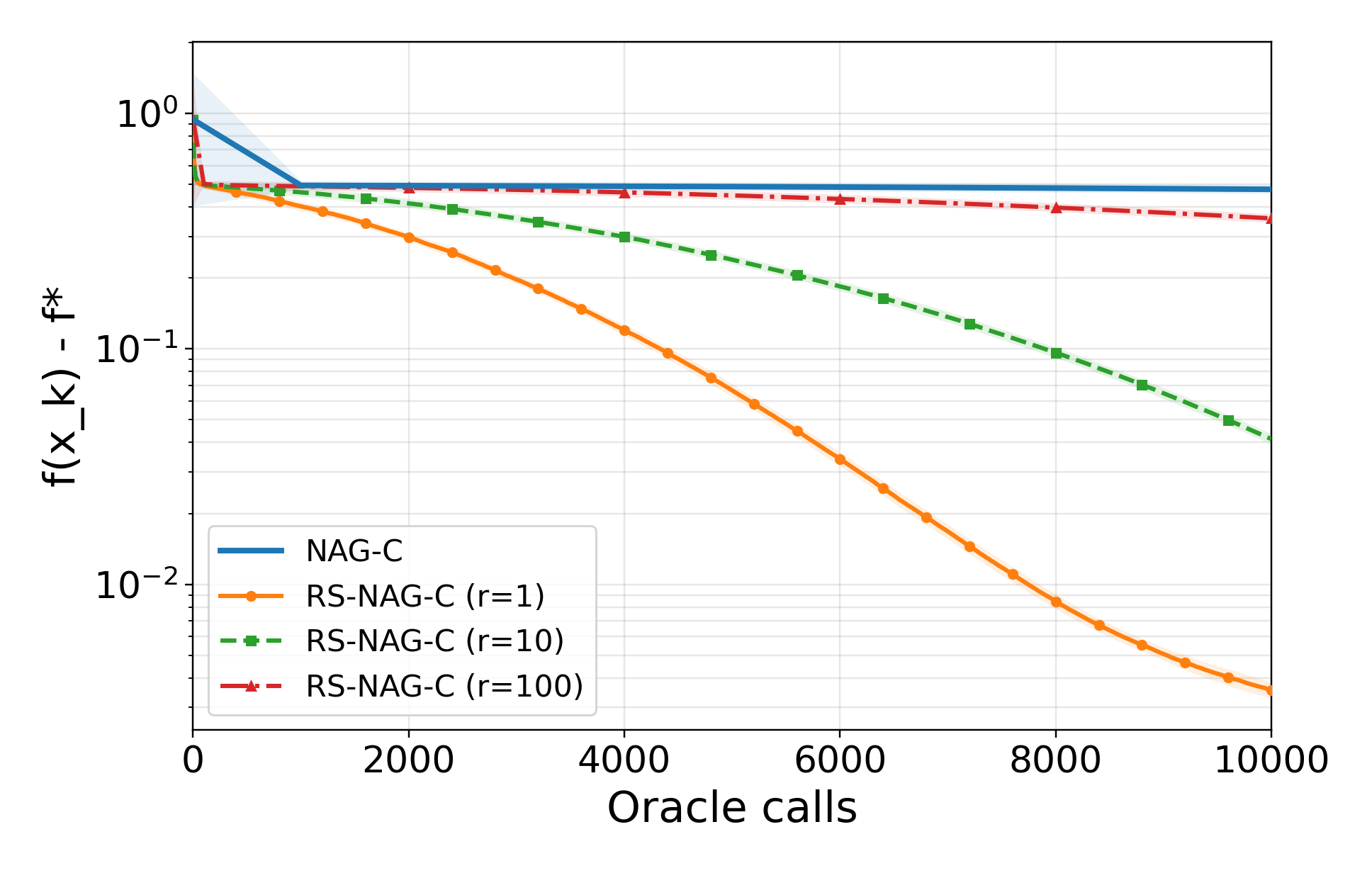}
    \caption{Convex dense / Haar}
    \label{fig:rscan-convex-dense-haar}
\end{subfigure}\hfill
\begin{subfigure}[t]{0.32\linewidth}
    \centering
    \includegraphics[width=\linewidth]{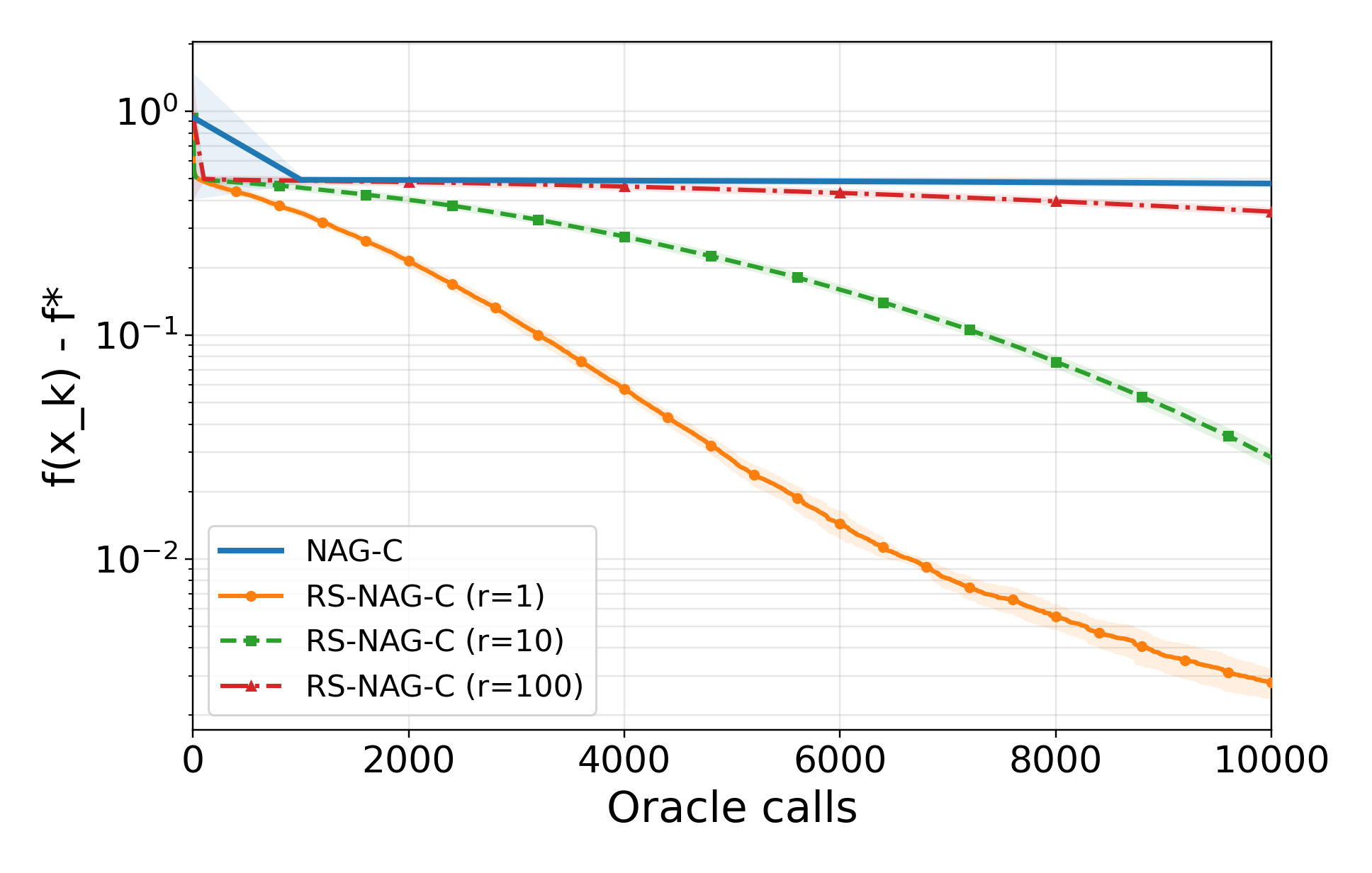}
    \caption{Convex dense / Block-coordinate}
    \label{fig:rscan-convex-dense-coord}
\end{subfigure}\hfill
\begin{subfigure}[t]{0.32\linewidth}
    \centering
    \includegraphics[width=\linewidth]{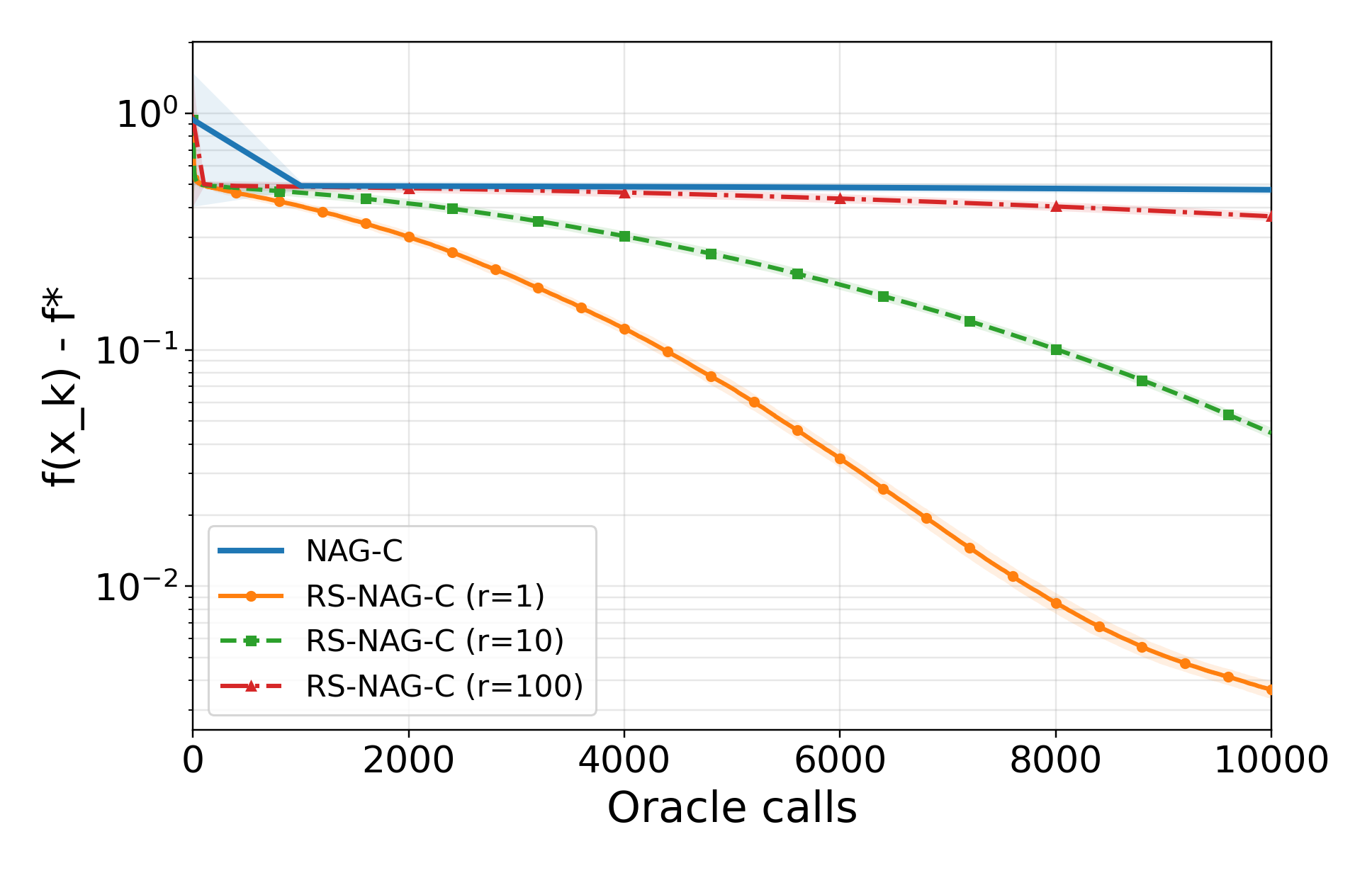}
    \caption{Convex dense / Gaussian}
    \label{fig:rscan-convex-dense-gauss}
\end{subfigure}

\vspace{0.5em}

\begin{subfigure}[t]{0.32\linewidth}
    \centering
    \includegraphics[width=\linewidth]{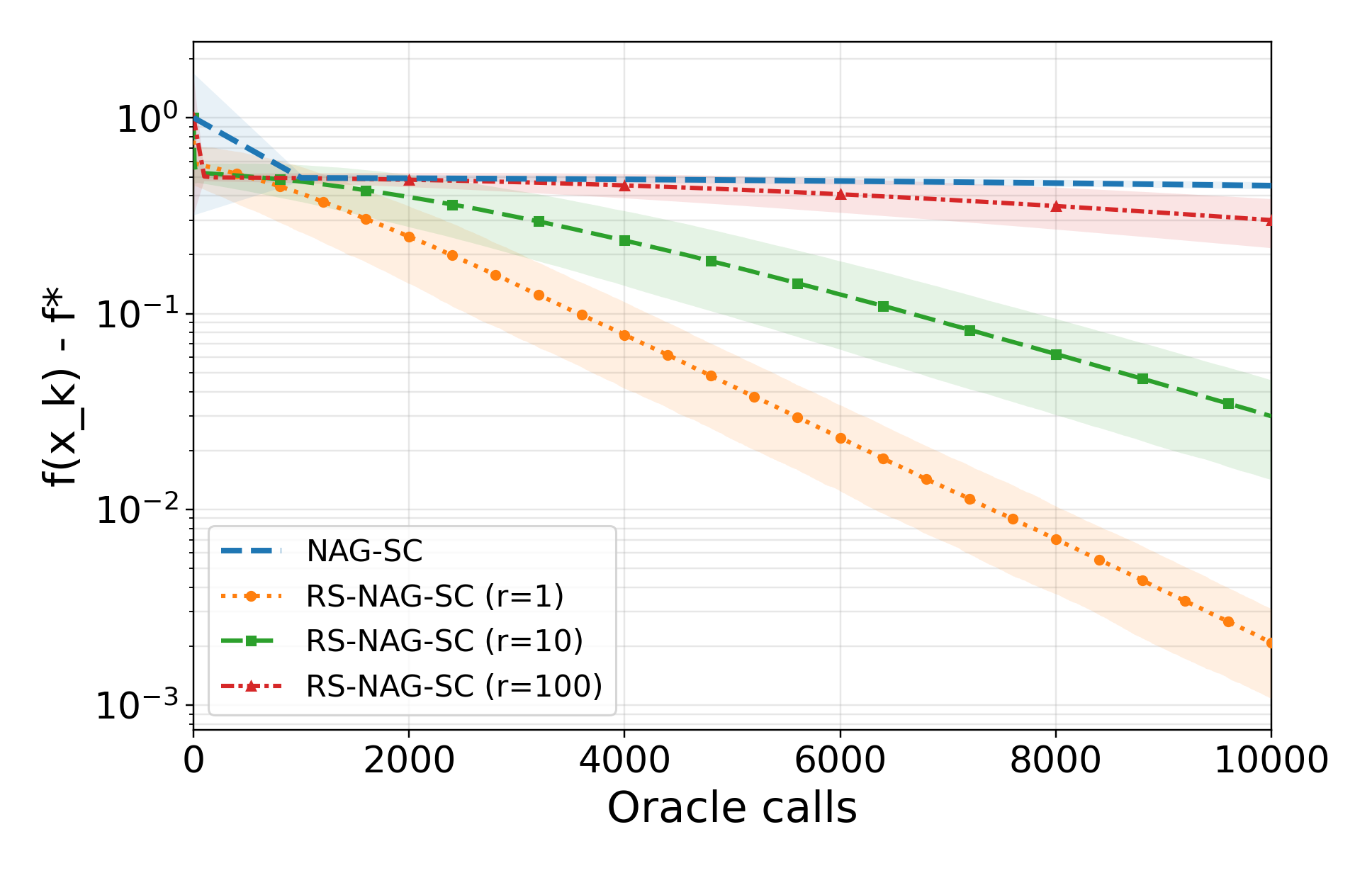}
    \caption{Strongly convex diagonal / Haar}
    \label{fig:rscan-sc-diag-haar}
\end{subfigure}\hfill
\begin{subfigure}[t]{0.32\linewidth}
    \centering
    \includegraphics[width=\linewidth]{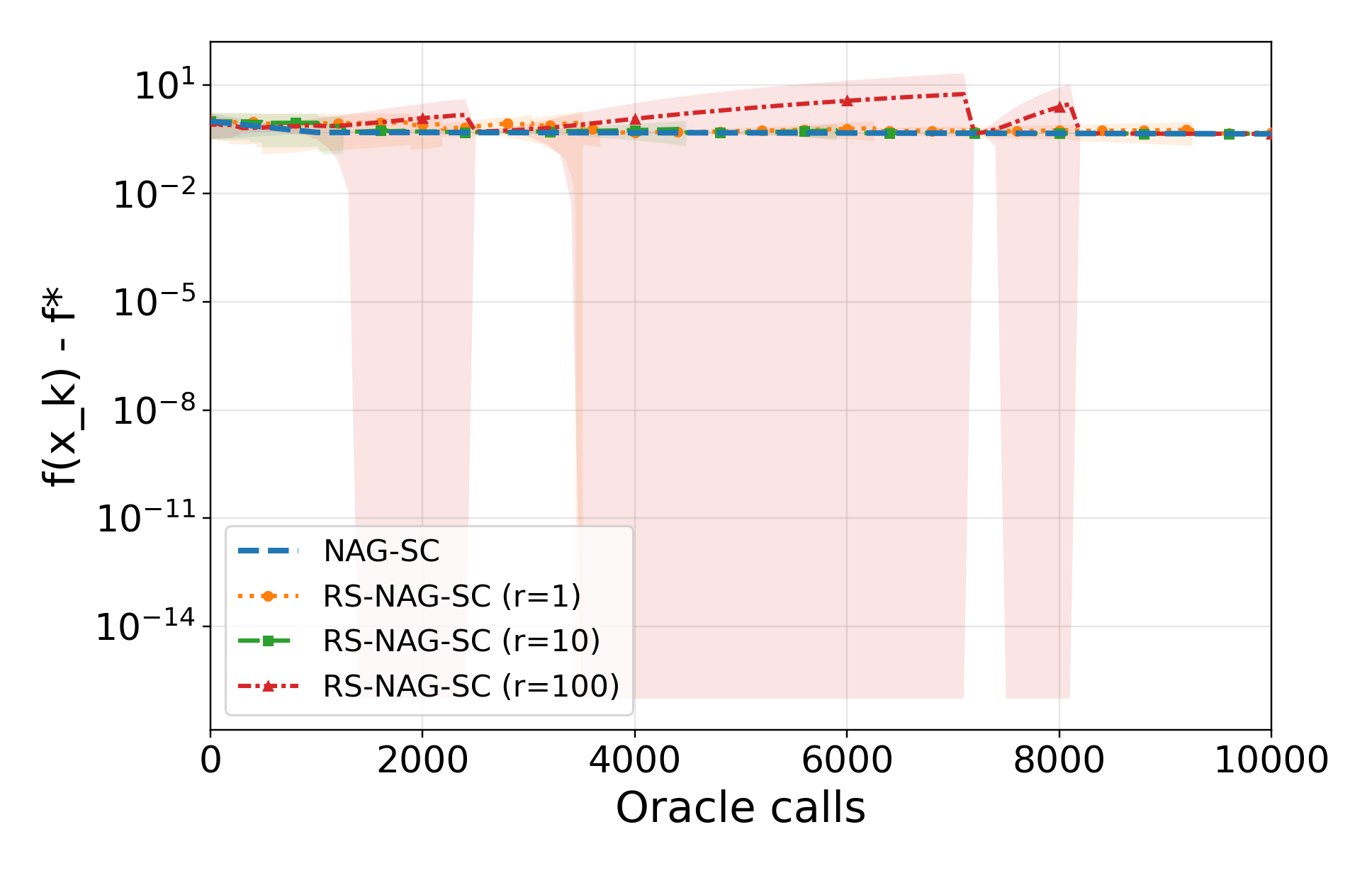}
    \caption{Strongly convex diagonal / Block-coordinate}
    \label{fig:rscan-sc-diag-coord}
\end{subfigure}\hfill
\begin{subfigure}[t]{0.32\linewidth}
    \centering
    \includegraphics[width=\linewidth]{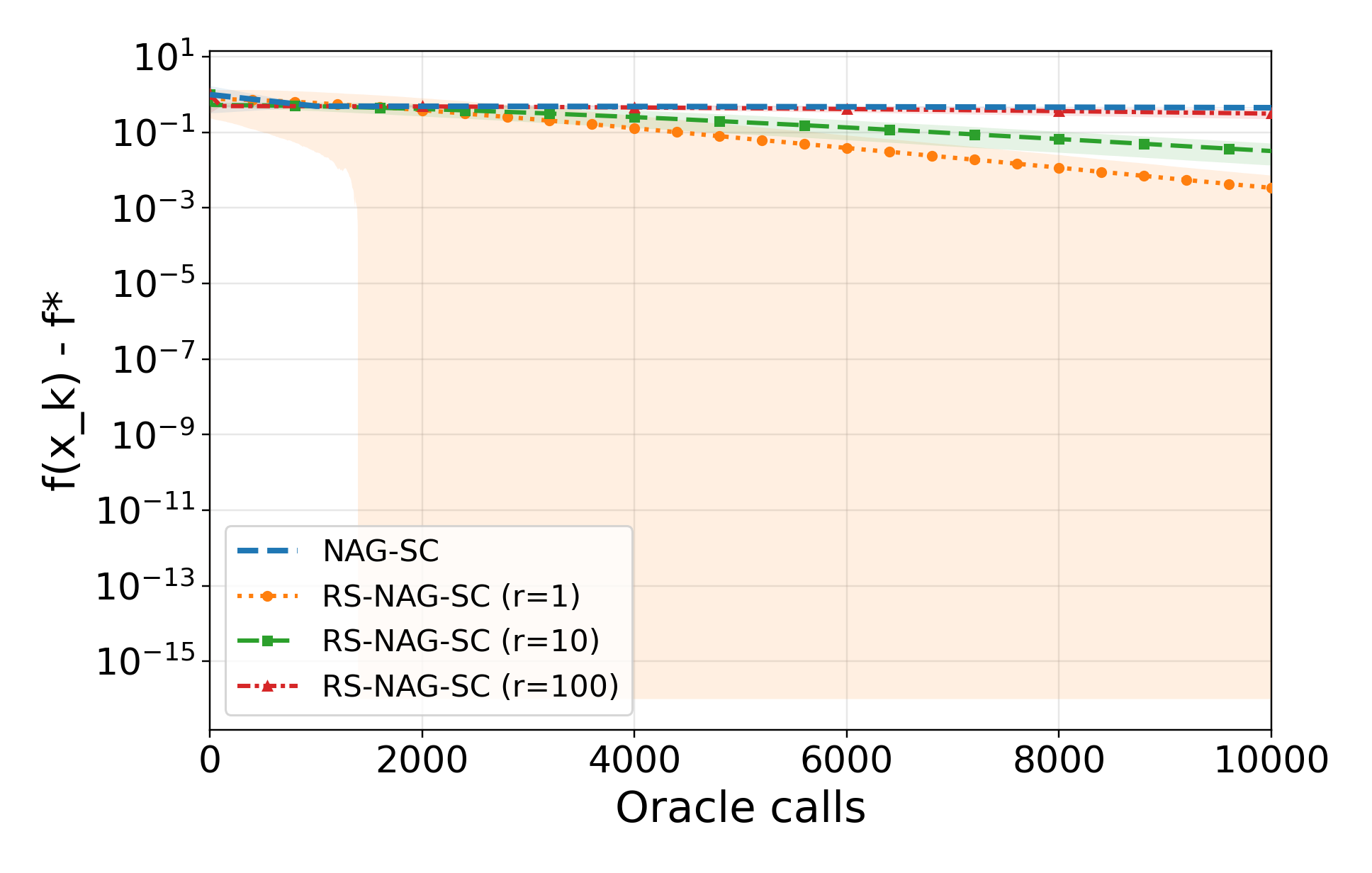}
    \caption{Strongly convex diagonal / Gaussian}
    \label{fig:rscan-sc-diag-gauss}
\end{subfigure}

\vspace{0.5em}

\begin{subfigure}[t]{0.32\linewidth}
    \centering
    \includegraphics[width=\linewidth]{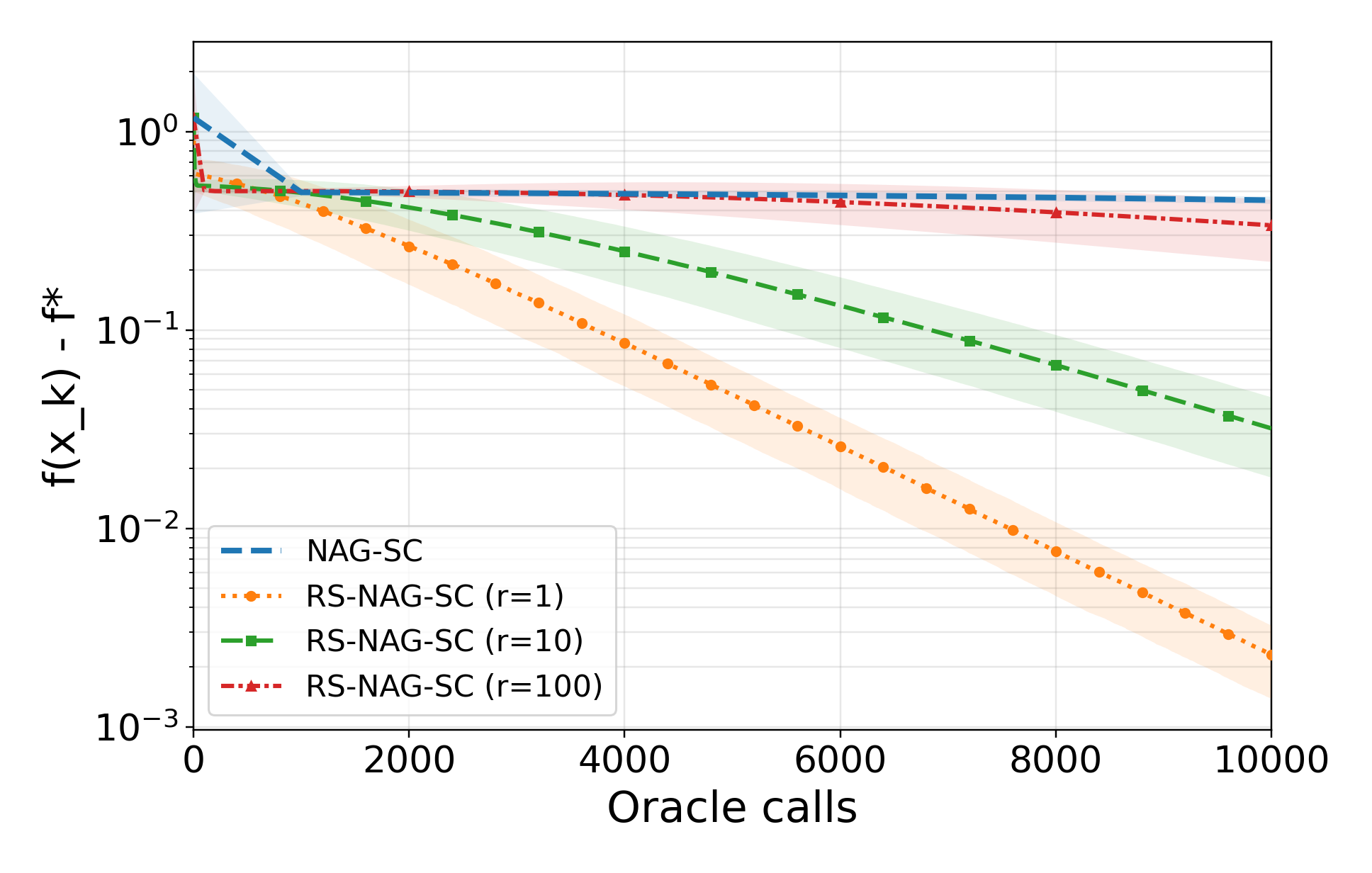}
    \caption{Strongly convex dense / Haar}
    \label{fig:rscan-sc-dense-haar}
\end{subfigure}\hfill
\begin{subfigure}[t]{0.32\linewidth}
    \centering
    \includegraphics[width=\linewidth]{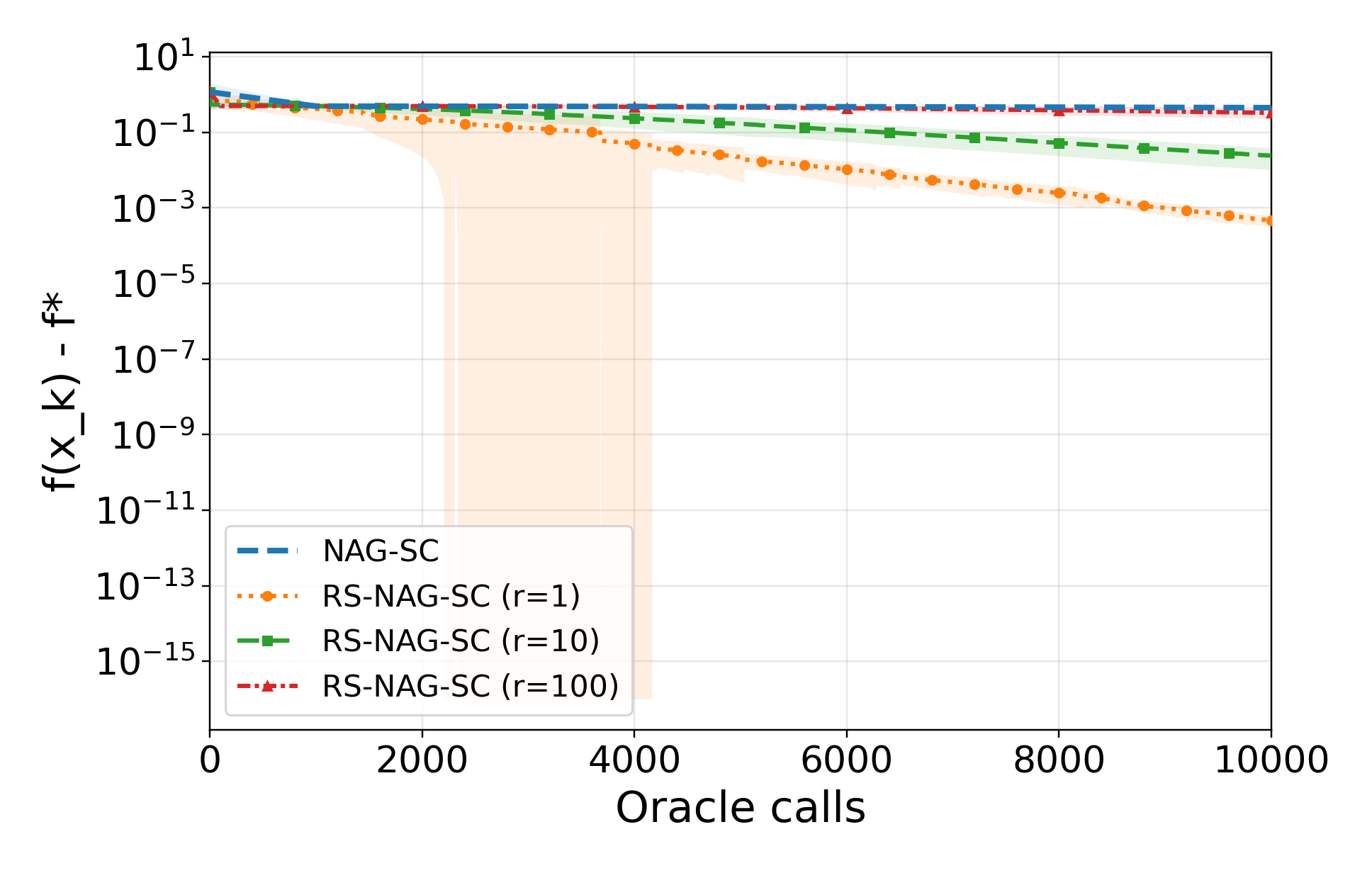}
    \caption{Strongly convex dense / Block-coordinate}
    \label{fig:rscan-sc-dense-coord}
\end{subfigure}\hfill
\begin{subfigure}[t]{0.32\linewidth}
    \centering
    \includegraphics[width=\linewidth]{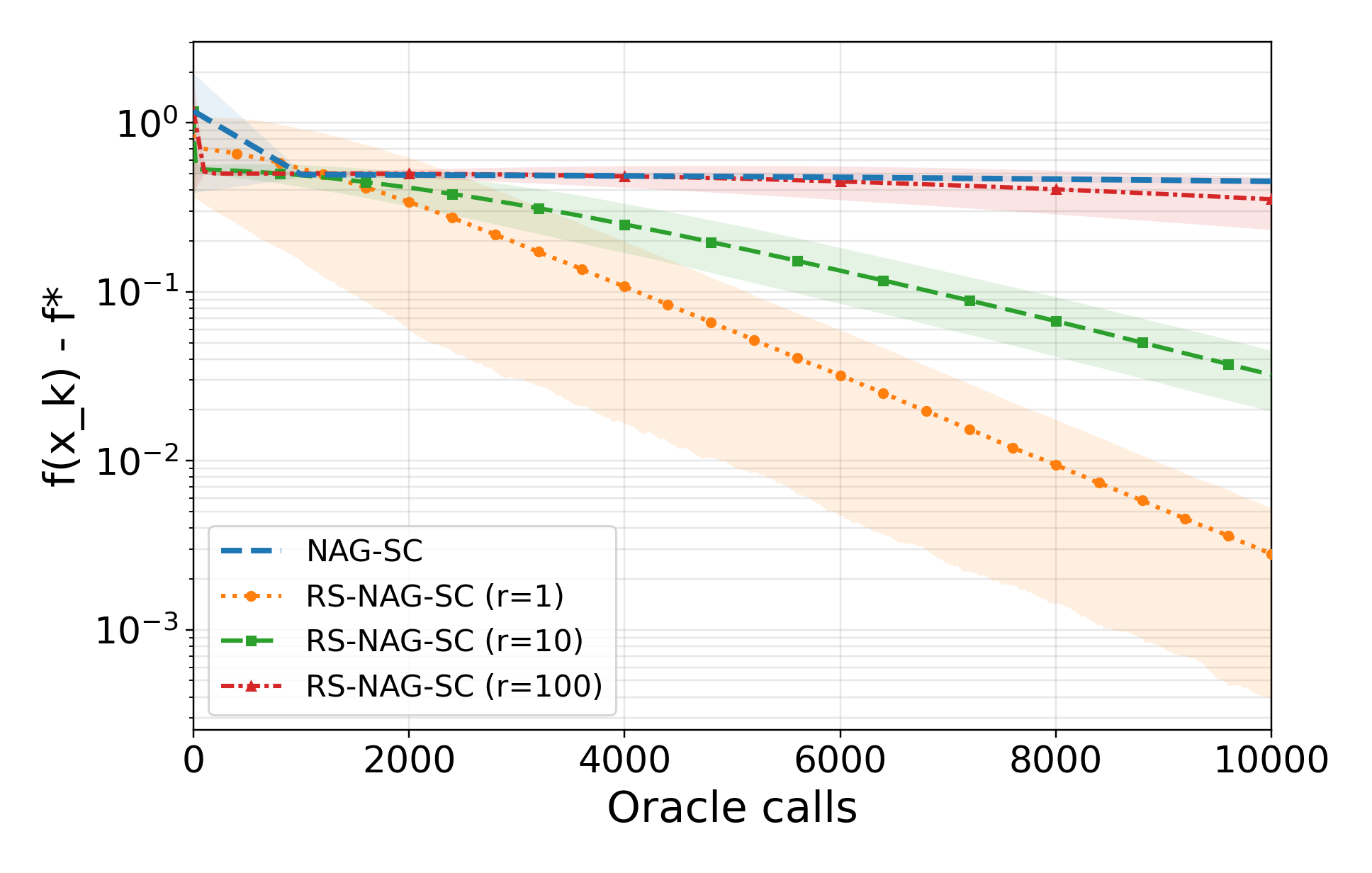}
    \caption{Strongly convex dense / Gaussian}
    \label{fig:rscan-sc-dense-gauss}
\end{subfigure}

\caption{Oracle-axis convergence in the sketch-dimension scan.
The horizontal axis shows the cumulative number of oracle calls, and the vertical axis shows the objective gap \(f(x_k)-f^\star\) on a logarithmic scale.
Each row corresponds to one quadratic instance, and each column corresponds to one sketch family.
Each panel compares RS-NAG with \(r\in\{1,10,100\}\) and the corresponding full-dimensional accelerated method.}
\label{fig:quadratic-rscan}
\end{figure}

The results are shown in \Cref{fig:quadratic-rscan}.
As summarized in \Cref{tab:sketch_constants}, the oracle-complexity comparison
is governed by the quantity \(\sqrt{\omega\ell r^2}\), and \Cref{prop:comp-sketch-basic}
shows that this quantity is minimized at \(r=1\) for the Haar, Block-coordinate,
and Gaussian sketches.

The plots are broadly consistent with this prediction.
In all Haar and Gaussian panels, and also in the Block-coordinate panels for the two dense
instances, the oracle-axis performance deteriorates as \(r\) increases from \(1\) to \(10\)
and \(100\).
The only clear exceptions are the Block-coordinate panels for the convex diagonal and
strongly convex diagonal instances; see
\Cref{fig:rscan-convex-diag-coord,fig:rscan-sc-diag-coord}.

This behavior is also explained by \Cref{tab:sketch_constants}.
For the Block-coordinate sketch,
\[
\sqrt{\omega_{\mathrm{Coord}}\ell_{\mathrm{Coord}}r^2}
=
d\sqrt{\frac{r-1}{d-1}+\frac{d-r}{d-1}\delta_{\mathrm{diag}}}.
\]
For the two diagonal instances considered here, we have \(\delta_{\mathrm{diag}}=1\), and hence
\[
\sqrt{\omega_{\mathrm{Coord}}\ell_{\mathrm{Coord}}r^2}
=
d\sqrt{\frac{r-1}{d-1}+\frac{d-r}{d-1}}
=
d,
\]
independently of \(r\).
Therefore, in these two cases, the theory itself does not predict any oracle-complexity
improvement from taking a smaller sketch dimension, which explains why the separation
among \(r=1,10,100\) is weak in
\Cref{fig:rscan-convex-diag-coord,fig:rscan-sc-diag-coord}.

Overall, except for these two Block-coordinate diagonal panels, the numerical results are
well aligned with the theoretical prediction that \(r=1\) is oracle-optimal in the present
quadratic examples.

\subsection{Matrix smoothness for logistic regression}
\label{app:logreg-smoothness}

We derive the matrix smoothness bound used in \Cref{subsec:logreg}. Recall the
\(\ell_2\)-regularized logistic regression objective
\[
f(x)
=
\frac{1}{n}\sum_{i=1}^n \log\!\bigl(1+\exp(-y_i a_i^\top x)\bigr)
+\frac{\mu}{2}\|x\|_2^2,
\qquad \mu>0,
\]
where \(a_i\in\R^d\) and \(y_i\in\{-1,+1\}\). Its gradient is
\[
\nabla f(x)
=
-\frac{1}{n}\sum_{i=1}^n
\frac{y_i a_i}{1+\exp(y_i a_i^\top x)}
+\mu x.
\]
Let \(\sigma(t)=1/(1+e^{-t})\). The Hessian is
\[
\nabla^2 f(x)
=
\frac{1}{n}\sum_{i=1}^n
\sigma(y_i a_i^\top x)\bigl(1-\sigma(y_i a_i^\top x)\bigr)\,a_i a_i^\top
+\mu I_d.
\]
Since
\[
0\le \sigma(t)(1-\sigma(t))\le \frac14
\qquad\text{for all }t\in\R,
\]
we have
\[
\nabla^2 f(x)
\preceq
\frac{1}{4n}\sum_{i=1}^n a_i a_i^\top+\mu I_d.
\]
If \(A\in\R^{n\times d}\) denotes the data matrix whose \(i\)th row is
\(a_i^\top\), then \(\sum_{i=1}^n a_i a_i^\top=A^\top A\). Thus we may take
\[
\Lmat
=
\frac{1}{4n}A^\top A+\mu I_d.
\]
This gives the matrix smoothness bound used in the logistic-regression
experiments.

\subsection{Additional logistic-regression results on standard benchmarks}
\label{app:logreg-additional-datasets}

\begin{figure*}[t]
\centering
\begin{subfigure}[t]{0.32\textwidth}
    \centering
    \includegraphics[width=\linewidth]{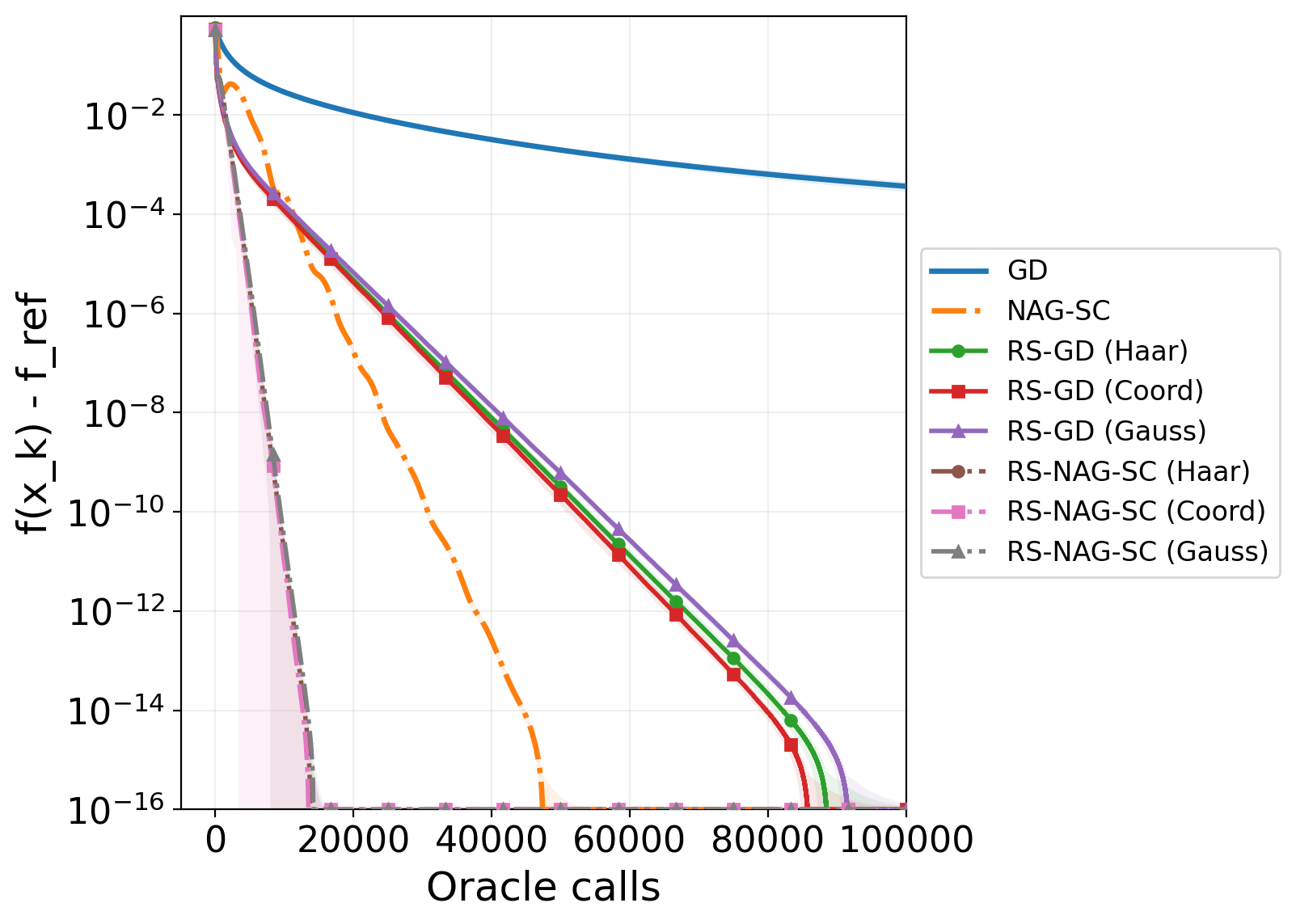}
    \caption{\text{phishing}}
\end{subfigure}
\hfill
\begin{subfigure}[t]{0.32\textwidth}
    \centering
    \includegraphics[width=\linewidth]{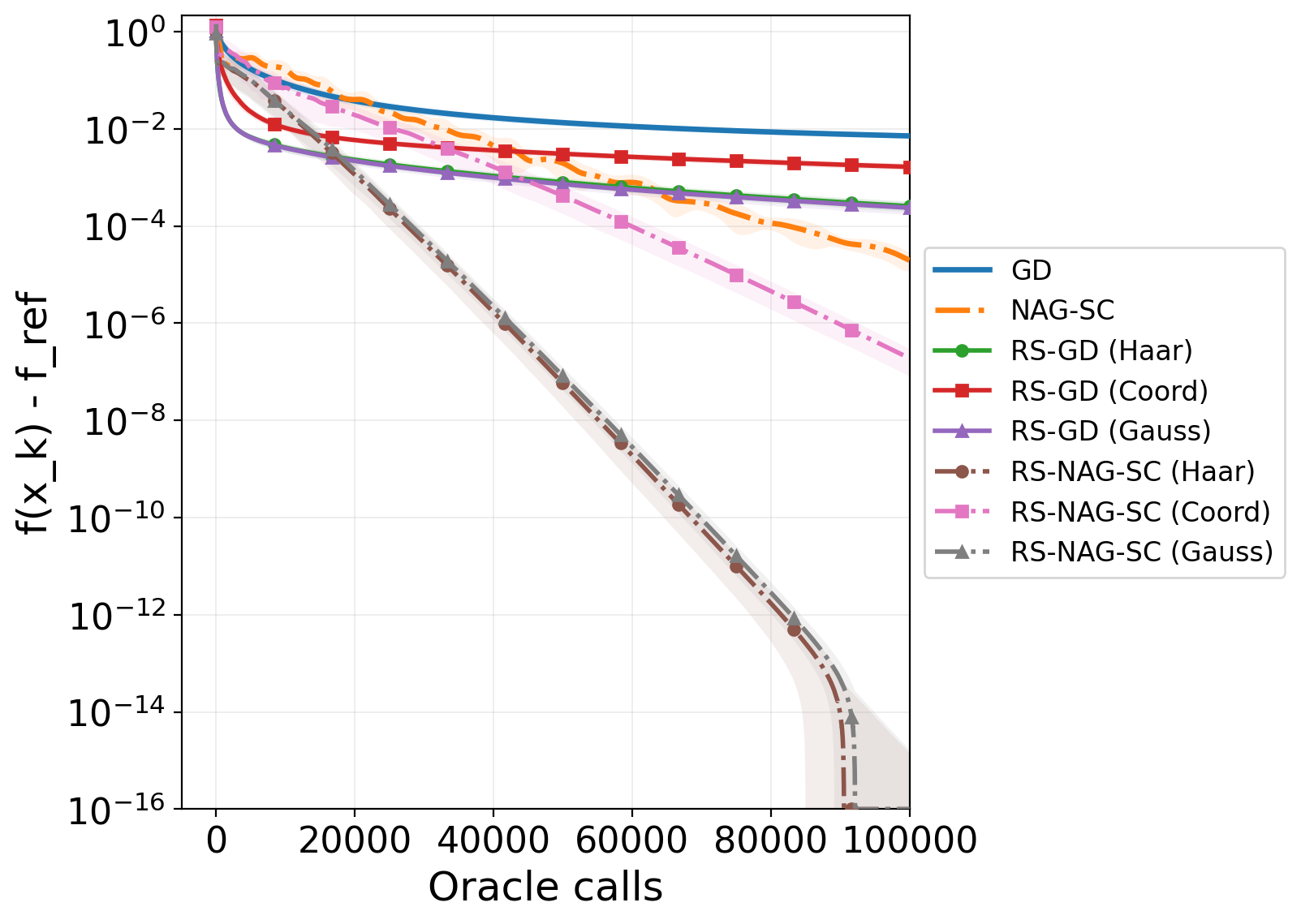}
    \caption{\text{a9a}}
\end{subfigure}
\hfill
\begin{subfigure}[t]{0.32\textwidth}
    \centering
    \includegraphics[width=\linewidth]{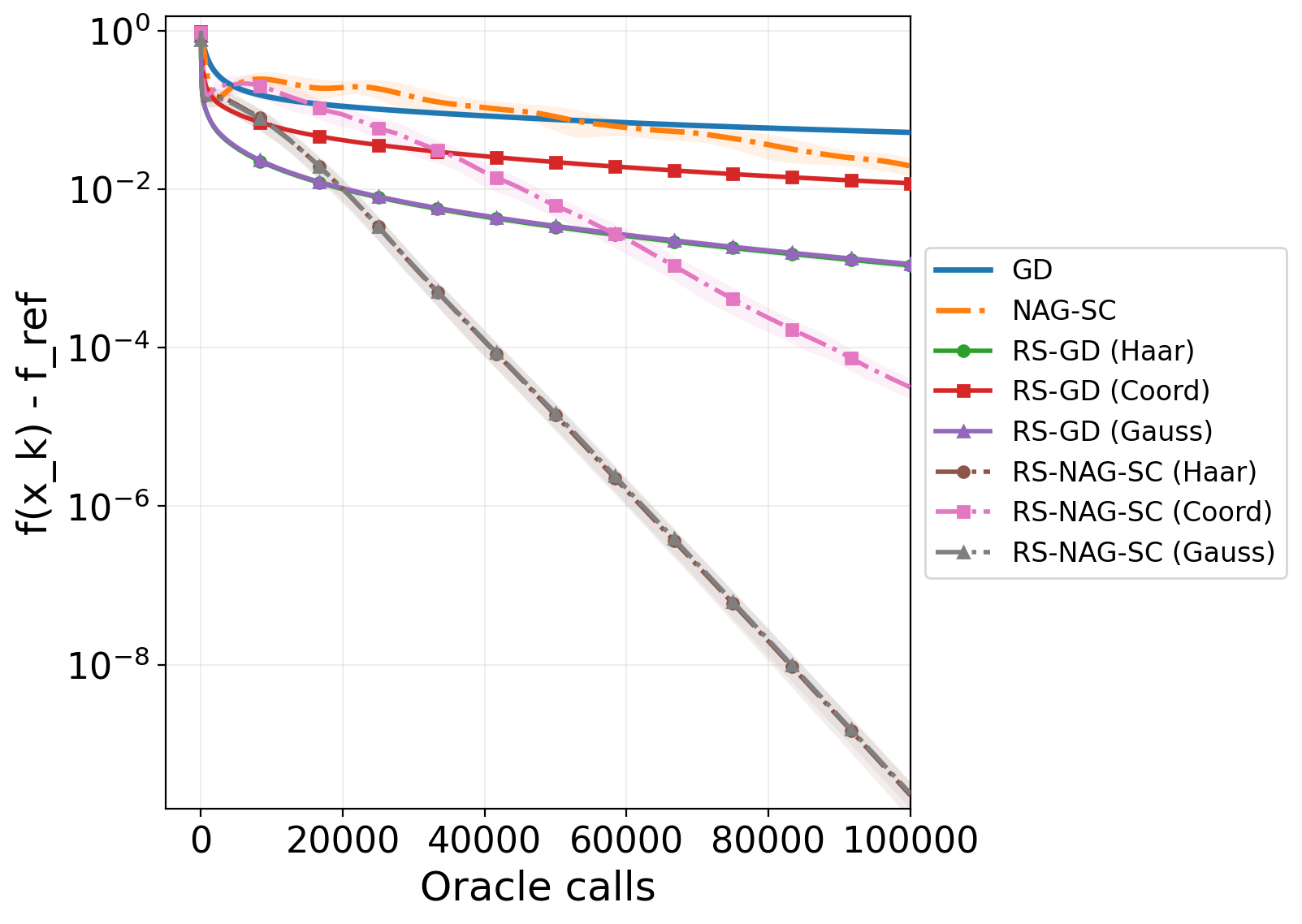}
    \caption{\text{w8a}}
\end{subfigure}

\vspace{0.6em}

\begin{subfigure}[t]{0.32\textwidth}
    \centering
    \includegraphics[width=\linewidth]{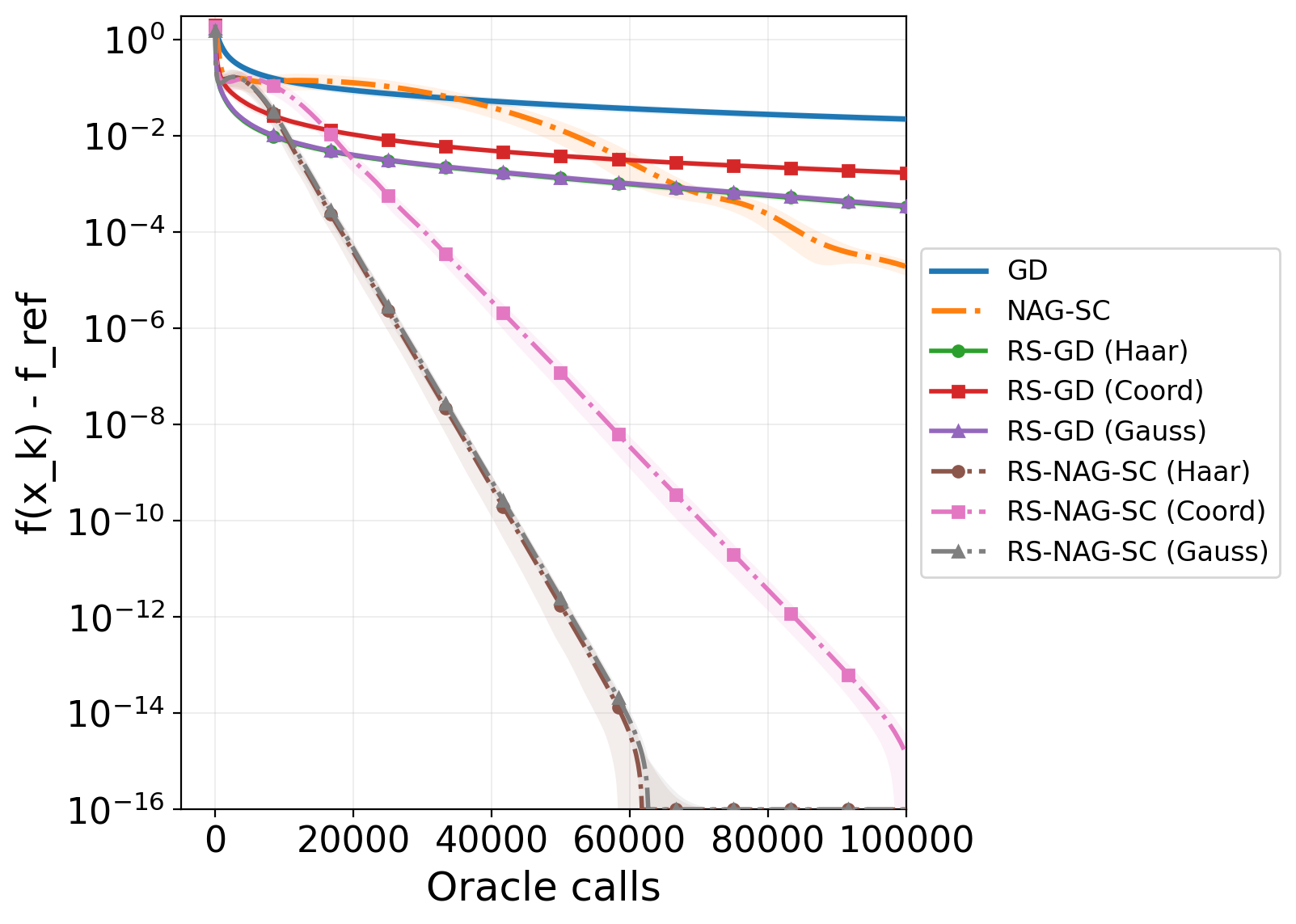}
    \caption{\text{mushroom}}
\end{subfigure}
\hfill
\begin{subfigure}[t]{0.32\textwidth}
    \centering
    \includegraphics[width=\linewidth]{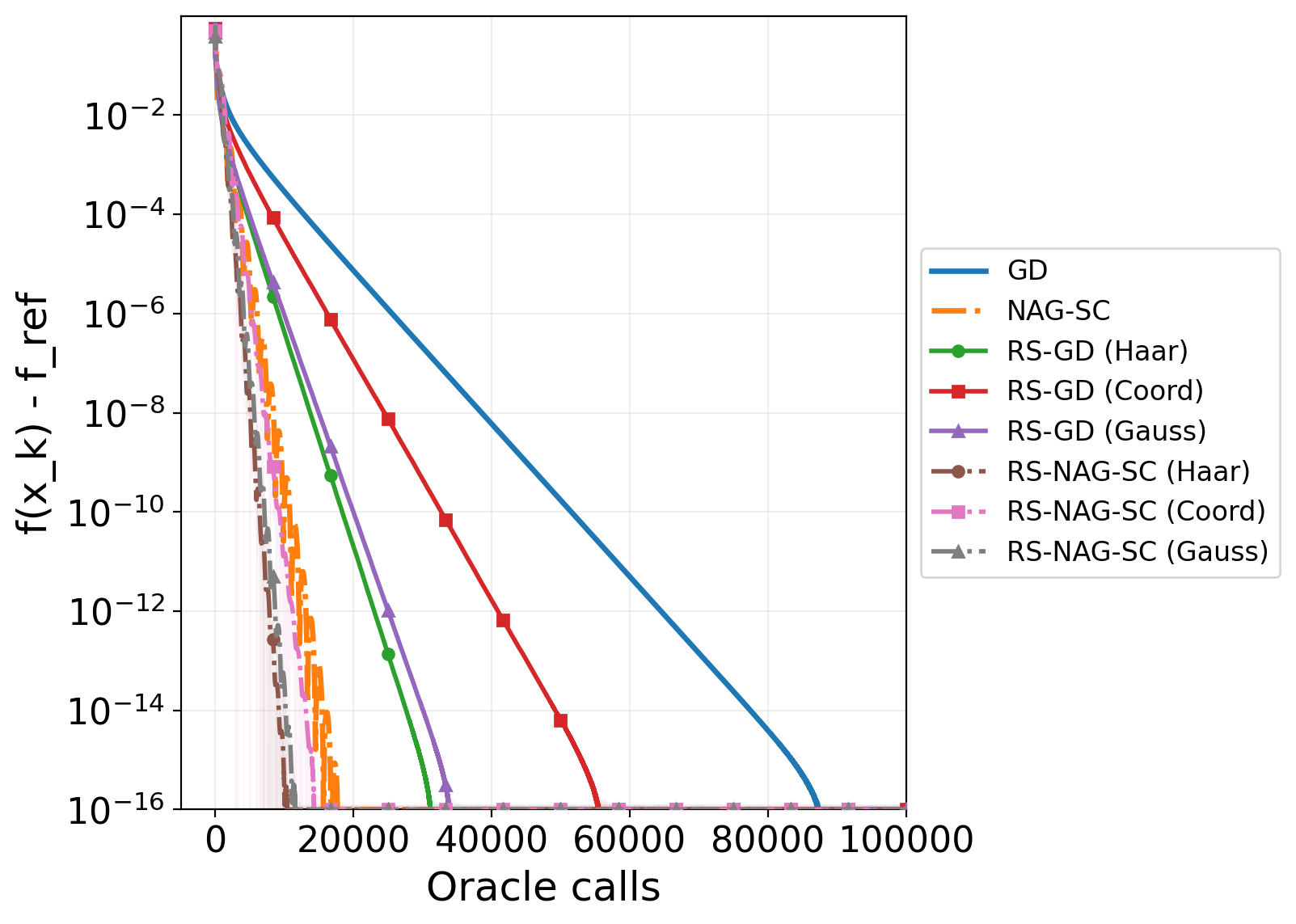}
    \caption{\text{ijcnn1}}
\end{subfigure}
\hfill
\begin{subfigure}[t]{0.32\textwidth}
    \centering
    \includegraphics[width=\linewidth]{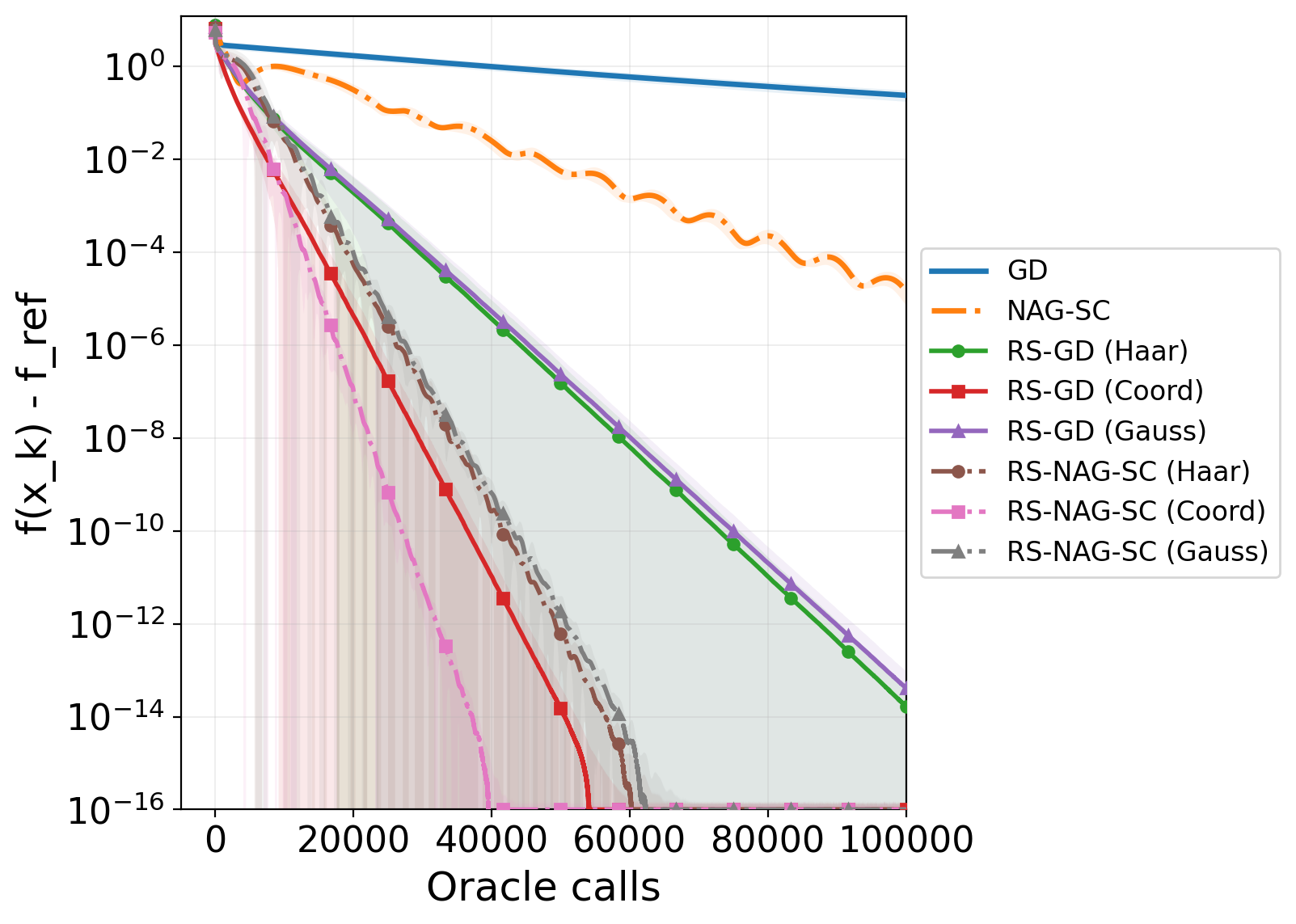}
    \caption{\text{splice}}
\end{subfigure}

\caption{
Oracle-axis comparison for $\ell_2$-regularized logistic regression on six real-world datasets.
The horizontal axis shows oracle calls, and the vertical axis shows $f(x_k)-f_{\mathrm{ref}}$ on a logarithmic scale, where $f_{\mathrm{ref}}$ is a reference value computed by L-BFGS-B.
We compare GD, NAG-SC, RS-GD, and RS-NAG-SC with Haar, coordinate, and Gaussian sketches.
Each plotted curve is the mean over $10$ random seeds, and the shaded region shows one standard deviation.
For each dataset, $\mu=1/n$, and the oracle budget is $100{,}000$
}
\label{fig:logreg-real-gap}
\end{figure*}

\begin{table}[t]
\centering
\caption{Dataset-dependent quantities for the logistic-regression datasets.
Here $d$ is the ambient dimension, $n$ is the number of training samples, and
$Q_{\mathrm H}, Q_{\mathrm G}, Q_{\mathrm C}$ denote the $r=1$ constants defined in Proposition~\ref{prop:comp-sketch-basic} for the Haar, Gaussian, and Coordinate sketches, respectively.}
\label{tab:logreg-q-values}
\small
\setlength{\tabcolsep}{4.5pt}
\renewcommand{\arraystretch}{1.1}
\begin{tabular}{@{}lrrrrrrr@{}}
\toprule
Dataset & $d$ & $Q_{\mathrm H}$ & $Q_{\mathrm G}$ & $Q_{\mathrm C}$ & $r_{\mathrm{eff}}$ & $\delta_{\mathrm{diag}}$ & $n$ \\
\midrule
phishing & 68  & 15.3670 & 15.8190 & 15.1421 & 1.5749 & 0.0496 & 11055 \\
a9a      & 123 & 22.5681 & 22.9350 & 47.8971 & 2.2081 & 0.1516 & 32561 \\
w8a      & 300 & 43.7200 & 44.0115 & 109.3820 & 4.4139 & 0.1329 & 49749 \\
mushroom & 112 & 21.0717 & 21.4480 & 34.8299 & 2.0352 & 0.0967 & 8124 \\
ijcnn1   & 22  & 13.1196 & 14.3123 & 17.5017 & 6.5350 & 0.6329 & 49990 \\
splice   & 60  & 13.5925 & 14.0456 & 9.5503  & 1.1819 & 0.0253 & 1000 \\
\bottomrule
\end{tabular}
\end{table}

\paragraph{Experimental setup.}
We evaluate the strongly convex logistic-regression setting on six real-world
binary-classification benchmarks:
\texttt{phishing}, \texttt{a9a}, \texttt{w8a},
\texttt{mushroom}, \texttt{ijcnn1}, and \texttt{splice}.
The UCI-derived datasets are cited collectively through the UCI Machine
Learning Repository~\citep{uci_repository}; for \texttt{w8a} and
\texttt{ijcnn1}, we follow the standard benchmark attributions
\citep{platt1998smo,prokhorov2001ijcnn}.

For each dataset, we use the $\ell_2$-regularized logistic objective in \eqref{eq:logreg-obj}, with the regularization parameter set to $\mu=1/n$, where $n$ is the number of training samples.
We compare the full-gradient methods GD and NAG-SC with the randomized-subspace methods RS-GD and RS-NAG-SC.
For the randomized-subspace methods, we use three sketch families: Haar, coordinate, and Gaussian.
We use $10$ random seeds, and initialize each run from a Gaussian random vector.
For logistic regression, we use
\[
\mathbb{L}=\frac{1}{4n}A^\top A+\mu I_d,
\]
and set $L=\|\mathbb{L}\|$, computed numerically as the largest eigenvalue of $\mathbb{L}$.
In the randomized-subspace methods, we use the $r=1$ setting throughout, in accordance with the theoretical comparison developed above.

\paragraph{Dataset-dependent quantities.}
For each dataset, we also compute the matrix-smoothness-derived quantities
$r_{\mathrm{eff}}$ and $\delta_{\mathrm{diag}}$, together with the corresponding $r=1$ constants
$Q_{\mathrm H}$, $Q_{\mathrm G}$, and $Q_{\mathrm C}$ for the Haar, Gaussian, and coordinate sketches, respectively, as defined in \Cref{prop:comp-sketch-basic}.
These values are reported in \Cref{tab:logreg-q-values}.

\paragraph{Discussion.}
We observe that RS-NAG-SC consistently achieves strong performance across all six datasets.
Moreover, the relative convergence behavior among the Haar, Gaussian, and Coordinate sketches is broadly consistent with the dataset-dependent $Q$ values reported in \Cref{tab:logreg-q-values}.
Indeed, on datasets such as \text{a9a}, \text{w8a}, \text{mushroom}, and \text{ijcnn1}, where $Q_{\mathrm H}$ and $Q_{\mathrm G}$ are relatively small, the Haar and Gaussian variants tend to converge faster than the Coordinate variant and also faster than NAG-SC.
On the other hand, for \text{phishing}, where the Coordinate constant is also relatively small, the Coordinate variant performs comparably well to the Haar and Gaussian variants.
For \text{splice}, where the Coordinate constant is particularly small, the Coordinate variant performs better.
Overall, these results suggest that the $Q$ values can serve as a useful practical guide when choosing the sketch distribution before running the method.

\subsection{Experimental resources, dataset sources, and terms of use}

\paragraph{Computational resources.}
All reported logistic-regression experiments were run on a single NVIDIA RTX
A5000 GPU with 24564 MiB of memory and CUDA support. The machine had an AMD
EPYC 7413 CPU and 503 GiB of system memory. The software environment used
Python 3.10.12, NumPy 2.2.6, SciPy 1.15.3, scikit-learn 1.7.2, Matplotlib
3.10.8, and PyTorch 2.11.0 with CUDA support. The main and additional logistic-regression experiments each took about
70 hours on this machine. 
Thus, the reported logistic-regression
experiments required about 140 GPU-hours in total.
The reported quadratic experiments were run separately on a local personal
computer using CPU only. The local machine had an Apple M1 processor and 8 GB
of memory. The main and appendix quadratic experiments each completed within
one wall-clock hour in our runs.

\paragraph{Dataset sources and terms of use.}
We use only previously released public benchmark datasets and do not redistribute
any dataset. The datasets \texttt{hiva\_agnostic} and \texttt{bioresponse} were
obtained from OpenML \citep{vanschoren2013openml}
with dataset IDs 1039 and 46912, respectively; their
OpenML license fields are listed as \texttt{Public} and \texttt{Public Domain},
respectively. 

The remaining datasets were obtained from the LIBSVM binary-classification dataset page \citep{chang2011libsvm}, which
provides LIBSVM-formatted versions of datasets from existing public benchmark collections and
lists source and preprocessing information for each dataset. We cite the corresponding original source papers or dataset records whenever
available. We checked the license or terms of use stated on the corresponding
source or access pages whenever available. We use the datasets only for research
evaluation and do not redistribute any dataset.

\section{Oracle complexity of the basic randomized-subspace gradient method}
\label{app:kozak-rates}

In \Cref{tab:oracle-comparison-intro}, the row labeled ``RS-GD (Kozak et al.)''
refers to the basic randomized-subspace gradient iteration
\begin{equation}
x_{k+1}=x_k-\eta\,P_kP_k^\top \nabla f(x_k),
\qquad k\ge 0,
\label{eq:kozak-basic-iter}
\end{equation}
introduced by \citet{kozak2021}.
The exact oracle-complexity expressions shown in
\Cref{tab:oracle-comparison-intro} are not stated in this form in
\citet{kozak2021}; for convenience, we record below short derivations.

\begin{proposition}[Convex rate for the basic randomized-subspace gradient method]
\label{prop:kozak-convex}
Suppose
\Cref{ass:matrix-smooth-common,ass:sketch-common,ass:func-convex} hold,
with \(\ell\) chosen so that \(\ell\le\omega\), as assumed throughout the paper.
Let \(\{x_k\}_{k\ge0}\) be generated by
\eqref{eq:kozak-basic-iter} with the constant step-size
\[
\eta=\frac{1}{2\omega L}.
\]

Then the expected objective values are nonincreasing:
\[
\E[f(x_{k+1})]\le \E[f(x_k)]
\qquad\text{for all }k\ge 0.
\]
Moreover, for every \(N\ge 1\),
\begin{equation}
\E[f(x_N)-f^\star]
\;\le\;
\frac{2\omega L\,\norm{x_0-x^\star}^2}{N}.
\label{eq:kozak-convex-rate}
\end{equation}
Consequently, defining \(R_0\coloneqq \norm{x_0-x^\star}\), it suffices to take
\[
N\ge \frac{2\omega L R_0^2}{\epsilon}
\]
to guarantee \(\E[f(x_N)-f^\star]\le \epsilon\). Since each iteration uses
\(r\) oracle calls, the oracle complexity is
\begin{equation}
\#\mathrm{Oracle}
=
\mathcal{O}\!\left(
r\omega R_0^2 \frac{L}{\epsilon}
\right).
\label{eq:kozak-convex-oracle}
\end{equation}
\end{proposition}

\begin{proof}
Let
\[
g_k\coloneqq \nabla f(x_k).
\]
Since \(\Lmat\succeq 0\) and \(L=\norm{\Lmat}\), we have \(\Lmat\preceq LI_d\).
Therefore \Cref{ass:matrix-smooth-common} implies the standard \(L\)-smoothness
inequality
\[
f(y)\le f(x)+\ip{\nabla f(x)}{y-x}+\frac{L}{2}\norm{y-x}^2
\qquad\text{for all }x,y\in\R^d.
\]
In particular, since \(f\) is convex and attains its minimum at \(x^\star\),
the standard smooth-convex inequality gives
\begin{equation}
\norm{\nabla f(x)}^2 \le 2L\bigl(f(x)-f^\star\bigr)
\qquad\text{for all }x\in\R^d.
\label{eq:grad-bound-convex-kozak}
\end{equation}

We first show monotonicity of the expected objective.
Applying \Cref{ass:matrix-smooth-common} with
\(x=x_k\) and \(y=x_{k+1}=x_k-\eta P_kP_k^\top g_k\), we obtain
\[
f(x_{k+1})
\le
f(x_k)-\eta\, g_k^\top P_kP_k^\top g_k
+\frac{\eta^2}{2}\, g_k^\top P_kP_k^\top \Lmat P_kP_k^\top g_k.
\]
Taking conditional expectation with respect to \(\F_k\), and using
\[
\E[P_kP_k^\top\mid \F_k]=I_d,
\qquad
\E[P_kP_k^\top \Lmat P_kP_k^\top\mid \F_k]\preceq \ell L I_d,
\]
we obtain
\[
\E[f(x_{k+1})\mid \F_k]
\le
f(x_k)-\eta\norm{g_k}^2+\frac{\eta^2\ell L}{2}\norm{g_k}^2
=
f(x_k)-\eta\Bigl(1-\frac{\eta\ell L}{2}\Bigr)\norm{g_k}^2.
\]
With \(\eta=1/(2\omega L)\) and \(\ell\le \omega\), we have
\[
\eta=\frac{1}{2\omega L}\le \frac{1}{2\ell L},
\]
and hence
\[
1-\frac{\eta\ell L}{2}\ge 0.
\]
Therefore,
\[
\E[f(x_{k+1})\mid \F_k]\le f(x_k).
\]

Next, expanding the squared distance to \(x^\star\), we obtain
\[
\norm{x_{k+1}-x^\star}^2
=
\norm{x_k-x^\star}^2
-2\eta\,\ip{x_k-x^\star}{P_kP_k^\top g_k}
+\eta^2 g_k^\top (P_kP_k^\top)^2 g_k.
\]
Taking conditional expectation and using
\[
\E[P_kP_k^\top\mid \F_k]=I_d,
\qquad
\E[(P_kP_k^\top)^2\mid \F_k]\preceq \omega I_d,
\]
we get
\[
\E[\norm{x_{k+1}-x^\star}^2\mid \F_k]
\le
\norm{x_k-x^\star}^2
-2\eta\,\ip{x_k-x^\star}{g_k}
+\eta^2\omega \norm{g_k}^2.
\]
By convexity,
\[
\ip{x_k-x^\star}{g_k}\ge f(x_k)-f^\star,
\]
and by \eqref{eq:grad-bound-convex-kozak},
\[
\norm{g_k}^2 \le 2L(f(x_k)-f^\star).
\]
Therefore
\[
\E[\norm{x_{k+1}-x^\star}^2\mid \F_k]
\le
\norm{x_k-x^\star}^2
-\bigl(2\eta-2\eta^2\omega L\bigr)\bigl(f(x_k)-f^\star\bigr).
\]
Taking expectation and using the tower property yields
\[
\E[\norm{x_{k+1}-x^\star}^2]
\le
\E[\norm{x_k-x^\star}^2]
-\bigl(2\eta-2\eta^2\omega L\bigr)\E[f(x_k)-f^\star].
\]
Summing for \(k=0,\dots,N-1\), we obtain
\[
\bigl(2\eta-2\eta^2\omega L\bigr)\sum_{k=0}^{N-1}\E[f(x_k)-f^\star]
\le
\norm{x_0-x^\star}^2.
\]
Since \(\E[f(x_k)]\) is nonincreasing,
\[
N\,\E[f(x_N)-f^\star]
\le
\sum_{k=0}^{N-1}\E[f(x_k)-f^\star]
\le
\frac{\norm{x_0-x^\star}^2}{2\eta-2\eta^2\omega L}.
\]
Using \(\eta=1/(2\omega L)\), we have
\[
2\eta-2\eta^2\omega L=\frac{1}{2\omega L}.
\]
Therefore
\[
\E[f(x_N)-f^\star]
\le
\frac{2\omega L\,\norm{x_0-x^\star}^2}{N},
\]
which proves \eqref{eq:kozak-convex-rate}.
The oracle bound
\eqref{eq:kozak-convex-oracle} follows immediately because each iteration uses
\(r\) oracle calls.
\end{proof}

\begin{proposition}[Strongly convex rate for the basic randomized-subspace gradient method]
\label{prop:kozak-strong}
Suppose
\Cref{ass:matrix-smooth-common,ass:sketch-common,ass:func} hold,
and let \(\{x_k\}_{k\ge0}\) be generated by
\eqref{eq:kozak-basic-iter} with the constant step-size
\[
\eta=\frac{1}{\ell L}.
\]
Then, for every \(k\ge 0\),
\begin{equation}
\E[f(x_{k+1})-f^\star]
\le
\left(1-\frac{\mu}{\ell L}\right)\E[f(x_k)-f^\star].
\label{eq:kozak-strong-onestep}
\end{equation}
Hence, for every \(N\ge 0\),
\begin{equation}
\E[f(x_N)-f^\star]
\le
\left(1-\frac{\mu}{\ell L}\right)^N \bigl(f(x_0)-f^\star\bigr)
\le
\exp\!\left(-\frac{\mu N}{\ell L}\right)\bigl(f(x_0)-f^\star\bigr).
\label{eq:kozak-strong-rate}
\end{equation}
Consequently, defining \(\Delta_0\coloneqq f(x_0)-f^\star\), it suffices to take
\[
N\ge \frac{\ell L}{\mu}\log\!{\frac{\Delta_0}{\epsilon}}
\]
to guarantee \(\E[f(x_N)-f^\star]\le \epsilon\). Since each iteration uses
\(r\) oracle calls, the oracle complexity is
\begin{equation}
\#\mathrm{Oracle}
=
\mathcal{O}\!\left(
r\ell \frac{L}{\mu}\log\frac{\Delta_0}{\epsilon}
\right).
\label{eq:kozak-strong-oracle}
\end{equation}
\end{proposition}

\begin{proof}
Let
\[
g_k\coloneqq \nabla f(x_k).
\]
Applying \Cref{ass:matrix-smooth-common} with
\(x=x_k\) and \(y=x_{k+1}=x_k-\eta P_kP_k^\top g_k\), we obtain
\[
f(x_{k+1})
\le
f(x_k)-\eta\, g_k^\top P_kP_k^\top g_k
+\frac{\eta^2}{2}\, g_k^\top P_kP_k^\top \Lmat P_kP_k^\top g_k.
\]
Taking conditional expectation with respect to \(\F_k\), and using
\[
\E[P_kP_k^\top\mid \F_k]=I_d,
\qquad
\E[P_kP_k^\top \Lmat P_kP_k^\top\mid \F_k]\preceq \ell L I_d,
\]
we obtain
\[
\E[f(x_{k+1})\mid \F_k]
\le
f(x_k)-\eta\norm{g_k}^2+\frac{\eta^2\ell L}{2}\norm{g_k}^2
=
f(x_k)-\eta\Bigl(1-\frac{\eta\ell L}{2}\Bigr)\norm{g_k}^2.
\]
With \(\eta=1/(\ell L)\), this yields
\[
\E[f(x_{k+1})\mid \F_k]
\le
f(x_k)-\frac{1}{2\ell L}\norm{g_k}^2.
\]
Since \(f\) is differentiable and \(\mu\)-strongly convex,
\[
\norm{\nabla f(x)}^2 \ge 2\mu\bigl(f(x)-f^\star\bigr)
\qquad\text{for all }x\in\R^d,
\]
we obtain
\[
\E[f(x_{k+1})-f^\star\mid \F_k]
\le
f(x_k)-f^\star-\frac{\mu}{\ell L}\bigl(f(x_k)-f^\star\bigr)
=
\left(1-\frac{\mu}{\ell L}\right)\bigl(f(x_k)-f^\star\bigr).
\]
Taking expectation proves \eqref{eq:kozak-strong-onestep}, and
\eqref{eq:kozak-strong-rate} follows by iteration. The oracle bound
\eqref{eq:kozak-strong-oracle} is immediate because each iteration uses
\(r\) oracle calls.
\end{proof}

\section{High-probability and almost-sure guarantees}
\label{app:hp-as}

We now record two consequences of the supermartingale Lyapunov bounds: one in the convex case and one in the strongly convex case.

In the convex case, the maximal inequality yields a uniform high-probability bound, and a dyadic Borel--Cantelli argument yields an almost sure eventual rate.
In the strongly convex case, after normalization by the linear contraction factor, the Lyapunov process is again a nonnegative supermartingale, and the almost sure argument becomes simpler: no dyadic reduction is needed.

\subsection{Convex case}

Recall the convex Lyapunov process
\[
\Phi_k^{\mathrm C}
\;\coloneqq\;
A_k\bigl(f(x_k)-f^\star\bigr)
+\frac12\|z_k-x^\star\|^2.
\]

\begin{proposition}[Convex case: high-probability and almost sure convergence]
\label{prop:convex-as}
Suppose \Cref{ass:matrix-smooth-common,ass:sketch-common,ass:func-convex} hold, and let
\(\{x_k,z_k,A_k\}_{k\ge 0}\) be generated by \Cref{alg:rs-NAG-c}.
Then, by \eqref{eq:conditional-psi-contraction}, the process
\((\Phi_k^{\mathrm C})_{k\ge 0}\) is a nonnegative supermartingale:
\[
\E[\Phi_{k+1}^{\mathrm C}\mid \mathcal F_k]\le \Phi_k^{\mathrm C}
\qquad\text{for all }k\ge 0.
\]
Consequently, the following hold.

\begin{enumerate}
\item[(i)] For every \(\eta\in(0,1)\), with probability at least \(1-\eta\),
\begin{equation}
\label{eq:convex-hp-AN}
A_k\bigl(f(x_k)-f^\star\bigr)\le \frac{\Phi_0^{\mathrm C}}{\eta}
\qquad\text{for all }k\ge 0.
\end{equation}
Consequently, since \(A_k>0\) for all \(k\ge 1\),
\begin{equation}
\label{eq:convex-hp-gap}
f(x_k)-f^\star\le \frac{\Phi_0^{\mathrm C}}{\eta\,A_k}
\qquad\text{for all }k\ge 1.
\end{equation}

\item[(ii)] Fix any \(\varepsilon>0\).
Then, almost surely, there exists a finite random integer \(K_\varepsilon\) such that for all integers \(k\ge K_\varepsilon\),
\begin{equation}
\label{eq:convex-as-AN}
A_k\bigl(f(x_k)-f^\star\bigr)
\le
\Phi_0^{\mathrm C}\,
\Bigl\lceil \log_2 k \Bigr\rceil
\Bigl(\log \bigl\lceil \log_2 k \bigr\rceil \Bigr)^{1+\varepsilon}.
\end{equation}

Consequently, for all \(k\ge K_\varepsilon\),
\begin{equation}
\label{eq:convex-as-gap}
f(x_k)-f^\star
\le
\frac{
\Phi_0^{\mathrm C}\,
\bigl\lceil \log_2 k \bigr\rceil
\bigl(\log \lceil \log_2 k \rceil \bigr)^{1+\varepsilon}
}{A_k}.
\end{equation}

Moreover, for the sequence \((A_k)\) generated by \Cref{alg:rs-NAG-c}, one has
\begin{equation}
\label{eq:Ak-lower-convex}
A_k\ge \frac{m}{2\omega}k^2=\frac{k^2}{4L\ell\omega}
\qquad\text{for all }k\ge 0,
\end{equation}
and therefore, for all \(k\ge K_\varepsilon\),
\begin{equation}
\label{eq:convex-as-gap-k2}
f(x_k)-f^\star
\le
4L\ell\omega\,\Phi_0^{\mathrm C}\,
\frac{
\bigl\lceil \log_2 k \bigr\rceil
\bigl(\log \lceil \log_2 k \rceil \bigr)^{1+\varepsilon}
}{k^2}.
\end{equation}
\end{enumerate}
\end{proposition}

\begin{proof}
We first prove part (i).
Since \((\Phi_k^{\mathrm C})\) is a nonnegative supermartingale, Ville's maximal inequality gives, for every \(b>0\),
\[
\mathbb P\!\left(\sup_{j\ge 0}\Phi_j^{\mathrm C}\ge b\right)\le \frac{\Phi_0^{\mathrm C}}{b}.
\]
Taking \(b=\Phi_0^{\mathrm C}/\eta\), we obtain
\[
\mathbb P\!\left(\sup_{j\ge 0}\Phi_j^{\mathrm C}\le \frac{\Phi_0^{\mathrm C}}{\eta}\right)\ge 1-\eta.
\]
Because
\[
A_k\bigl(f(x_k)-f^\star\bigr)\le \Phi_k^{\mathrm C}
\qquad\text{a.s. for every }k\ge 0,
\]
this proves \eqref{eq:convex-hp-AN}. Since \(A_k>0\) for all \(k\ge 1\),
\eqref{eq:convex-hp-gap} follows by dividing by \(A_k\).

We now prove part (ii).
Fix \(\varepsilon>0\), and for every integer \(m\ge 2\), define
\[
N_m\coloneqq 2^m,
\qquad
b_m\coloneqq \Phi_0^{\mathrm C}\,m(\log m)^{1+\varepsilon},
\]
and the bad events
\[
E_m
\coloneqq
\left\{
\max_{0\le j\le 2^m}\Phi_j^{\mathrm C}>b_m
\right\}.
\]
Applying Ville's maximal inequality on the finite horizon \(\{0,\dots,2^m\}\) gives
\[
\mathbb P(E_m)\le \frac{\Phi_0^{\mathrm C}}{b_m}
=
\frac{1}{m(\log m)^{1+\varepsilon}}.
\]
Since
\[
\sum_{m=2}^{\infty}\frac{1}{m(\log m)^{1+\varepsilon}}<\infty,
\]
the Borel--Cantelli lemma implies
\[
\mathbb P(E_m\ \mathrm{i.o.})=0.
\]
Therefore, almost surely, there exists a finite random integer
\(m_\varepsilon\ge 2\) such that for all \(m\ge m_\varepsilon\),
\begin{equation}
\label{eq:convex-dyadic-control}
\max_{0\le j\le 2^m}\Phi_j^{\mathrm C}
\le
\Phi_0^{\mathrm C}\,m(\log m)^{1+\varepsilon}.
\end{equation}

Now let
\[
K_\varepsilon\coloneqq 2^{m_\varepsilon}.
\]
Fix any integer \(k\ge K_\varepsilon\), and set
\[
m\coloneqq \left\lceil \log_2 k \right\rceil.
\]
Then \(m\ge m_\varepsilon\) and \(k\le 2^m\). Hence, by \eqref{eq:convex-dyadic-control},
\[
A_k\bigl(f(x_k)-f^\star\bigr)
\le
\Phi_k^{\mathrm C}
\le
\max_{0\le j\le 2^m}\Phi_j^{\mathrm C}
\le
\Phi_0^{\mathrm C}\,m(\log m)^{1+\varepsilon}.
\]
Since \(m=\lceil \log_2 k\rceil\), this proves \eqref{eq:convex-as-AN}. Dividing by \(A_k\) gives \eqref{eq:convex-as-gap}.

The lower bound \eqref{eq:Ak-lower-convex} follows directly from
\eqref{eq:AN-lower}. Finally, substituting \eqref{eq:Ak-lower-convex}
into \eqref{eq:convex-as-gap} yields \eqref{eq:convex-as-gap-k2}.
\end{proof}

\paragraph{Comparison with the expectation rate}
The expectation-level rate from the convex supermartingale argument is
\[
\E[f(x_k)-f^\star]\le \frac{\Phi_0^{\mathrm C}}{A_k}.
\]
By contrast, the uniform high-probability bound in \Cref{prop:convex-as}(i) differs only by the multiplicative factor \(1/\eta\):
\[
f(x_k)-f^\star\le \frac{\Phi_0^{\mathrm C}}{\eta\,A_k}
\qquad\text{for all }k\ge 0
\]
with probability at least \(1-\eta\).
The almost sure eventual bound in \Cref{prop:convex-as}(ii) differs from the expectation rate by the explicit logarithmic factor
\[
\bigl\lceil \log_2 k \bigr\rceil
\bigl(\log \lceil \log_2 k \rceil \bigr)^{1+\varepsilon}.
\]
No monotonicity of \(f(x_k)\) is used anywhere in the proof.

\subsection{Strongly convex case}

Define the strongly convex Lyapunov process by
\[
\Phi_k^{\mathrm{SC}}
\coloneqq
f(x_k)-f^\star+\frac{\mu}{2}\|z_k-x^\star\|^2.
\]
Set
\[
\rho\coloneqq 1-\theta.
\]
By \Cref{thm:rs-nesterov-strong}, \(\theta\in(0,1]\), and hence
\(\rho=1-\theta\in[0,1)\).Moreover, by
\eqref{eq:conditional-contraction},
\begin{equation}
\label{eq:sc-supermartingale}
\E[\Phi_{k+1}^{\mathrm{SC}}\mid \mathcal F_k]
\le
\rho\,\Phi_k^{\mathrm{SC}}
\qquad\text{for all }k\ge 0.
\end{equation}
Also, by the definition of \(\Phi_k^{\mathrm{SC}}\),
\begin{equation}
\label{eq:sc-dominate-gap}
f(x_k)-f^\star \le \Phi_k^{\mathrm{SC}}
\qquad\text{for all }k\ge 0.
\end{equation}

\begin{proposition}[Strongly convex case: high-probability and almost sure convergence]
\label{prop:sc-as}
Let \(\rho\coloneqq 1-\theta\).

If \(\rho=0\), then
\[
\Phi_k^{\mathrm{SC}}=0
\qquad\text{for all }k\ge 1
\]
almost surely. Consequently, for every \(\eta\in(0,1)\), with probability one,
\[
f(x_0)-f^\star\le \frac{\Phi_0^{\mathrm{SC}}}{\eta},
\qquad
f(x_k)-f^\star=0
\quad\text{for all }k\ge 1,
\]
and for every \(q\in(0,1)\), \eqref{eq:sc-as-eventual} holds with \(K_q=1\).

If \(\rho\in(0,1)\), then the following hold.
\begin{enumerate}
\item[(i)] For every \(\eta\in(0,1)\), with probability at least \(1-\eta\),
\begin{equation}
\label{eq:sc-hp-uniform}
f(x_k)-f^\star
\le
\frac{\Phi_0^{\mathrm{SC}}}{\eta}\,\rho^k
\qquad\text{for all }k\ge 0.
\end{equation}

\item[(ii)] For every number \(q\in(\rho,1)\), almost surely there exists a finite random integer \(K_q\) such that for all \(k\ge K_q\),
\begin{equation}
\label{eq:sc-as-eventual}
f(x_k)-f^\star \le \Phi_0^{\mathrm{SC}}\,q^k.
\end{equation}
\end{enumerate}
\end{proposition}

\begin{proof}
First consider the case \(\rho=0\). Since \(\Phi_{k+1}^{\mathrm{SC}}\ge 0\) and
\eqref{eq:sc-supermartingale} gives
\[
\E[\Phi_{k+1}^{\mathrm{SC}}\mid\mathcal F_k]\le 0,
\]
we have \(\Phi_{k+1}^{\mathrm{SC}}=0\) almost surely for every \(k\ge 0\).
Therefore \(\Phi_k^{\mathrm{SC}}=0\) almost surely for all \(k\ge 1\). The claims in
the case \(\rho=0\) follow from \eqref{eq:sc-dominate-gap}.

It remains to consider the case \(\rho\in(0,1)\). Define the normalized process
\[
M_k\coloneqq \rho^{-k}\Phi_k^{\mathrm{SC}}.
\]
Then
\[
\E[M_{k+1}\mid \mathcal F_k]
=
\rho^{-(k+1)}\E[\Phi_{k+1}^{\mathrm{SC}}\mid \mathcal F_k]
\le
\rho^{-k}\Phi_k^{\mathrm{SC}}
=
M_k.
\]
Hence \((M_k)_{k\ge 0}\) is a nonnegative supermartingale.

For part (i), Ville's maximal inequality gives
\[
\mathbb P\!\left(\sup_{j\ge 0}M_j\ge \frac{\Phi_0^{\mathrm{SC}}}{\eta}\right)\le \eta.
\]
Therefore, with probability at least \(1-\eta\),
\[
M_k\le \frac{\Phi_0^{\mathrm{SC}}}{\eta}
\qquad\text{for all }k\ge 0,
\]
that is,
\[
\Phi_k^{\mathrm{SC}}\le \frac{\Phi_0^{\mathrm{SC}}}{\eta}\rho^k
\qquad\text{for all }k\ge 0.
\]
Using \eqref{eq:sc-dominate-gap}, we obtain \eqref{eq:sc-hp-uniform}.

For part (ii), fix any \(q\in(\rho,1)\), and define
\[
\eta_k\coloneqq \left(\frac{\rho}{q}\right)^k.
\]
Since \(\rho/q<1\), the series \(\sum_{k=0}^{\infty}\eta_k\) converges.
Moreover, because \(\E[M_k]\le M_0=\Phi_0^{\mathrm{SC}}\), Markov's inequality gives
\[
\mathbb P\!\left(
M_k>\frac{\Phi_0^{\mathrm{SC}}}{\eta_k}
\right)
\le
\eta_k.
\]
Equivalently,
\[
\mathbb P\!\left(
\Phi_k^{\mathrm{SC}}>\Phi_0^{\mathrm{SC}}\,q^k
\right)
\le
\left(\frac{\rho}{q}\right)^k.
\]
Since
\[
\sum_{k=0}^{\infty}\left(\frac{\rho}{q}\right)^k<\infty,
\]
the Borel--Cantelli lemma implies that, almost surely, only finitely many of these events occur.
Hence, almost surely, there exists a finite random integer \(K_q\) such that for all \(k\ge K_q\),
\[
\Phi_k^{\mathrm{SC}}\le \Phi_0^{\mathrm{SC}}\,q^k.
\]
Using \eqref{eq:sc-dominate-gap}, we obtain \eqref{eq:sc-as-eventual}.
\end{proof}

\paragraph{Comparison with the expectation rate}
The expectation-level strongly convex rate is
\[
\E[f(x_k)-f^\star]\le \Phi_0^{\mathrm{SC}}\rho^k.
\]
The uniform high-probability bound in \Cref{prop:sc-as}(i) differs only by the multiplicative factor \(1/\eta\):
\[
f(x_k)-f^\star\le \frac{\Phi_0^{\mathrm{SC}}}{\eta}\rho^k
\qquad\text{for all }k\ge 0
\]
with probability at least \(1-\eta\).
The almost sure eventual bound in \Cref{prop:sc-as}(ii) replaces the contraction factor \(\rho\) by any prescribed factor \(q\in(\rho,1)\).
Thus the almost sure linear rate can be made arbitrarily close to the expectation linear rate, but it is not exactly identical.

\section{Why the classical two-sequence template is not directly portable}
\label{app:two-sequence-obstruction}

In this appendix, we formalize a limitation of a direct transplantation of the
classical two-sequence Nesterov template to randomized-subspace updates.
These results should be interpreted as proof-template obstructions rather than impossibility
results for all conceivable two-sequence accelerated schemes.
Rather, they show that, under the sketch structure satisfied by the Haar and coordinate sketches,
the classical proof mechanism is not directly compatible with the naive sketched analogue
outside a near-full-dimensional regime.

\begin{proposition}
\label{prop:convex-two-seq-obstruction}
Let \(f:\R^d\to\R\) be convex and \(L\)-smooth, and let \(x^\star\) be a minimizer of \(f\).
Let \(\{P_k\}_{k\ge 0}\) be an i.i.d.\ sketch sequence satisfying
\[
\E[P_kP_k^\top]=I_d,
\qquad
P_k^\top P_k=\frac{d}{r}I_r
\]
for some \(1\le r\le d\).
Consider the direct sketched two-sequence recursion
\[
x_{k+1}=y_k-\eta P_kP_k^\top \nabla f(y_k),
\qquad
y_{k+1}=x_{k+1}+\beta_k(x_{k+1}-x_k),
\]
where \(\eta>0\) is constant and \(\{\beta_k\}_{k\ge 0}\) is arbitrary.

For each \(k\ge0\), let
\[
\F_k\coloneqq \sigma(P_0,\dots,P_{k-1}),
\qquad
g_k\coloneqq \nabla f(y_k),
\qquad
M_k\coloneqq P_kP_k^\top.
\]

Suppose, as in the classical proof template, that there exist sequences
\(\{\lambda_k\}_{k\ge0}\subset(0,1]\) and
\(\{z_k\}_{k\ge0}\subset\R^d\) such that, for every \(k\ge0\),
\begin{equation}
y_k=(1-\lambda_k)x_k+\lambda_k z_k,
\qquad
z_{k+1}=z_k-\frac{\eta}{\lambda_k}M_k g_k.
\label{eq:convex-obstruction-affine}
\end{equation}

Then, for every \(k\ge0\),
\begin{align}
&\E\!\left[
\frac{1}{\lambda_k^2}\bigl(f(x_{k+1})-f^\star\bigr)
+\frac{1}{2\eta}\|z_{k+1}-x^\star\|^2
\;\middle|\; \F_k
\right]
\notag\\
&\le
\frac{1-\lambda_k}{\lambda_k^2}\bigl(f(x_k)-f^\star\bigr)
+\frac{1}{2\eta}\|z_k-x^\star\|^2
+\frac{\eta}{2\lambda_k^2}
\Bigl(
\frac{d}{r}(1+L\eta)-2
\Bigr)\|g_k\|^2.
\label{eq:convex-obstruction-main} 
\end{align}

In particular, if \(r\le d/2\), then the coefficient
\[
\frac{\eta}{2\lambda_k^2}
\Bigl(
\frac{d}{r}(1+L\eta)-2
\Bigr)
\]
is strictly positive for every \(\eta>0\).
Therefore, the estimate \eqref{eq:convex-obstruction-main} retains a strictly positive residual coefficient, and hence this affine hidden-sequence template does not recover the classical two-term decrease by the standard argument.
\end{proposition}

\begin{proof}
Fix \(k\ge0\).
Since \(P_k^\top P_k=(d/r)I_r\),
\[
M_k^2
=
P_k(P_k^\top P_k)P_k^\top
=
\frac{d}{r}M_k.
\]
Moreover, \(P_k\) is independent of \(\F_k\), and \(\E[M_k]=I_d\), so
\[
\E[M_k\mid\F_k]=I_d,
\qquad
\E[M_k^2\mid\F_k]
=
\frac{d}{r}I_d.
\label{eq:convex-obstruction-M2-proof}
\]

By \(L\)-smoothness of \(f\),
\[
f(v)\le f(u)+\ip{\nabla f(u)}{v-u}+\frac{L}{2}\|v-u\|^2
\qquad
\text{for all }u,v\in\R^d.
\]
Applying this with \(u=y_k\) and \(v=x_{k+1}=y_k-\eta M_k g_k\), we obtain
\[
f(x_{k+1})
\le
f(y_k)-\eta g_k^\top M_k g_k+\frac{L\eta^2}{2}g_k^\top M_k^2 g_k.
\]
Taking conditional expectation and using \(\E[M_k\mid\F_k]=I_d\) and
\(\E[M_k^2\mid\F_k]=\frac{d}{r}I_d\) gives
\begin{equation}
\E[f(x_{k+1})\mid\F_k]
\le
f(y_k)-\eta\|g_k\|^2+\frac{L\eta^2}{2}\frac{d}{r}\|g_k\|^2.
\label{eq:convex-obstruction-xstep}
\end{equation}

We now estimate the \(z\)-term. From \eqref{eq:convex-obstruction-affine},
\[
z_{k+1}-x^\star
=
(z_k-x^\star)-\frac{\eta}{\lambda_k}M_k g_k,
\]
and hence
\[
\|z_{k+1}-x^\star\|^2
=
\|z_k-x^\star\|^2
-\frac{2\eta}{\lambda_k}\ip{z_k-x^\star}{M_k g_k}
+\frac{\eta^2}{\lambda_k^2}\|M_k g_k\|^2.
\]
Multiplying by \(1/(2\eta)\), taking conditional expectation, and using again
\(\E[M_k\mid\F_k]=I_d\) and \(\E[M_k^2\mid\F_k]=\frac{d}{r}I_d\), we get
\begin{align}
&\E\!\left[
\frac{1}{2\eta}\|z_{k+1}-x^\star\|^2
\;\middle|\;\F_k
\right]
\notag\\
&=
\frac{1}{2\eta}\|z_k-x^\star\|^2
-\frac{1}{\lambda_k}\ip{z_k-x^\star}{g_k}
+\frac{\eta}{2\lambda_k^2}\frac{d}{r}\|g_k\|^2.
\label{eq:convex-obstruction-zstep}
\end{align}

Subtracting \(f^\star\) from \eqref{eq:convex-obstruction-xstep}, multiplying by \(1/\lambda_k^2\), and adding \eqref{eq:convex-obstruction-zstep}, we obtain
\begin{align}
&\E\!\left[
\frac{1}{\lambda_k^2}\bigl(f(x_{k+1})-f^\star\bigr)
+\frac{1}{2\eta}\|z_{k+1}-x^\star\|^2
\;\middle|\;\F_k
\right]
\notag\\
&\le
\frac{1}{\lambda_k^2}\bigl(f(y_k)-f^\star\bigr)
+\frac{1}{2\eta}\|z_k-x^\star\|^2
-\frac{1}{\lambda_k}\ip{z_k-x^\star}{g_k}
\notag\\
&\qquad
+\frac{\eta}{2\lambda_k^2}
\Bigl(
\frac{d}{r}(1+L\eta)-2
\Bigr)\|g_k\|^2.
\label{eq:convex-obstruction-before-convexity}
\end{align}

It remains to control the mixed term. Since
\[
y_k=(1-\lambda_k)x_k+\lambda_k z_k,
\]
we have
\[
\frac{1-\lambda_k}{\lambda_k^2}(x_k-y_k)
+\frac{1}{\lambda_k}(x^\star-y_k)
=
\frac{1}{\lambda_k}(x^\star-z_k).
\]
Taking inner products with \(g_k\) gives
\begin{equation}
-\frac{1}{\lambda_k}\ip{z_k-x^\star}{g_k}
=
\frac{1-\lambda_k}{\lambda_k^2}\ip{x_k-y_k}{g_k}
+\frac{1}{\lambda_k}\ip{x^\star-y_k}{g_k}.
\label{eq:convex-obstruction-keyid}
\end{equation}
By convexity of \(f\),
\[
f(x_k)\ge f(y_k)+\ip{g_k}{x_k-y_k},
\qquad
f^\star=f(x^\star)\ge f(y_k)+\ip{g_k}{x^\star-y_k},
\]
so
\[
\ip{x_k-y_k}{g_k}\le f(x_k)-f(y_k),
\qquad
\ip{x^\star-y_k}{g_k}\le f^\star-f(y_k).
\]
Substituting these bounds into \eqref{eq:convex-obstruction-keyid} yields
\begin{equation}
-\frac{1}{\lambda_k}\ip{z_k-x^\star}{g_k}
\le
\frac{1-\lambda_k}{\lambda_k^2}\bigl(f(x_k)-f(y_k)\bigr)
+
\frac{1}{\lambda_k}\bigl(f^\star-f(y_k)\bigr).
\label{eq:convex-obstruction-keyupper}
\end{equation}

Substituting \eqref{eq:convex-obstruction-keyupper} into
\eqref{eq:convex-obstruction-before-convexity}, we get
\begin{align*}
&\E\!\left[
\frac{1}{\lambda_k^2}\bigl(f(x_{k+1})-f^\star\bigr)
+\frac{1}{2\eta}\|z_{k+1}-x^\star\|^2
\;\middle|\;\F_k
\right] \\
&\le
\frac{1}{\lambda_k^2}\bigl(f(y_k)-f^\star\bigr)
+\frac{1-\lambda_k}{\lambda_k^2}\bigl(f(x_k)-f(y_k)\bigr)
+\frac{1}{\lambda_k}\bigl(f^\star-f(y_k)\bigr)
+\frac{1}{2\eta}\|z_k-x^\star\|^2 \\
&\qquad
+\frac{\eta}{2\lambda_k^2}
\Bigl(
\frac{d}{r}(1+L\eta)-2
\Bigr)\|g_k\|^2.
\end{align*}
The \(f(y_k)\)-terms cancel, since
\[
\frac{1}{\lambda_k^2}
-\frac{1-\lambda_k}{\lambda_k^2}
-\frac{1}{\lambda_k}
=0.
\]
Therefore,
\[
\frac{1}{\lambda_k^2}\bigl(f(y_k)-f^\star\bigr)
+\frac{1-\lambda_k}{\lambda_k^2}\bigl(f(x_k)-f(y_k)\bigr)
+\frac{1}{\lambda_k}\bigl(f^\star-f(y_k)\bigr)
=
\frac{1-\lambda_k}{\lambda_k^2}\bigl(f(x_k)-f^\star\bigr),
\]
and \eqref{eq:convex-obstruction-main} follows.

Finally, if \(r\le d/2\), then \(d/r\ge2\), and thus
\[
\frac{d}{r}(1+L\eta)-2
\ge
2(1+L\eta)-2
=
2L\eta
>
0
\qquad
\text{for every }\eta>0.
\]
Hence the coefficient of \(\|g_k\|^2\) in \eqref{eq:convex-obstruction-main} is strictly positive.
Therefore, \eqref{eq:convex-obstruction-main} retains a strictly positive residual coefficient, and thus the estimate does not collapse to the classical two-term decrease by the standard argument.
\end{proof}

\begin{proposition}
\label{prop:strong-two-seq-obstruction-collapse}
Let \(f:\mathbb{R}^d\to\mathbb{R}\) be \(\mu\)-strongly convex and \(L\)-smooth, with \(\mu>0\).
Let \(\{P_k\}_{k\ge 0}\) be an i.i.d.\ sketch sequence satisfying
\[
\mathbb{E}[P_kP_k^\top]=I_d,
\qquad
P_k^\top P_k=\frac{d}{r}I_r
\]
for some \(1\le r\le d\).
Consider the direct sketched two-sequence recursion
\begin{equation}
x_{k+1}=y_k-\eta P_kP_k^\top \nabla f(y_k),
\qquad
y_{k+1}=x_{k+1}+\beta(x_{k+1}-x_k),
\label{eq:strong-direct-two-seq-collapse}
\end{equation}
where \(\eta>0\) and \(\beta\in[0,1)\) are constants. Define
\[
\theta\coloneqq \frac{1-\beta}{1+\beta}\in(0,1],
\qquad
\gamma\coloneqq \frac{\eta}{\theta},
\qquad
z_k\coloneqq \frac{(1+\theta)y_k-x_k}{\theta},
\qquad
u_k\coloneqq (1-\theta)z_k+\theta y_k,
\]
and
\[
\Phi_k\coloneqq f(x_k)-f^\star+\frac{\mu}{2}\|z_k-x^\star\|^2.
\]
Then
\begin{equation}
y_k=\frac{1}{1+\theta}x_k+\frac{\theta}{1+\theta}z_k,
\qquad
z_{k+1}=u_k-\gamma P_kP_k^\top \nabla f(y_k).
\label{eq:strong-hidden-two-seq-collapse}
\end{equation}

Writing \(g_k\coloneqq \nabla f(y_k)\), \(M_k\coloneqq P_kP_k^\top\), and
\(\mathcal F_k\coloneqq \sigma(P_0,\dots,P_{k-1})\), one has
\begin{align}
&\mathbb{E}\!\left[\Phi_{k+1}\middle|\mathcal F_k\right]-(1-\theta)\Phi_k
\notag\\
&\le
\theta\bigl(f(y_k)-f^\star\bigr)
+(1-\theta)\bigl(f(y_k)-f(x_k)\bigr)
-\mu\gamma\langle g_k,y_k-x^\star\rangle
+\mu\gamma\frac{1-\theta}{\theta}\langle g_k,x_k-y_k\rangle
\notag\\
&\qquad
+\frac{\mu\theta}{2}\|y_k-x^\star\|^2
-\frac{\mu(1-\theta)}{2\theta}\|x_k-y_k\|^2
+\left(
-\eta+\frac{d}{2r}(L\eta^2+\mu\gamma^2)
\right)\|g_k\|^2.
\label{eq:strong-precollapse-general}
\end{align}

If one wants the right-hand side of \eqref{eq:strong-precollapse-general} to collapse into the same two brackets as in the standard strongly-convex proof, namely
\begin{align}
&\mathbb{E}\!\left[\Phi_{k+1}\middle|\mathcal F_k\right]-(1-\theta)\Phi_k
\notag\\
&\le
\theta\Bigl(
f(y_k)-f^\star-\langle g_k,y_k-x^\star\rangle+\frac{\mu}{2}\|y_k-x^\star\|^2
\Bigr)
\notag\\
&\qquad
+(1-\theta)\Bigl(
f(y_k)-f(x_k)+\langle g_k,x_k-y_k\rangle
\Bigr)
\notag\\
&\qquad
-\frac{\mu(1-\theta)}{2\theta}\|x_k-y_k\|^2
\label{eq:strong-desired-collapse}
\end{align}
then matching the coefficients of the inner-product terms yields
\begin{equation}
\mu\gamma=\theta,
\qquad\text{that is,}\qquad
\gamma=\frac{\theta}{\mu},
\qquad\text{equivalently}\qquad
\eta=\frac{\theta^2}{\mu}.
\label{eq:strong-coupling-choice-collapse}
\end{equation}
Under \eqref{eq:strong-coupling-choice-collapse}, \eqref{eq:strong-precollapse-general} becomes
\begin{align}
&\mathbb{E}\!\left[\Phi_{k+1}\middle|\mathcal F_k\right]-(1-\theta)\Phi_k
\notag\\
&\le
\theta\Bigl(
f(y_k)-f^\star-\langle g_k,y_k-x^\star\rangle+\frac{\mu}{2}\|y_k-x^\star\|^2
\Bigr)
\notag\\
&\qquad
+(1-\theta)\Bigl(
f(y_k)-f(x_k)+\langle g_k,x_k-y_k\rangle
\Bigr)
\notag\\
&\qquad
-\frac{\mu(1-\theta)}{2\theta}\|x_k-y_k\|^2
+\frac{\theta^2}{\mu}
\left[
-1+\frac{d}{2r}\left(1+\frac{L\theta^2}{\mu}\right)
\right]\|g_k\|^2.
\label{eq:strong-collapse-final}
\end{align}

Therefore, in order for the remaining coefficient of \(\|g_k\|^2\) in
\eqref{eq:strong-collapse-final} to be nonpositive, it is necessary that
\begin{equation}
\frac{d}{r}\left(1+\frac{L\theta^2}{\mu}\right)\le 2.
\label{eq:strong-obstruction-condition-theta-collapse}
\end{equation}
Equivalently, since \(\eta=\theta^2/\mu\),
\begin{equation}
\frac{d}{r}(1+L\eta)\le 2.
\label{eq:strong-obstruction-condition-eta-collapse}
\end{equation}
In particular, if \(r\le d/2\), then \eqref{eq:strong-obstruction-condition-theta-collapse}
cannot hold for any \(\beta\in[0,1)\).
Therefore, outside a near-full-dimensional regime, the direct sketched analogue of the
classical strongly-convex two-sequence Nesterov template is incompatible with the
standard hidden-sequence proof mechanism.
\end{proposition}

\begin{proof}
Fix \(k\ge0\), and let
\[
\mathcal F_k\coloneqq \sigma(P_0,\dots,P_{k-1}),
\qquad
M_k\coloneqq P_kP_k^\top,
\qquad
g_k\coloneqq \nabla f(y_k).
\]

We first rewrite \eqref{eq:strong-direct-two-seq-collapse} in hidden-sequence form.
By the definition of \(z_k\),
\[
z_k=\frac{(1+\theta)y_k-x_k}{\theta},
\]
so that
\[
y_k=\frac{1}{1+\theta}x_k+\frac{\theta}{1+\theta}z_k.
\]
Also,
\[
z_{k+1}
=\frac{(1+\theta)y_{k+1}-x_{k+1}}{\theta}.
\]
Using
\[
y_{k+1}=x_{k+1}+\frac{1-\theta}{1+\theta}(x_{k+1}-x_k),
\]
we obtain
\begin{align*}
z_{k+1}
&=
\frac{(1+\theta)y_{k+1}-x_{k+1}}{\theta}\\
&=
\frac{1}{\theta}x_{k+1}-\frac{1-\theta}{\theta}x_k\\
&=
\frac{1}{\theta}(y_k-\eta M_k g_k)-\frac{1-\theta}{\theta}x_k.
\end{align*}
On the other hand,
\[
(1-\theta)z_k+\theta y_k
=
\frac{1-\theta}{\theta}\bigl((1+\theta)y_k-x_k\bigr)+\theta y_k
=
\frac{1}{\theta}y_k-\frac{1-\theta}{\theta}x_k.
\]
Hence
\[
z_{k+1}=(1-\theta)z_k+\theta y_k-\frac{\eta}{\theta}M_k g_k
=u_k-\gamma M_k g_k,
\]
which proves \eqref{eq:strong-hidden-two-seq-collapse}.

We next derive the one-step Lyapunov estimate for a general \(\gamma\), and only then identify
the value of \(\gamma\) for which the same collapse pattern as in the standard proof occurs.

From
\[
x_k-y_k=\theta(y_k-z_k),
\]
we obtain
\[
z_k-x^\star
=
y_k-x^\star-\frac{1}{\theta}(x_k-y_k).
\]
Also, since
\[
u_k=(1-\theta)z_k+\theta y_k
=
y_k-\frac{1-\theta}{\theta}(x_k-y_k),
\]
the update \(z_{k+1}=u_k-\gamma M_k g_k\) yields
\[
z_{k+1}-x^\star
=
y_k-x^\star-\frac{1-\theta}{\theta}(x_k-y_k)-\gamma M_k g_k.
\]
Therefore, a direct expansion gives
\begin{align}
&\|z_{k+1}-x^\star\|^2-(1-\theta)\|z_k-x^\star\|^2
\notag\\
&=
\theta\|y_k-x^\star\|^2
-\frac{1-\theta}{\theta}\|x_k-y_k\|^2
-2\gamma\langle M_k g_k,y_k-x^\star\rangle
+\frac{2\gamma(1-\theta)}{\theta}\langle M_k g_k,x_k-y_k\rangle
+\gamma^2 g_k^\top M_k^2 g_k.
\label{eq:strong-square-expansion-general}
\end{align}

Now
\[
M_k^2=P_k(P_k^\top P_k)P_k^\top=\frac{d}{r}M_k.
\]
Since \(P_k\) is independent of \(\mathcal F_k\) and \(\mathbb E[M_k]=I_d\), we have
\[
\mathbb E[M_k\mid\mathcal F_k]=I_d,
\qquad
\mathbb E[M_k^2\mid\mathcal F_k]=\frac{d}{r}I_d.
\]
Taking conditional expectation in \eqref{eq:strong-square-expansion-general}, we obtain
\begin{align}
&\mathbb E\!\left[
\frac{\mu}{2}\|z_{k+1}-x^\star\|^2
\middle|\mathcal F_k
\right]
-\frac{(1-\theta)\mu}{2}\|z_k-x^\star\|^2
\notag\\
&=
\frac{\mu\theta}{2}\|y_k-x^\star\|^2
-\mu\gamma\langle g_k,y_k-x^\star\rangle
+\mu\gamma\frac{1-\theta}{\theta}\langle g_k,x_k-y_k\rangle
-\frac{\mu(1-\theta)}{2\theta}\|x_k-y_k\|^2
+\frac{d}{2r}\mu\gamma^2\|g_k\|^2.
\label{eq:strong-z-diff-general}
\end{align}

On the function-value side, by \(L\)-smoothness and
\[
x_{k+1}=y_k-\eta M_k g_k,
\]
we have
\[
f(x_{k+1})
\le
f(y_k)-\eta g_k^\top M_k g_k+\frac{L\eta^2}{2}g_k^\top M_k^2 g_k.
\]
Taking conditional expectation gives
\begin{equation}
\mathbb E[f(x_{k+1})\mid \mathcal F_k]
\le
f(y_k)-\eta\|g_k\|^2+\frac{d}{2r}L\eta^2\|g_k\|^2.
\label{eq:strong-smooth-step-general}
\end{equation}

We now expand \(\mathbb E[\Phi_{k+1}\mid\mathcal F_k]-(1-\theta)\Phi_k\):
\begin{align}
&\mathbb E\!\left[\Phi_{k+1}\middle|\mathcal F_k\right]-(1-\theta)\Phi_k
\notag\\
&=
\mathbb E[f(x_{k+1})\mid\mathcal F_k]-f^\star-(1-\theta)(f(x_k)-f^\star)
\notag\\
&\qquad
+\mathbb E\!\left[
\frac{\mu}{2}\|z_{k+1}-x^\star\|^2
\middle|\mathcal F_k
\right]
-\frac{(1-\theta)\mu}{2}\|z_k-x^\star\|^2
\notag\\
&=
\bigl(\mathbb E[f(x_{k+1})\mid\mathcal F_k]-f(y_k)\bigr)
+\theta(f(y_k)-f^\star)
+(1-\theta)(f(y_k)-f(x_k))
\notag\\
&\qquad
+\mathbb E\!\left[
\frac{\mu}{2}\|z_{k+1}-x^\star\|^2
\middle|\mathcal F_k
\right]
-\frac{(1-\theta)\mu}{2}\|z_k-x^\star\|^2.
\label{eq:strong-phi-decomp-general}
\end{align}
Substituting \eqref{eq:strong-z-diff-general} and \eqref{eq:strong-smooth-step-general} into
\eqref{eq:strong-phi-decomp-general}, we obtain
\begin{align}
&\mathbb E\!\left[\Phi_{k+1}\middle|\mathcal F_k\right]-(1-\theta)\Phi_k
\notag\\
&\le
\theta(f(y_k)-f^\star)
+(1-\theta)(f(y_k)-f(x_k))
-\mu\gamma\langle g_k,y_k-x^\star\rangle
+\mu\gamma\frac{1-\theta}{\theta}\langle g_k,x_k-y_k\rangle
\notag\\
&\qquad
+\frac{\mu\theta}{2}\|y_k-x^\star\|^2
-\frac{\mu(1-\theta)}{2\theta}\|x_k-y_k\|^2
+\left(
-\eta+\frac{d}{2r}(L\eta^2+\mu\gamma^2)
\right)\|g_k\|^2,
\end{align}
which is \eqref{eq:strong-precollapse-general}.

We now derive the value of \(\gamma\) for which the same collapse pattern as in the
standard proof occurs. In \eqref{eq:strong-precollapse-general}, the inner-product terms are
\[
-\mu\gamma\langle g_k,y_k-x^\star\rangle
+\mu\gamma\frac{1-\theta}{\theta}\langle g_k,x_k-y_k\rangle.
\]
In order to rewrite the inner-product terms in \eqref{eq:strong-precollapse-general}
with the same coefficients as in \eqref{eq:strong-desired-collapse}, one must impose
\[
\mu\gamma=\theta,
\qquad
\mu\gamma\frac{1-\theta}{\theta}=1-\theta.
\]
These two identities are equivalent, and therefore
\[
\gamma=\frac{\theta}{\mu}.
\]
Since \(\gamma=\eta/\theta\), this is equivalent to
\[
\eta=\frac{\theta^2}{\mu},
\]
which proves \eqref{eq:strong-coupling-choice-collapse}.

Substituting \(\gamma=\theta/\mu\) and \(\eta=\theta^2/\mu\) into
\eqref{eq:strong-precollapse-general}, we obtain
\begin{align}
&\mathbb E\!\left[\Phi_{k+1}\middle|\mathcal F_k\right]-(1-\theta)\Phi_k
\notag\\
&\le
\theta\Bigl(
f(y_k)-f^\star-\langle g_k,y_k-x^\star\rangle+\frac{\mu}{2}\|y_k-x^\star\|^2
\Bigr)
\notag\\
&\qquad
+(1-\theta)\Bigl(
f(y_k)-f(x_k)+\langle g_k,x_k-y_k\rangle
\Bigr)
\notag\\
&\qquad
-\frac{\mu(1-\theta)}{2\theta}\|x_k-y_k\|^2
+\left(
-\eta+\frac{d}{2r}(L\eta^2+\mu\gamma^2)
\right)\|g_k\|^2.
\label{eq:strong-after-matching}
\end{align}
Since
\[
-\eta+\frac{d}{2r}(L\eta^2+\mu\gamma^2)
=
-\frac{\theta^2}{\mu}
+\frac{d}{2r}
\left(
L\frac{\theta^4}{\mu^2}
+
\frac{\theta^2}{\mu}
\right)
=
\frac{\theta^2}{\mu}
\left[
-1+\frac{d}{2r}\left(1+\frac{L\theta^2}{\mu}\right)
\right],
\]
this yields \eqref{eq:strong-collapse-final}.

Finally, by \(\mu\)-strong convexity,
\[
f(y_k)-f^\star-\langle g_k,y_k-x^\star\rangle+\frac{\mu}{2}\|y_k-x^\star\|^2\le 0,
\]
and by convexity,
\[
f(y_k)-f(x_k)+\langle g_k,x_k-y_k\rangle\le 0.
\]
Thus, in order for the remaining coefficient of \(\|g_k\|^2\) in
\eqref{eq:strong-collapse-final} to be nonpositive, it is necessary that
\[
\frac{d}{r}\left(1+\frac{L\theta^2}{\mu}\right)\le 2,
\]
which is \eqref{eq:strong-obstruction-condition-theta-collapse}. The equivalent form
\eqref{eq:strong-obstruction-condition-eta-collapse} follows from \(\eta=\theta^2/\mu\).

Finally, if \(r\le d/2\), then \(d/r\ge2\), and hence for every \(\theta\in(0,1]\),
\[
\frac{d}{r}\left(1+\frac{L\theta^2}{\mu}\right)
\ge
2\left(1+\frac{L\theta^2}{\mu}\right)
>
2.
\]
Therefore \eqref{eq:strong-obstruction-condition-theta-collapse} cannot hold for any
\(\theta\in(0,1]\), and hence for any \(\beta\in[0,1)\).
This completes the proof.

\end{proof}

\end{document}